\documentclass[11pt]{amsart}
\usepackage{amssymb,mathrsfs,graphicx,enumerate,color}
\usepackage{colortbl}
\definecolor{black}{rgb}{0.0, 0.0, 0.0}
\definecolor{red}{rgb}{1.0, 0.5, 0.5}

\title[   ]{Uniqueness and stability of entropy shocks to the isentropic Euler system in a class of inviscid limits from a large family of Navier-Stokes systems}

\author[Kang]{Moon-Jin Kang}
\address[Moon-Jin Kang]{
\newline Department of Mathematic \& Research Institute of Natural Sciences, \newline Sookmyung Women's University, Seoul 140-742, Korea}
\email{moonjinkang@sookmyung.ac.kr}

\author[Vasseur]{Alexis F. Vasseur}
\address[Alexis F. Vasseur]{\newline Department of Mathematics, \newline The University of Texas at Austin, Austin, TX 78712, USA}
\email{vasseur@math.utexas.edu}

\newtheorem{theorem}{Theorem}[section]
\newtheorem{lemma}{Lemma}[section]

\newtheorem{proposition}{Proposition}[section]
\newtheorem{remark}{Remark}[section]

\newcommand{\bbr}{\mathbb R}

\newcommand{\pt}{p(\tilde{v}_\eps)}
\newcommand{\deo}{\delta_0}

\numberwithin{figure}{section}
\newcommand{\beq}{\begin{equation}}
\newcommand{\eeq}{\end{equation}}
\newcommand{\bsp}{\begin{split}}
\newcommand{\esp}{\end{split}}

\def\eps{\varepsilon }

\newcommand\adots{\mathinner{\mkern2mu\raise1pt\hbox{.}
\mkern3mu\raise4pt\hbox{.}\mkern1mu\raise7pt\hbox{.}}}

\def\charf {\mbox{{\text 1}\kern-.30em {\text l}}}

\newcommand \deltat {\delta_2}

\newcommand{\vt}{\tilde{v}_\eps}

\begin{document}
\bibliographystyle{plain}

\date{\today}


\thanks{\textbf{Acknowledgment.}  The first author was partially supported by the NRF-2017R1C1B5076510 and the NRF-2019R1C1C1009355.
The second author was partially supported by the NSF grant: DMS 1614918. 
}

\begin{abstract}
We prove the uniqueness and stability of entropy shocks to the isentropic Euler systems among all 
 vanishing viscosity limits of solutions to associated  Navier-Stokes systems.
To take into account the vanishing viscosity limit, we show a contraction property for any large perturbations of viscous shocks to the Navier-Stokes system.
The contraction estimate does not depend on the strength of the viscosity. This provides a good control on  the inviscid limit process. We prove that, for any initial value,  there exist a vanishing viscosity limit to  solutions of the  Navier-Stokes system. The convergence holds in a weak topology. However,   this limit satisfies some stability estimates measured by the relative entropy with respect to an entropy shock.  In particular, our result provides the uniqueness of entropy shocks to the shallow water equation  in a class of inviscid limits of solutions to the viscous shallow water equations.
  \end{abstract}
\maketitle \centerline{\date}


\tableofcontents

\section{Introduction}
\setcounter{equation}{0}

We consider the vanishing viscosity limit ($\nu\to0$) of the one-dimensional barotropic Navier-Stokes system in the Lagrangian coordinates:
\begin{align}
\begin{aligned}\label{inveq}
\left\{ \begin{array}{ll}
        v^{\nu}_t - u^{\nu}_x =0,\\
       u^{\nu}_t+p(v^{\nu})_x = \nu\Big(\frac{\mu(v^{\nu})}{v^{\nu}} u^{\nu}_x\Big)_x, \end{array} \right.
\end{aligned}
\end{align}
where $v$ denotes the specific volume, $u$ is the fluid velocity, and $p(v)$ is the pressure law. We consider the case of  a polytropic perfect gas where the pressure is given by
\beq\label{pressure}
p(v)= v^{-\gamma},\quad \gamma> 1,
\eeq
with  $\gamma$ the adiabatic constant. Here, $\mu$ denotes the viscosity coefficient given by 
\beq\label{mu-def}
\mu(v) = bv^{-\alpha}. 
\eeq
Notice that if $\alpha>0$, $\mu(v)$ degenerates near the vacuum, i.e., near $v=+\infty$.
Very often, the viscosity coefficient is assumed to be constant, i.e., $\alpha=0$. However, in the  physical context  the viscosity of a gas depends on the temperature (see Chapman and Cowling \cite{CC}).  In the barotropic case,  the temperature depends directly on the density ($\rho=1/v$). 
The viscosity  is expected to degenerate near the vacuum as a power of the density, which is translated into $\mu(v) = bv^{-\alpha}$ in terms of $v$ with $\alpha>0$.

At least formally, as $\nu\to0$, the limit system of \eqref{inveq} is given by the isentropic Euler system:
\begin{align}
\begin{aligned} \label{Euler}
\left\{ \begin{array}{ll}
       v_t - u_x =0,\\
       u_t+p(v)_x =0.\end{array} \right.
\end{aligned}
\end{align}

The idea of approximating inviscid gases by viscous gases with vanishing viscosity is due to the seminal paper by Stokes \cite{Stokes}. The vanishing viscosity limit has been used later to 
construct  entropy solutions to the isentropic Euler system, see DiPerna \cite{Di,Di_CMP}, Hoff-Liu \cite{HL}, Goodman-Xin \cite{GX}, Yu \cite{Yu}. Recently,  Chen and Perepelitsa \cite{CP} proved the  convergence of solutions to the Navier-Stokes system with $\alpha=0$ towards an entropy solution of the isentropic Euler system with finite energy initial data. 
Some results exist for the  inviscid limit of the Navier-Stokes-Fourier system in very special cases, see for instance   Feireisl \cite{Fei} and \cite{VW}, and the references cited therein.

In this article, we prove the existence of vanishing viscosity limits of solutions to \eqref{inveq} in some weak sense, and obtain a stability estimate of those limits. We then  present the class of inviscid limits, in which the entropy shocks to \eqref{Euler} are unique.

Our result provides an answer, in the case of a shock, to the  conjecture: The compressible Euler system admits a unique entropy solution in the class of vanishing viscosity solutions to the associated  compressible Navier-Stokes system. 
As a breakthrough result related to this conjecture, Bianchini-Bressan \cite{BB} constructed a globally-in-time unique entropy solution to a strictly hyperbolic $n\times n$ system with small BV initial datum, which is obtained from vanishing ``artificial'' viscosity limit of the associate parabolic system. 
However, to the best of our knowledge, there is no result on uniqueness of discontinuous entropy solutions in the class of vanishing physical viscosity solutions to the Navier-Stokes systems. 

Previous uniqueness results for special discontinuous solutions (as solutions to the Riemann problem) required suitable regularity like locally $BV$ or strong trace properties (see  Chen-Frid-Li \cite{Chen1} and Vasseur et al. \cite{KVARMA, LV, Vasseur-2013}). Unfortunately, the global-in-time propagation of those regularities is unknown in general (except for the system with  $\gamma=3$  see  \cite{Vasseur_gamma3}). 
In the multi-D case,  De Lellis-Sz\'ekelyhidi\cite{DS10} and Chiodaroli et al. \cite{Ch,CDK,CFK,CK} showed  non-uniqueness of entropy solutions. They showed that entropy solutions to the isentropic Euler systems in more than one space dimension are not unique, by constructing infinitely many entropy solutions based on the convex integration method \cite{DS09, DS10}.\\

It is well known that the system \eqref{inveq} admits viscous shock waves connecting two end states 
$(v_-,u_-)$ and $(v_+,u_+)$, provided the two end states satisfy the Rankine-Hugoniot condition and the Lax entropy condition (see Matsumura and Wang \cite{MW}): 
\begin{align}
\begin{aligned}\label{end-con} 
&\exists~\sigma\quad\mbox{s.t. }~\left\{ \begin{array}{ll}
       -\sigma (v_+-v_-) - (u_+-u_-) =0,\\
       -\sigma (u_+-u_-) +p(v_+)-p(v_-)=0, \end{array} \right. \\
&\mbox{and either $v_->v_+$ and $u_->u_+$ or $v_-<v_+$ and $u_->u_+$ holds.}        
\end{aligned}
\end{align} 
In other words, for given constant states $(v_-,u_-)$ and $(v_+,u_+)$ satisfying \eqref{end-con}, there exists a viscous shock wave $(\tilde v^{\nu},\tilde u^{\nu})(x-\sigma t)$ as a solution of
\begin{align}
\begin{aligned}\label{shock_0} 
\left\{ \begin{array}{ll}
       -\sigma (\tilde v^{\nu})' - (\tilde u^{\nu})' =0,\\
       -\sigma (\tilde u^{\nu})'+p( \tilde v^{\nu})'= \nu\Big(\frac{\mu( \tilde v^{\nu})}{ \tilde v^{\nu}}  (\tilde u^{\nu})'\Big)'\\
       \lim_{\xi\to\pm\infty}(\tilde v^{\nu},\tilde u^{\nu})(\xi)=(v_{\pm}, u_\pm). \end{array} \right.
\end{aligned}
\end{align} 
Here, if $v_->v_+$, $(\tilde v,\tilde u)(x-\sigma t)$ is a 1-shock wave with velocity $\sigma=- \sqrt{-\frac{p(v_+)-p(v_-)}{v_+-v_-}}$, whereas if $v_-<v_+$, that is a 2-shock wave with $\sigma= \sqrt{-\frac{p(v_+)-p(v_-)}{v_+-v_-}}$.\\

Let $(\bar v, \bar u)$ be an associated entropy (inviscid) shock wave connecting the two end states $(v_-,u_-)$ and $(v_+,u_+)$ satisfying \eqref{end-con} as follows:
\beq\label{shock-0}
(\bar v, \bar u)(x -\sigma t)=\left\{ \begin{array}{ll}
         (v_-,u_-)  \quad\mbox{if $ x -\sigma t<0$},\\
        (v_+,u_+) \quad \mbox{if $x -\sigma t>0$}.\end{array} \right.
\eeq

As mentioned above, our goal is to show the uniqueness of the entropy shock to \eqref{Euler} in a suitable class, based on a generalization of our recent result \cite{Kang-V-NS17} on the contraction property of viscous shocks to \eqref{inveq}. 
More precisely, we prove the contraction of any large perturbations of viscous shocks to \eqref{inveq} in the case of $0<\alpha\le \gamma\le \alpha+1$ and $\gamma>1$, which improves the special case $\gamma=\alpha$ in \cite{Kang-V-NS17}. 
The contraction holds up to a shift, and is measured by a weighted relative entropy. 
Notice that since the relative entropy is locally quadratic, the contraction measured by the relative entropy can be regarded as $L^2$-type contraction.  
To prove the contraction, we employ the new approach introduced by the authors \cite{Kang-V-NS17}, which basically uses the relative entropy method. 
The relative entropy method has been extensively used in studying the contraction (or stability) of viscous (or inviscid) shock waves (see \cite{CKKV,CV,Kang19,Kang,Kang-V-NS17,KVARMA,Kang-V-1,KVW,Leger,Serre-Vasseur,SV_16,SV_16dcds,Vasseur-2013,VW}).

\subsection{Main results}

To handle the stability and uniqueness of the entropy shocks, we use the relative entropy associated to the entropy of\eqref{Euler} as follows:
For any functions $v_1,u_1,v_2,u_2$,
\beq\label{eta_def}
\eta((v_1,u_1)|(v_2,u_2)) :=\frac{|u_1-u_2|^2}{2} +Q(v_1|v_2),
\eeq
where $Q(v_1|v_2)$ is the relative functional associated with the strictly convex function 
\[
Q(v):=\frac{v^{-\gamma+1}}{\gamma-1},\quad v>0,
\]
that is, 
\[
Q(v_1|v_2) :=Q(v_1)-Q(v_2)-Q'(v_2)(v_1-v_2).
\]
However, the first components $v_1$ that we will consider are limit of Navier-Stokes equations, for which we obtain only uniform bounds in $L^1$. Therefore, the limit can be a measure in $t,x$. This is actually physical, and is related to the possible appearance of cavitation. For this reason, we need to extend the definition of relative entropy to measures defined on $\bbr^+\times \bbr$. We will restrict the definition in the case where we compare a measure $d v$ with a simple function $\bar v$ only taking two values  $v_-$ and  $v_+$. Let  $v_a$ denote the  Radon-Nikodym derivative  of $dv$ with respect to the Lebesgue measure and $dv_s$  its singular part, i.e., $dv=v_a \,dt\,dx+dv_s$.
The relative entropy  is then itself a measure defined as
\beq\label{dQ}
dQ(v|\bar v)(t,x) := Q\left(v_a|\bar v\right) dt dx +  |Q'(\overline V(t,x))| dv_s (t,x) ,
\eeq
where we need to define $\overline V$ everywhere. Denote $\Omega_{M}=\{t,x: \bar v(t,x)=\max(v_-,v_+)\}$, we set
\begin{eqnarray*}
\overline{V}(t,x)
=\left\{ \begin{array}{ll}
        \max(v_-,v_+) \qquad \mathrm{for  \ \ } (t,x)\in \overline{\Omega_M} \mbox{ (closure of $\Omega_M$)},\\
        \min(v_-,v_+) \qquad \mathrm{for  \ \ } (t,x)\in (\overline{\Omega_M})^c .\end{array} \right.
\end{eqnarray*}
Note that $|Q'( \max(v_-,v_+))| \le |Q'( \min(v_-,v_+))| $. Also, note that if $v\in L^\infty(\bbr^+;L^\infty(\bbr)+\mathcal{M}(\bbr))$, then   $dQ(v|\bar v)$ is defined in $L^\infty(\bbr^+; L^\infty(\bbr)+\mathcal{M}(\bbr))$, where $\mathcal{M}$ denotes the space of nonnegative Radon measures.

\vskip0.3cm
For the global-in-time existence of solutions to \eqref{inveq}, we introduce the function space:
\begin{align*}
\begin{aligned}
\mathcal{X}_T := \{ (v,u)\in\bbr^+\times\bbr~&|~ v-\underline v, ~u- \underline u \in L^\infty (0,T; H^1(\bbr)),\\
&\qquad\qquad u-\underline u \in L^2 (0,T; H^2(\bbr)),~  v^{-1}\in L^\infty((0,T)\times \bbr) \} ,
\end{aligned}
\end{align*}
where $\underline v$ and $\underline u$ are smooth monotone functions such that
\beq\label{sm-end}
\underline v(x) = v_\pm \quad\mbox{and}\quad \underline u(x) = u_\pm\quad\mbox{for } \pm x \ge 1.
\eeq

The first theorem is on stability and uniqueness of the entropy shocks to \eqref{Euler}:

\begin{theorem}\label{thm_inviscid} 
Let $\gamma>1$ and $\alpha, b>0$ be any constants satisfying $\alpha\le \gamma \le \alpha +1$. For each $\nu>0$, consider the system \eqref{inveq}-\eqref{mu-def}. For a given constant state $(v_-,u_-)\in\bbr^+\times\bbr$, 
there exists a constant  $\eps_0>0$ such that for any $\eps<\eps_0$ and any $(v_+,u_+)\in\bbr^+\times\bbr$ satisfying \eqref{end-con} with $|p(v_-)-p(v_+)|=\eps$, the following holds.\\
Let $(\tilde v^{\nu}, \tilde u^{\nu})$ be a viscous shock connecting the two end states $(v_-,u_-)$ and $(v_+,u_+)$ as a solution of \eqref{shock_0}.\\
Then for a given initial datum $(v^0,u^0)$ of \eqref{Euler} satisfying 
\beq\label{basic_ini}
\mathcal{E}_0:=\int_{-\infty}^{\infty} \eta\big((v^0,u^0)| (\bar v, \bar u)\big) dx <\infty,
\eeq
the following is true.\\
(i) (Well-prepared initial data) There exists a sequence of smooth functions $\{(v^{\nu}_0, u^{\nu}_0)\}_{\nu>0}$ such that 
\begin{align}
\begin{aligned}\label{ini_conv}
&\lim_{\nu\to0} v^{\nu}_0 = v^0,\quad \lim_{\nu\to0} u^{\nu}_0 = u^0\quad \mbox{a.e.},\quad v^{\nu}_0>0,\\
&\lim_{\nu\to 0} \int_\bbr  \left(\frac{1}{2}\left(u^{\nu}_0 +\nu\Big(p(v^{\nu}_0)^{\frac{\alpha}{\gamma}}\Big)_x -\tilde u^{\nu}-\nu\Big(p(\tilde v^{\nu})^{\frac{\alpha}{\gamma}}\Big)_x  \right)^2 +Q(v^{\nu}_0|\tilde v^{\nu})
\right)dx  = \mathcal{E}_0 .
\end{aligned}
\end{align}
(ii) For a given $T>0$, let $\{(v^{\nu}, u^{\nu})\}_{\nu>0}$ be a sequence of solutions in $\mathcal{X}_T$ to \eqref{inveq} with the initial datum $(v^{\nu}_0, u^{\nu}_0)$ as above. 
Then there exist limits $v_{\infty}$ and $u_{\infty}$ such that as $\nu\to0$ (up to a subsequence),
\beq\label{wconv}
v^{\nu}\rightharpoonup  v_{\infty},\quad u^{\nu} \rightharpoonup  u_{\infty}  \quad \mbox{in} ~\mathcal{M}_{\mathrm{loc}}((0,T)\times\bbr) ~\mbox{(space of locally bounded Radon measures)} ,
\eeq
where $v_\infty$ lies in $L^\infty(0,T,L^\infty(\bbr)+\mathcal{M}(\bbr))$.\\
  In addition, there exist shift $X_{\infty}\in \mbox{BV}((0,T))$ and constant $C>0$ such that  $d Q(v_\infty|\bar v)\in L^\infty(0,T;\mathcal{M}(\bbr))$, and for almost every $t\in (0,T)$,
\beq\label{uni-est}
\int_{\bbr } \frac{|u_\infty(t,x)-\bar u(x-X_{\infty}(t))|^2}{2}  dx + \left(\int_{x\in \bbr } d Q(v_\infty | \bar v (x-X_{\infty}(\cdot))) \right)(t) \  \le C \mathcal{E}_0.
\eeq

Moreover, the shift $X_\infty$ satisfies 
\beq\label{X-control}
|X_\infty(t) -\sigma t| \le \frac{C}{|v_--v_+|}\Big( \mathcal{E}_0 + (1+t)\sqrt{\mathcal{E}_0} \Big).
\eeq
Therefore, entropy shocks \eqref{shock-0} (with small amplitude) of the isentropic Euler system \eqref{Euler} are stable and unique in the class of weak inviscid limits of solutions to the Navier-Stokes system \eqref{inveq}.
\end{theorem}

\begin{remark}
1. By \eqref{uni-est}, the limits $v_{\infty}, u_{\infty}$ satisfy $v_{\infty}\in \bar v + L^{\infty}(0,T; L^\infty(\bbr)+\mathcal{M}(\bbr))$ and $u_{\infty} \in \bar u + L^{\infty}(0,T; L^2(\bbr))$, where $\mathcal{M}(\bbr)$ is the set of bounded Radon measures on $\bbr$. The control in measure of $v_\infty$ is due to the fact that $Q(v|\overline{v})\geq c_2|v-\overline{v}|$ for $v\geq 3v_-$ (see (\ref{rel_Q}) in Lemma \ref{lem-pro}). Especially,  $v_\infty$ may have some measure concentration at the limit. This corresponds physically to cavitation (the creation of bubbles in the fluid) and appearance of vacuum. It is interesting to see that this does not affect the contraction property (and the uniqueness of the shock at the limit).  \\
2. Theorem \ref{thm_inviscid} provides the stability and uniqueness of weak Euler shocks in the wide class of weak inviscid limits of solutions to the Navier-Stokes system. \\
Indeed, for the uniqueness, if $\mathcal{E}_0=0$, then \eqref{uni-est} and \eqref{X-control} imply that for a.e. $t\in (0,T)$,
\[
\int_{\bbr } \frac{|u_\infty(t,x)-\bar u(x-\sigma t)|^2}{2}  dx + \int_{\bbr }  Q(v_a(t,x) | \bar v (x-\sigma t)) dx = 0 ,
\]
where $dv_\infty=v_a \,dt\,dx+dv_s$, and the singular part $v_s$ vanishes.  Therefore, we have
\[
u_\infty(t,x) = \bar u(x-\sigma t),\qquad  v_\infty(t,x) =  \bar v (x-\sigma t),\quad \mbox{ a.e. } (t,x)\in [0,T]\times\bbr.
\]
3. In fact, the smallness of amplitude of shocks is not needed in the proof of Theorem \ref{thm_inviscid}. The constraint is due to Theorem \ref{thm_general}.\\
4.  It is worth emphasizing from the assumption on $\alpha$ and $\gamma$ that Theorem \ref{thm_inviscid} also holds in the case of the shallow water equations (i.e., $\gamma=2$) in a class of inviscid limits of solutions to the viscous shallow water equations (i.e., $\gamma=2$, $\alpha=1$). We refer to Gerbeau-Perthame \cite{GP} for a derivation of the viscous shallow water equations from the incompressible Navier-Stokes equations with free boundary.
\end{remark}

\begin{remark}\label{rem-sol}
For the global-in-time existence (and uniqueness) of any large solutions to \eqref{inveq} in $\mathcal{X}_T$, we refer to \cite{MV_sima} in the case of $\alpha<1/2$ and $\gamma>1$. More precisely, they proved that $\rho=1/v$ and $u$ satisfy
\[
\rho-\underline \rho, u-\underline u \in L^\infty (0,T; H^1(\bbr)),~u-\underline u \in L^2 (0,T; H^2(\bbr)),~  \rho^{-1}\in L^\infty((0,T)\times \bbr).
\]
This implies that there exists a solution in $\mathcal{X}_T$ to the system \eqref{inveq} with $\alpha<1/2$ and $\gamma>1$, since the system \eqref{inveq} is equivalent to the one in the Eulerian coordinates for such strong solutions. The result of \cite{MV_sima} was extended by Haspot \cite{Haspot} to the case of $\alpha\in (1/2,1]$. Recently, Constantin-Drivas-Nguyen-Pasqualotto \cite[Theorem 1.6]{CDNP} extended it to the case of $\alpha\ge 0$ and $\gamma\in [\alpha, \alpha+1]$ with $\gamma>1$, but they handled it on the periodic domain. 
Recently, the authors \cite{KV_exist19} extends the result \cite[Theorem 1.6]{CDNP}  to the case where smooth solutions connect possibly two different limits at the infinity on the whole space, which implies our solution space $\mathcal{X}_T$ .
\end{remark}

The starting point of the proof of Theorem \ref{thm_inviscid} is to derive a uniform-in-$\nu$ estimate for any large perturbations of viscous shocks to \eqref{inveq}. It is equivalent to  obtain the contraction property of any LARGE perturbations of viscous shocks to \eqref{inveq} with a fixed $\nu=1$:
\begin{align}
\begin{aligned}\label{main}
\left\{ \begin{array}{ll}
        v_t - u_x =0,\\
       u_t+p(v)_x = \Big(\frac{\mu(v)}{v} u_x\Big)_x. \end{array} \right.
\end{aligned}
\end{align}
As in \cite{Kang-V-NS17}, we first introduce the following relative functional $E(\cdot|\cdot)$ to measure the contraction:
\begin{align}
\begin{aligned}\label{psedo}
&\mbox{for any functions } v_1,u_1,v_2,u_2,\\ 
&E((v_1,u_1)|(v_2,u_2)) :=\frac{1}{2}\left(u_1 +\Big(p(v_1)^{\frac{\alpha}{\gamma}}\Big)_x -u_2 -\Big(p(v_2)^{\frac{\alpha}{\gamma}}\Big)_x  \right)^2 +Q(v_1|v_2),
\end{aligned}
\end{align}
where the constants $\gamma, \alpha$ are in \eqref{pressure} and \eqref{mu-def}.
The functional $E$ is associated to the BD entropy (see Bresch-Desjardins \cite{BD_03,BD_06,BDL}). Since $Q(v_1|v_2)$ is positive definite, \eqref{psedo} is also positive definite, that is, for any functions $(v_1,u_1)$ and $(v_2,u_2)$ we have $E((v_1,u_1)|(v_2,u_2))\ge 0$, and 
\[
\quad E((v_1,u_1)|(v_2,u_2))= 0~\mbox{a.e.} \quad\Leftrightarrow\quad (v_1,u_1)=(v_2,u_2)~\mbox{a.e.}
\]
The following result provides a contraction property  measured by the relative functional \eqref{psedo}. 

\begin{theorem}\label{thm_general}
For any $\gamma>1$ and $\alpha, b>0$ satisfying $\alpha\le \gamma \le \alpha +1$, consider the system \eqref{main} with \eqref{pressure}-\eqref{mu-def}.  For a given constant state $(v_-,u_-)\in\bbr^+\times\bbr$, 
there exist  constants  $\eps_0, \delta_0>0$ such that the following is true.\\
For any $\eps<\eps_0$, $\delta_0^{-1}\eps<\lambda<\delta_0$, and any $(v_+,u_+)\in\bbr^+\times\bbr$ satisfying \eqref{end-con} with $|p(v_-)-p(v_+)|=\eps$, there exists  a smooth monotone function $a:\bbr\to\bbr^+$ with $\lim_{x\to\pm\infty} a(x)=1+a_{\pm}$ for some  constants $a_-, a_+$ with $|a_+-a_-|=\lambda$ such that the following holds.\\
Let $\tilde U:=(\tilde v,\tilde u)$ be the viscous shock connecting $(v_-,u_-)$ and $(v_+,u_+)$ as a solution of \eqref{shock_0} with $\nu=1$.
For a given $T>0$, let $U:=(v,u)$ be a solution in $\mathcal{X}_T$ to \eqref{main} with a initial datum $U_0:=(v_0,u_0)$ satisfying $\int_{-\infty}^{\infty} E(U_0| \tilde U) dx<\infty$. Then there exists a shift $X\in W^{1,1}((0,T))$ such that 
\begin{align}
\begin{aligned}\label{cont_main}
&\int_{-\infty}^{\infty} a(x) E\big(U(t,x+X(t))| \tilde U(x)\big) dx \\
&\qquad  +\delta_0\frac{\eps}{\lambda} \int_{0}^T \int_{-\infty}^{\infty}  |\sigma a' (x)| Q\left(v(t,x+X(t))|\tilde v(x)\right) dx dt \\
&\qquad +\delta_0 \int_{0}^T \int_{-\infty}^{\infty}a(x) v^{\gamma-\alpha}(t,x+X(t))\big|\partial_x\big(p(v(t,x+X(t)))-p(\tilde v(x))\big)\big|^2dxdt  \\
&\quad\le  \int_{-\infty}^{\infty} a(x) E\big(U_0(x)| \tilde U(x)\big) dx,
\end{aligned}
\end{align}
and 
\begin{align}
\begin{aligned} \label{est-shift}
&|\dot X(t)|\le \frac{1}{\eps^2}(1 + f(t)),\\
&\mbox{for some positive function $f$ satisfying}\quad\|f\|_{L^1(0,T)} \le\frac{2\lambda}{\delta_0\eps}\int_{-\infty}^{\infty} E(U_0| \tilde U) dx.
\end{aligned}
\end{align}
\end{theorem}

\begin{remark}
1. Theorem \ref{thm_general} provides a contraction property for viscous shocks with suitably small amplitude parametrized by $\eps=|p(v_-)-p(v_+)|$. This smallness together with \eqref{end-con} implies $|v_--v_+|=\mathcal{O}(\eps)$ and $|u_--u_+|=\mathcal{O}(\eps)$. For such a fixed small shock, the contraction holds for any large solutions  to \eqref{main}, without any smallness condition imposed on $U_0$. This implies that the contraction still holds for any large solutions  to \eqref{inveq}, which provides a weak compactness to prove Theorem \ref{thm_inviscid} as the inviscid limit problem  ($\nu\to0$).\\
2. In \eqref{cont_main}, the dissipation terms will be used to show the convergence of $\{u^\nu\}_{\nu>0}$ in \eqref{wconv}.
\end{remark}

\begin{remark}
The contraction property is non-homogenous in $x$, as measured by the function $x\to a(x)$. This is consistant with the hyperbolic case (with $\nu=0$). In the  hyperbolic case, it was shown in \cite{Serre-Vasseur} that a homogenous contraction cannot hold  for the full Euler system. However, the contraction property is true if we consider a non-homogenous pseudo-distance \cite{Vasseur-2013} providing the so-called $a$-contraction \cite{KVARMA}. Our main result shows that the non-homogeneity of the pseudo-distance can be chosen of a similar size as the strength of the shock (as measured by the quantity $\lambda$).
\end{remark}

The rest of the paper is as follows. We explain main ideas of proofs of the mains results in Section \ref{sec:idea}. In Section \ref{sec:pre}, we provides a transformation of the system \eqref{main}, and an equivalent version of Theorem \ref{thm_general}, and useful inequalities. Section \ref{section_theo} is dedicated to the proof of Theorem \ref{thm_main}. Finally, Section \ref{sec:main} is dedicated to the proof of the main Theorem.

\section{Ideas of the proof.}\label{sec:idea}
\setcounter{equation}{0}

We describe in this section the methodology and main ideas of our results.
\vskip0.5cm
\noindent{\bf Uniform estimates with respect to the viscosity.}
The main results of this paper boil down to the proof of stability of the viscous shocks to the Navier-Stokes equations UNIFORMLY with respect to the strength of the viscosity. This can be obtained by considering only the case of the viscosity $\nu=1$, replacing the notion of stability by the notion of contraction, valid even for large perturbations (Theorem \ref{thm_general}). 

Indeed, if $(v^{\nu}, u^{\nu})$ is a solution of \eqref{inveq}, then 
$$
v(t,x)=v^{\nu}(\nu t,\nu x), \qquad u(t,x)=u^{\nu}(\nu t,\nu x)
$$
is a solution to \eqref{main}, i.e., the Navier-Stokes equations with $\nu=1$.  Note that, even if the initial perturbation $(v_0^{\nu}-\tilde{v}^\nu,u_0^{\nu}-\tilde{u}^\nu)$ is small, let say of order $\mathcal{E}$, then the  perturbation $(v_0-\tilde{v},u_0-\tilde{u})$ is big (of order $\mathcal{E}/\nu$):
$$
\int_\bbr E(U_0(\xi)|\tilde{U}(\xi))\,d\xi=\frac{1}{\nu}\int_\bbr E_{\nu}(U^\nu_0(\xi)|\tilde{U}^\nu(\xi))\,d\xi=\frac{\mathcal{E}}{\nu},
$$
where $E(\cdot|\cdot)$ is defined in \eqref{psedo}, and the rescaled $\nu$-dependent functional $E_\nu(\cdot|\cdot)$ is defined in \eqref{E_nu}.\\
However, a contraction independent of the size of the perturbation in the case $\nu=1$, as in Theorem \ref{thm_general}, provides, after rescaling,  a similar contraction for any values $\nu$: 
$$
\int_\bbr a(\xi/\nu) E_\nu(U^{\nu}(t,\xi+\nu X(t/\nu))|\tilde{U}^\nu(\xi))\,d\xi\leq \int_\bbr a(\xi/\nu)E_\nu(U_0^\nu|\tilde{U}^\nu)\,d\xi.
$$
This gives a uniform stability result with respect to $\nu$ provided that the weight function $a$ is uniformly bounded from below and from above,  that we have a control on the shift which is independent of the transformation
$$
X(t) \longrightarrow  \nu X\left(\frac{t}{\nu}\right),
$$
and that we have a uniform bound of $\int_\bbr E_\nu(U_0^\nu|\tilde{U}_0)\,d\xi$. The first two conditions are verified by Theorem \ref{thm_general}  thanks to (\ref{est-shift}), and considering $\lambda<1$, and the last one is verified thanks to  the well-prepared initial data \eqref{ini_conv}.

\vskip0.5cm
\noindent{\bf The contraction when $\nu=1$: Theorem \ref{thm_main}.}  This result is a generalization of the result in \cite{Kang-V-NS17} where only the case $\alpha =\gamma$ was considered. The extension introduces severe technical difficulties. A key to the extension is the local minimization explained below.
\vskip0.1cm
{\it Step one:  Considering a new velocity variable.}  We need to control the growth of the perturbation due to the hyperbolic terms (flux functionals). Thanks to the relative entropy method, the linear fluxes are easier to handle (the relative functional of linear quantity vanishes). Therefore, the main hyperbolic quantities to control are the pressure terms depending only on the specific volume $v$. At the core of the method, we are using a generalized Poincar\'e inequality Proposition \ref{prop:W}, first proved in \cite{Kang-V-NS17}. The Navier-Stokes system can be seen as a degenerate parabolic system. But the diffusion is in the other variable, the velocity variable $u$. Bresch and Desjardins (see \cite{BDL, BD_03,BD_06}) showed that compressible Navier-Stokes systems have a natural perturbed velocity quantity associated to the viscosity:
$$
h^\nu=u^\nu+\nu \left(p(v^\nu)^{\frac{\alpha}{\gamma}}\right)_x.
$$
Remarkably, the system in the variables $(v^\nu, h^\nu)$ exhibits a diffusion in the $v$ variable (the Smoluchowski equation), rather than in the velocity variable. 
For this reason, we are working with the natural relative entropy of this system, which corresponds to the usual relative entropy of the associated p-system in the $U^\nu_h=(v^\nu, h^\nu)$ variable:
$$
\eta(U^\nu_h|\tilde{U}^\nu_h)= E_\nu(U^\nu|\tilde{U}^\nu).
$$
For the rest of the proof of this theorem, we consider only $\nu=1$ and work only in the new set of variable $(v,h)$. To simplify the notation,  we denote $U=(v,h)$ from now on.

\vskip0.1cm
{\it Step 2: Evolution of the relative entropy.}
Computing the evolution of the relative entropy in Lemma \ref{lem-rel}, we get 
\begin{eqnarray*}
&& \frac{d}{dt}\int_{-\infty}^{\infty} a(\xi) \eta\big(U(t,\xi+X(t))| \tilde U_\eps (\xi)\big) d\xi\\
 &&\qquad=\dot X(t) Y(U(t,\cdot+X(t))) +\mathcal{J}^{bad}(U(t,\cdot+X(t)))- \mathcal{J}^{good}(U(t,\cdot+X(t))).
\end{eqnarray*}
The functional  $\mathcal{J}^{good}(U)$ is non-negative  (good term) and can be split into three terms (see \eqref{ybg-first}, \eqref{ggd}):
$$
\mathcal{J}^{good}(U)=\mathcal{J}^{good}_1(U)+\mathcal{G}_2(U)+\mathcal{D}(U),
$$
where only $\mathcal{J}^{good}_1(U)$ depends on $h$ (and actually does not depend on $v$). The term $\mathcal{D}(U)$ corresponds to the diffusive term (which depends on $v$ only, thanks to the transformation of the system).

\vskip0.1cm
{\it Step 3: Construction of the shift.}
The shift $X(t)$ produces  the term $\dot{X}(t) Y(U)$.  The key idea of the technique is to take advantage of this term when $Y(U(t,\cdot))$ is not two small, by compensating all the other terms via the choice of the velocity of the shift (see (\ref{X-def})). Specifically, we algebraically ensure  that the contraction holds as long as $|Y(U(t))|\geq\eps^2$. The rest of the analysis is to ensure that when $|Y(U(t))|\leq \eps^2$, the contraction still holds. 

The condition  $|Y(U(t))|\leq \eps^2$ ensures a smallness condition that we want to fully exploit. This is where the non-homogeneity of the semi-norm is crucial. In the  case where the function $a$ is constant,
$Y(U)$ is a linear functional in $U$. The smallness of $Y(U)$ gives only that a certain weighted mean value of $U$ is almost null. However, when $a$ is decreasing, $Y(U)$ becomes convex. The smallness $Y(U(t))\leq \eps^2$ implies, for this fixed time $t$ (See Lemma \ref{lemmeC2} with \eqref{d-weight} and \eqref{tail}):
\begin{equation}\label{small}
\int_{\bbr}\eps e^{-C\eps |\xi|}Q(v(t,\xi +X(t))|\tilde{v}_\eps(\xi))\,d\xi\leq C\left(\frac{\eps}{\lambda}\right)^2.
\end{equation}
This gives a control in $L^2$ for moderate values of $v$, and in  $L^1$ for big values of $v$,  in the layer region ($|\xi-X(t)|\lesssim 1/\eps$). 

The problem now looks, at first glance, as a typical problem of stability with a smallness condition.
There are, however, three major difficulties:  The bad term $\mathcal{J}^{bad}(U)$ has some terms depending on the variable $h$ for which we do not have diffusion, we have some smallness  in $v$, only for a very weak norm, and only localized in the layer region. More importantly, the smallness is measured with respect to the smallness of the shock. It basically says that, considering only the moderate values of $v$:
 the perturbation is not bigger than $\eps/\lambda$ (which is still very big with respect to the size of the shock  $\eps$). Actually, as we will see later, it is not possible to consider only the linearized problem: Third order terms appear in the  expansion using the smallness condition (the energy method involving the linearization would have only second order term in $\eps$).
 
In the argument, for the values of $t$ such that $|Y(U(t))|\leq\eps^2$,  we construct the shift as a solution to the ODE: $\dot X(t)=-Y(U(t,\cdot+X(t)))/\eps^4$. 
 From this point, we forget that $U=U(t,\xi)$ is a solution to \eqref{NS} and $X(t)$ is the shift. That is, we leave out $X(t)$ and the $t$-variable of $U$. Then we show that for any function $U$ satisfying 
 $|Y(U)|\leq \eps^2$, we have 
 \begin{equation}\label{but}
 -\frac{1}{\eps^4}Y^2(U)+\mathcal{J}^{bad}(U)-\mathcal{J}^{good}(U)\leq0.
 \end{equation}
This is the main Proposition \ref{prop:main} (actually, the proposition is slightly stronger to ensure the control of the shift). This implies clearly the contraction. 
From now on, we are focusing on the proof of this statement.

\vskip0.1cm
{\it Step 4: Maximization in $h$ for $v$ fixed.} We need to get rid of the dependence on the $h$ variable from the bad parts $\mathcal{J}^{bad}(U)$. The idea in \cite{Kang-V-NS17}  (for $\gamma=\alpha$) is to maximize the bad term with respect to $h$ for $v$ fixed:
$$
\mathcal{B}(v)=\sup_{h} \left(\mathcal{J}^{bad}(v,h)-\mathcal{J}^{good}_1(h)\right).
$$
We then had an inequality depending only on $v$ and $\partial_x v$ (through $\mathcal{D}(U)$) for which we can apply a generalized Poincar\'e inequality.
This does not work anymore when $\gamma\neq \alpha$. This is because $\mathcal{B}(v)$ involves powers of $p(v)$ which cannot be controlled by the good terms due to big values of $p(v)$. The new idea is to maximize in $h$ ONLY for the fixed values of $v$ such that $p(v)-p(\tilde{v}_\eps)\leq \delta_3$ for a constant $\delta_3$ to be determined (and depending on the Poincar\'e inequality). This  leads to the decomposition (\ref{prop:est}), (\ref{ggd}). The bad terms involving values $p(v)-p(\tilde{v}_\eps)> \delta_3$ can now be controlled using additional information from the unconditional estimate $|Y(U)|\leq \eps^2$ (see \eqref{n12}, \eqref{ns1}).

\vskip0.1cm
{\it Step 5: Expansion in $\eps$.}  Although we have no control on the supremum of $|p(v)-p(\tilde{v}_\eps)|$, we can control independently the contribution of the values $|p(v)-p(\tilde{v}_\eps)|\geq \delta_3$ in Proposition \ref{prop_out} (for the same $\delta_3$ related to the maximization process above. The coefficient $\delta_3$ can be chosen very small, but  independent of $\eps$ and of $\eps/\lambda$). The last step is to perform an expansion in the size of the shock $\eps$ every small, uniformly in $v$ (but for a fixed small value of $\delta$). As in \cite{Kang-V-NS17}, the expansion has to be done up to the third order. It leads to the exact same generic expression \eqref{Winst}. The generalized nonlinear Poincar\'e inequality, Proposition \ref{prop:W} concludes the proof.

\vskip0.5cm
\noindent{\bf The inviscid limit:  Theorem \ref{thm_inviscid}.}  

We  have now a stability result uniform with respect to the viscosity. It is natural to expect a stability result on the corresponding inviscid limit. The result, however, is not immediate. Several difficulties have to be overcome. First, due to the BD representation as above, the stability result for $\nu$ fixed is on the quantities:
$$
U^\nu_h=(v^\nu, h^\nu),\quad h^\nu=u^\nu +\nu\left(p(v^\nu)^{\frac{\alpha}{\gamma}}\right)_x.
$$ 
This is the reason we need a compatibility condition on the family of initial values $U_0^\nu$. This also leads to a very weak convergence (in measure in $(t,x)$ only). The next difficulty is that for small values of $v$,  the relative entropy control only the $L^1$ norm of $Q(v)=1/v^{\gamma-1}$. Therefore the pressure $p(v)=1/v^\gamma$ cannot be controlled at all. Therefore, we do not control the time derivative of $u$ in any distributional sense in $x$. Moreover, we have to study  carefully the effect of small and big values of $v$ together with big values of $|u|$  through truncations (see (\ref{v-trunc}) and (\ref{h-trunc})).  This is particularly important to pass to the limit on the shift in the contraction inequality (note that the shift converges only in $L^p_{\mathrm{loc}}(\bbr^+)$, for $1\leq p<\infty$). Finally, it has to be shown that the shift converges to $\sigma t$ when the perturbation converges to 0. This can be obtained, thanks to the convergence of $v$ in $C^0(\bbr^+, W^{-s,1}_{\mathrm{loc}}(\bbr))$. It is interesting to note that the continuity (in time) of $v$ is enough. We do not obtain any such control on $u$ (nor $h$).

\section{Preliminaries} \label{sec:pre}
\setcounter{equation}{0}

\subsection{Transformation of the system \eqref{main}} 
We here provides an equivalent version of Theorem \ref{thm_general} as in \cite{Kang-V-NS17}.
First of all, since the strength of the coefficient $b$ in $\mu(v)$ does not affect our analysis, as in \cite{Kang-V-NS17}, we set $b=\gamma$ (for simplification) and introduce a new effective velocity 
\[
h:=u + \Big(p(v)^{\frac{\alpha}{\gamma}}\Big)_x.
\]
The system \eqref{main} with $\mu(v)=\gamma v^{-\alpha}$ is then  transformed into
\begin{align}
\begin{aligned}\label{NS_1}
\left\{ \begin{array}{ll}
       v_t - h_x = -\Big( v^\beta p(v)_x\Big)_{x}\\
       h_t+p(v)_x =0, \end{array} \right.
\end{aligned}
\end{align}
where $\beta:=\gamma-\alpha$.
Notice that the above system has a parabolic regularization on the specific volume, contrary to the regularization on the velocity for the original system \eqref{main}. This is better for our analysis, since the hyperbolic part of the system is linear in $u$ (or $h$) but nonlinear in $v$ (via the pressure).
This effective velocity  was first introduced by Shelukhin \cite{Shel}  for $\alpha=0$, and in the general case (in Eulerian coordinates)  by Bresch-Desjardins \cite{BD_03,BD_06,BDL}, and Haspot \cite{H1,H2,H3}.  It was also used  in \cite{VY}.

As mentioned in Theorem \ref{thm_general}, we consider  shock waves with suitably small amplitude $\eps$. For that, let $(\tilde v_\eps,\tilde u_\eps)(x-\sigma_\eps t)$ denote a shock wave with amplitude $|p(v_-)-p(v_+)|=\eps$ as a solution of \eqref{shock_0} with $\mu(v)=\gamma v^{-\gamma}$. Then, setting $\tilde h_\eps:=\tilde u_\eps + \Big(p(\tilde v_\eps)^{\frac{\alpha}{\gamma}}\Big)_x$, the shock wave $(\tilde v_\eps,\tilde h_\eps)(x-\sigma_\eps t)$ satisfies
\begin{align}
\begin{aligned}\label{small_shock1} 
\left\{ \begin{array}{ll}
       -\sigma_\eps \tilde v_{\eps}' - \tilde h_{\eps}' =-\Big( {\tilde v_\eps}^\beta p(\tilde v_\eps)' \Big)'\\
       -\sigma_\eps \tilde h_{\eps}'+p( \tilde v_\eps)'=0\\
       \lim_{\xi\to\pm\infty}(\tilde v_\eps, \tilde h_\eps)(\xi)=(v_{\pm}, u_\pm). \end{array} \right.
\end{aligned}
\end{align} 

For simplification of our analysis, we rewrite \eqref{NS_1} into the following system, based on the change of variable $(t,x)\mapsto (t, \xi=x-\sigma_\eps t)$: 
\begin{align}
\begin{aligned}\label{NS}
\left\{ \begin{array}{ll}
       v_t -\sigma_\eps v_{\xi} - h_{\xi} = -\Big( v^\beta p(v)_{\xi}\Big)_{\xi}\\
       h_t-\sigma_\eps h_{\xi}+p(v)_{\xi} =0\\
       v|_{t=0}=v_0,\quad h|_{t=0}=u_0. \end{array} \right.
\end{aligned}
\end{align}

For the global-in-time existence of solutions to \eqref{NS}, we consider the function space:
\begin{align}
\begin{aligned}\label{sol-HT}
       \mathcal{H}_T := \{ (v,h)\in\bbr^+\times\bbr~ &|~ v-\underline v \in C (0,T; H^1(\bbr)), \\
       &\qquad\qquad ~h-\underline u \in C (0,T; L^2(\bbr)),~  v^{-1}\in L^\infty((0,T)\times \bbr) \} ,
\end{aligned}
\end{align}
where $\underline v$ and $\underline u$ are as in \eqref{sm-end}.

Theorem \ref{thm_general} is a direct consequence of the following theorem on the contraction of shocks to the system \eqref{NS}. \\

\begin{theorem}\label{thm_main}
Assume $\gamma>1$ and $\alpha>0$ satisfying $\alpha\le \gamma \le \alpha +1$.
For a given constant state $(v_-,u_-)\in\bbr^+\times\bbr$, there exist constants $\eps_0,\delta_0>0$ such that the following holds.\\
For any $\eps<\eps_0$, $\delta_0^{-1}\eps<\lambda<\delta_0$, and any $(v_+,u_+)\in\bbr^+\times\bbr$ satisfying \eqref{end-con} with $|p(v_-)-p(v_+)|=\eps$, there exists a  smooth monotone function $a:\bbr\to\bbr^+$ with $\lim_{x\to\pm\infty} a(x)=1+a_{\pm}$ for some  constants $a_-, a_+$ with $|a_--a_+|=\lambda$ such that the following holds.\\
Let $\tilde U_\eps:=(\tilde v_\eps,\tilde h_\eps)$ be a viscous shock connecting $(v_-,u_-)$ and $(v_+,u_+)$ as a solution of \eqref{small_shock1}.
For a given $T>0$, let $U:=(v,h)$ be a solution in $\mathcal{H}_T$ to \eqref{NS} with a initial datum $U_0:=(v_0,u_0)$ satisfying $\int_{-\infty}^{\infty} \eta(U_0| \tilde U_\eps) dx<\infty$, there exists a shift function $X\in W^{1,1}((0,T))$ such that 
\begin{align}
\begin{aligned}\label{cont_main2}
&\int_\bbr  a(\xi) \eta\big(U(t,\xi+X(t))| \tilde U_\eps (\xi)\big) d\xi \\
&\qquad  +\delta_0\frac{\eps}{\lambda} \int_{0}^{T}\int_{-\infty}^{\infty}  \sigma a' (\xi) Q\left(v(t,\xi+X(t))|\tilde v_\eps(\xi)\right) d\xi dt \\
&\qquad +\delta_0 \int_{0}^{T}\int_{-\infty}^{\infty}a(\xi) v^{\gamma-\alpha}(t,\xi+X(t))\Big|\partial_x\big(p(v(t,\xi+X(t)))-p(\tilde v_\eps(\xi))\big)\Big|^2d\xi dt  \\
&\quad\le   \int_{\bbr} a(\xi) \eta \big(U_0| \tilde U_\eps \big) d\xi,
\end{aligned}
\end{align}
and
\begin{align}
\begin{aligned} \label{est-shift1}
&|\dot X(t)|\le \frac{1}{\eps^2}\Big(f(t) + C\int_{-\infty}^{\infty} \eta(U_0|\tilde U_\eps ) d\xi +1  \Big) \quad \mbox{ for \textit{a.e.} }t\in[0,T] ,\\
&\mbox{for some positive function $f$ satisfying}\quad\|f\|_{L^1(0,T)} \le \frac{2\lambda}{\delta_0\eps}\int_{-\infty}^{\infty} \eta(U_0| \tilde U_\eps) d\xi.
\end{aligned}
\end{align}
\end{theorem}

\begin{remark}
1. In \cite{Kang-V-NS17}, the authors proved Theorem \ref{thm_general} in the case of $\alpha=\gamma$. Therefore, it suffices to prove the remaining cases where $0<\alpha< \gamma \le \alpha +1$. That is, $\beta=\gamma-\alpha\in(0,1]$.\\
2. Notice that it is enough to prove Theorem \ref{thm_main}  for  1-shocks. Indeed, the result for 2-shocks is obtained by the change of variables $x\to -x$, $u\to -u$, $\sigma_\eps\to -\sigma_\eps$. \\
Therefore, from now on, we only consider a 1-shock $(\tilde v_\eps,\tilde h_\eps)$, i.e., $v_->v_+$, $u_->u_+$, and
\beq\label{RH-con}
\sigma_\eps=- \sqrt{-\frac{p(v_+)-p(v_-)}{v_+-v_-}}.
\eeq
\end{remark}

\begin{remark} \label{rem-HT}
As mentioned in Remark \ref{rem-sol}, we consider the solution $(v,u)\in \mathcal{X}_T$ to \eqref{main}. Then, \eqref{NS} admits the solution $(v,h)$ in $\mathcal{H}_T$. Indeed,
since $v_t=u_x \in L^2(0,T; H^1(\bbr))$ by $\eqref{main}_1$, we have $v-\underline v \in C (0,T; H^1(\bbr))$. 
To show $h-\underline u \in C (0,T; L^2(\bbr))$, we first find that for $(v,u)\in \mathcal{X}_T$,
\[
h-\underline u=u-\underline u + \frac{\alpha}{\gamma} p(v)^{\frac{\alpha}{\gamma}-1} v_x \in L^\infty (0,T; L^2(\bbr)).
\]
Moreover, together with the fact that $v\in L^\infty((0,T)\times \bbr)$ by Sobolev embedding, we find that
\begin{align*}
\begin{aligned} 
& u_t = -p'(v) v_x + \frac{d}{dv}\Big(\frac{\mu(v)}{v}\Big)v_x u_x + \frac{\mu(v)}{v} u_{xx} \in L^2 (0,T; L^2(\bbr)),\\
&\Big(p(v)^{\frac{\alpha}{\gamma}-1} v_x \Big)_t = (\frac{\alpha}{\gamma}-1)p(v)^{\frac{\alpha}{\gamma}-2} v_t v_x + p(v)^{\frac{\alpha}{\gamma}-1} v_{xt}  \in L^2 (0,T; L^2(\bbr)),
\end{aligned}
\end{align*}
which implies $h_t \in L^2 (0,T; L^2(\bbr))$, and therefore $h-\underline u \in C (0,T; L^2(\bbr))$.
\end{remark}

\subsection{Global and local estimates on the relative quantities}
We here present useful inequalities on $Q$ and $p$ that are crucially used for the proofs of main results. First, the following lemma provides some global inequalities on the relative function $Q(\cdot|\cdot)$ corresponding to the convex function $Q(v)=\frac{v^{-\gamma+1}}{\gamma-1}$, $v>0$, $\gamma>1$. 
\begin{lemma}\label{lem-pro}
For given constants $\gamma>1$, and $v_->0$, there exist constants $c_1, c_2>0$ such that  the following inequalities hold.\\
1)  For any $w\in (0,v_-)$,
\begin{align}
\begin{aligned}\label{rel_Q}
& Q(v|w)\ge c_1 |v-w|^2,\quad \mbox{for all } 0<v\le 3v_-,\\
 & Q(v|w)\ge  c_2 |v-w|,\quad  \mbox{for all } v\ge 3v_-.
\end{aligned}
\end{align}
2) Moreover if $0<w\leq u\leq v$ or $0<v\leq u\leq w$ then 
\beq\label{Q-sim}
Q(v|w)\geq Q(u|w),
\eeq
and for any $\delta_*>0$ there exists a constant $C>0$ such that if, in addition, 
$v_->w>v_--\delta_*/2$ and $|w-u|>\delta_*$, we have
\beq\label{rel_Q1}
Q(v|w)-Q(u|w)\geq C|u-v|. 
\eeq
3)  For any $w\in (v_-/4, v_-)$,
\beq\label{pressure2}
|p(v)-p(w)| \le c_5 |v-w|,\quad \mbox{for all } v\ge v_-/2,
\eeq
\end{lemma}
\begin{proof}
We refer to \cite[Lemma 2.4, 2.5]{Kang-V-NS17}. 
\end{proof}

Next, we use \eqref{rel_Q} and \eqref{Q-sim} above to prove the following lemma, which is used for the proof of Theorem \ref{thm_inviscid}.

\begin{lemma}\label{lem_Q1}
For given constants $\gamma, M>1$, there exist constants $C>0$ and $k_0>1$  such that the following inequalities hold.\\
1) For any $k\ge 3M$ and $w\in (M^{-1}, M)$,
\beq\label{Q1}
\max\{(k^{-1}-v)_+, (v-k)_+ \} \le C Q(v|w),\quad\mbox{for any $v>0$}.
\eeq
2) For any $w_1, w_2 \in (M^{-1}, M)$,
\beq\label{Q2}
Q(v|w_1) \le CQ(v|w_2),\quad \mbox{for any $v\in (0,k_0^{-1})\cup(k_0,\infty)$}.
\eeq
\end{lemma}
\begin{proof}
$\bullet$ {\it proof of \eqref{Q1}} : i) If $k^{-1}\le v\le k$, then $\max\{(k^{-1}-v)_+, (v-k)_+ \}=0\le C Q(v|w)$.\\
ii) If $0<v<k^{-1}$, we have
\beq\label{Lhs1}
\max\{(k^{-1}-v)_+, (v-k)_+ \}=k^{-1}-v<k^{-1}.
\eeq
Since $v<k^{-1}\le M^{-1}/3$, we use \eqref{rel_Q} and \eqref{Q-sim} to have
\[
Q(v|w)\ge Q(M^{-1}/3|w)\ge c_1 |M^{-1}/3-w|^2.
\]
Moreover, since $w>M^{-1}$, we have
\[
Q(v|w)\ge c_14M^{-2}/9 \ge (c_14M^{-1}/3)k^{-1} ,
\]
which together with \eqref{Lhs1} implies the desired inequality.\\
iii) If $v>k$, we have
\[
\max\{(k^{-1}-v)_+, (v-k)_+ \}=v-k\le v-3M.
\]
Likewise, using \eqref{rel_Q}, we have
\[
Q(v|w)\ge c_2 |v-w|.
\]
Since $v>3M> w$, we have
\[
Q(v|w)\ge c_2 (v-3M),
\]
which completes the desired inequality.\\

$\bullet$ {\it proof of \eqref{Q2}} : We set $C:=2\max\big\{\frac{Q'(M)}{Q'(M^{-1})},1\big\}$. Since $Q'<0$, there exists $k_1>1$ such that for all $v>k_1$,
\[
-Q'(M)v\ge (1-C)Q(v) +\big(Q'(M)M -Q(M) \big) - C\big(Q'(M^{-1})M^{-1}-Q(M^{-1})\big).  
\]
Moreover, since $Q'$ is increasing and $\frac{d}{dv}\big(Q'(v)v-Q(v)\big)>0$, we have
\begin{align*}
\begin{aligned}
&\big(-CQ'(w_2)+Q'(w_1)\big) v \ge \big(-CQ'(M^{-1})+Q'(M)\big) v \ge -Q'(M) v\\
&\quad \ge (1-C)Q(v)+\big(Q'(M)M -Q(M) \big) - C\big(Q'(M^{-1})M^{-1}-Q(M^{-1})\big)\\
&\quad \ge (1-C)Q(v) + \big(Q'(w_1)w_1 -Q(w_1) \big) - C\big(Q'(w_2)w_2-Q(w_2)\big).  
\end{aligned}
\end{align*}
which together with the definition of $Q(\cdot|\cdot)$ yields that
\[
Q(v|w_1) \le CQ(v|w_2),\quad \mbox{for all }  v>k_1.
\]
On the other hand, since $Q(v)\to +\infty$ as $v\to 0+$, there exists $k_0>k_1$ such that for all $v<k_0^{-1}$,
\[
Q(v)\ge \big(CQ'(M^{-1})-Q'(M)\big) v  +\big(Q'(M)M -Q(M) \big) - C\big(Q'(M^{-1})M^{-1}-Q(M^{-1})\big).  
\]
Then we have
\begin{align*}
\begin{aligned}
&(C-1)Q(v)\ge Q(v)\\
&\quad \ge \big(CQ'(M^{-1})-Q'(M)\big) v  +\big(Q'(M)M -Q(M) \big) - C\big(Q'(M^{-1})M^{-1}-Q(M^{-1})\big)\\
&\quad \ge \big(CQ'(w_2)-Q'(w_1)\big) v  + \big(Q'(w_1)w_1 -Q(w_1) \big) - C\big(Q'(w_2)w_2-Q(w_2)\big),
\end{aligned}
\end{align*}
which yields that $Q(v|w_1) \le CQ(v|w_2) ~\mbox{for all }  v<k_0^{-1}.$
\end{proof}

We  present now  some local estimates on $p(v|w)$ and $Q(v|w)$.  
\begin{lemma}\label{lem:local}
For given constants $\gamma>1$ and $v_->0$ 
there exist positive constants $C$ and $\delta_*$ such that for any $0<\delta<\delta_*$, the following is true.\\
1) For any $(v, w)\in \bbr_+^2$ 
satisfying $|p(v)-p(w)|<\delta$ and  $|p(w)-p(v_-)|<\delta$, 
\begin{align}
\begin{aligned}\label{p-est1}
p(v|w)&\le \bigg(\frac{\gamma+1}{2\gamma} \frac{1}{p(w)} + C\delta \bigg) |p(v)-p(w)|^2.
\end{aligned}
\end{align}
2) For any $(v, w)\in \bbr_+^2$ such that  $|p(w)-p(v_-)|\leq \delta$,  and satisfying either $Q(v|w)<\delta$ or $|p(v)-p(w)|<\delta$,
\beq\label{pQ-equi0}
|p(v)-p(w)|^2 \le C Q(v|w).
\eeq
\end{lemma}
\begin{proof}
We refer to \cite[Lemma 2.6]{Kang-V-NS17}. 
\end{proof}


\vspace{1cm}

\section{Proof of Theorem \ref{thm_main}}\label{section_theo}
\setcounter{equation}{0}

Throughout this section, $C$ denotes a positive constant which may change from line to line, but which stays independent on $\eps$ (the shock strength) and $\lambda$ (the total variation of the function $a$). 
We will consider two smallness conditions, one on $\eps$, and the other on $\eps/\lambda$. In the argument, $\eps$ will be far smaller than $\eps/\lambda$ .

\subsection{Properties of small shock waves}
In this subsection, we  present  useful properties of the 1-shock waves $(\tilde v_\eps,\tilde h_\eps)$ with small amplitude $\eps$. In the sequel, without loss of generality, we consider the 1-shock wave $(\tilde v_\eps,\tilde h_\eps)$ satisfying $\tilde v_\eps(0)=\frac{v_-+v_+}{2}$. Notice that the estimates in the following lemma also hold for $\tilde h_\eps$ since we have  $\tilde h_\eps'=\frac{p'(\tilde v_\eps)}{\sigma_\eps} \tilde v_\eps'$ and $C^{-1}\le\frac{p'(\tilde v_\eps)}{\sigma_\eps}\le C$. But, since the below estimates for $\tilde v_\eps$ are enough in our analysis, we give the estimates only for $\tilde v_\eps$.

\begin{lemma}
We fix $v_->0$ and $h_-\in \bbr$. Then there exists $\eps_0>0$ such that for any $0<\eps<\eps_0$ the following is true. 
Let $\tilde v_{\eps}$ be the 1-shock wave with amplitude $|p(v_-) -p(v_+)|=\eps$ and such that $\tilde v_\eps(0)=\frac{v_-+v_+}{2}$. Then, there exist constants $C, C_1, C_2>0$ such that
\beq\label{tail}
-C^{-1}\eps^2 e^{-C_1 \eps |\xi|} \le \tilde v_\eps'(\xi) \le -C\eps^2 e^{-C_2 \eps |\xi|},\quad \forall\xi\in\bbr.
\eeq
Therefore, as a consequence, we have
\beq\label{lower-v}
\inf_{\left[-\frac{1}{\eps},\frac{1}{\eps}\right]}| v'_{\eps}|\ge C\eps^2.
\eeq
\end{lemma}
\begin{proof} 
Using $v_-/2<\tilde v_\eps<v_-$, the proof follows the same arguments as in \cite[Lemma 2.1]{Kang-V-NS17}. Therefore, we omit its details.
\end{proof}

\subsection{Relative entropy method}
Our analysis is  based on the relative entropy. The method  is purely nonlinear, and allows to handle rough and large perturbations. The relative entropy method was first introduced by Dafermos \cite{Dafermos1} and Diperna \cite{DiPerna} to prove the $L^2$ stability and uniqueness of Lipschitz solutions to the hyperbolic conservation laws endowed with a convex entropy.

To use the relative entropy method, we rewrite \eqref{NS} into the following general system of viscous conservation laws:
\beq\label{system-0}
\partial_t U +\partial_\xi A(U)= { -\partial_{\xi}\big(v^\beta \partial_{\xi}p(v)\big) \choose 0},
\eeq
where 
\[
U:={v \choose h},\quad A(U):={-\sigma_\eps v-h \choose -\sigma_\eps h+p(v)}.
\]
The system \eqref{system-0} has a convex entropy $\eta(U):=\frac{h^2}{2}+Q(v)$, where $Q(v)=\frac{v^{-\gamma+1}}{\gamma-1}$, i.e., $Q'(v)=-p(v)$.\\
Using the derivative of the entropy as 
\beq\label{nablae}
\nabla\eta(U)={-p(v)\choose h},
\eeq
the above system \eqref{system-0} can be rewritten as
\beq\label{system}
\partial_t U +\partial_\xi A(U)= \partial_\xi\Big(M(U) \partial_\xi\nabla\eta(U) \Big),
\eeq
where $M(U)={v^\beta \quad 0 \choose 0\quad 0}$, and \eqref{small_shock1} can be rewritten as 
\beq\label{re_shock}
\partial_\xi A(\tilde U_\eps)= \partial_\xi\Big(M(\tilde U_\eps) \partial_\xi\nabla\eta(\tilde U_\eps) \Big).
\eeq
Consider the relative entropy function defined by
\[
\eta(U|V)=\eta(U)-\eta(V) -\nabla\eta(V) (U-V),
\]
and the relative flux defined by
\[
A(U|V)=A(U)-A(V) -\nabla A(V) (U-V).
\] 
Let $G(\cdot;\cdot)$ be the flux of the relative entropy defined by
\[
G(U;V) = G(U)-G(V) -\nabla \eta(V) (A(U)-A(V)),
\]
where $G$ is the entropy flux of $\eta$, i.e., $\partial_{i}  G (U) = \sum_{k=1}^{2}\partial_{k} \eta(U) \partial_{i}  A_{k} (U),\quad 1\le i\le 2$.\\
Then, for our system \eqref{system-0}, we have
\begin{align}
\begin{aligned}\label{relative_e}
&\eta(U|\tilde U_\eps)=\frac{|h-\tilde h_\eps |^2}{2} + Q(v|\tilde v_\eps),\\
& A(U|\tilde U_\eps)={0 \choose p(v|\tilde v_\eps)},\\
&G(U;\tilde U_\eps)=(p(v)-p(\tilde v_\eps)) (h-\tilde h_\eps)-\sigma_\eps \eta(U|\tilde U_\eps),
\end{aligned}
\end{align}
where the relative pressure is defined as
\begin{equation}\label{pressure-relative}
p(v|w)=p(v)-p(w)-p'(w)(v-w).
\end{equation}

We consider a weighted relative entropy between the solution $U$ of \eqref{system} and the viscous shock $\tilde U_\eps:={\tilde v_\eps \choose \tilde h_\eps}$ in \eqref{small_shock1} up to a shift $X(t)$ :
\[
a(\xi)\eta\big(U(t,\xi+X(t))|\tilde U_\eps(\xi)\big).
\]
where $a$ is a smooth weight function.\\

In Lemma \ref{lem-rel}, we will derive a quadratic structure on $\frac{d}{dt}\int_{\bbr} a(\xi)\eta\big(U(t,\xi+X(t))|\tilde U_\eps(\xi)\big) d\xi$.\\
For that, we introduce a simple notation: for any function $f : \bbr^+\times \bbr\to \bbr$ and the shift $X(t)$, 
\[
f^{\pm X}(t, \xi):=f(t,\xi\pm X(t)).
\]
We also introduce the function space:
\[
\mathcal{H}:=\{(v,h)\in\bbr^+\times\bbr~|~  v^{-1}, v \in L^{\infty}(\bbr),~ h-\tilde h_\eps\in L^2( \bbr), ~ \partial_\xi \big(p(v)-p(\tilde v_\eps) \big)\in L^2( \bbr) \},
\]
on which the functionals $Y, \mathcal{J}^{bad},\mathcal{J}^{good}$ in \eqref{ybg-first} are well-defined for all $t\in (0,T)$. 

\begin{remark} \label{rem-H}
As mentioned in Remark \ref{rem-HT}, we consider the solution $(v,h)\in \mathcal{H}_T$ to \eqref{NS}.
Then, using the fact that $v_\xi \in C (0,T; L^2(\bbr)),~\tilde v_\eps ' \in L^2( \bbr)$, and $ v^{-1}, v \in C (0,T; L^{\infty}(\bbr))$, we find
\[
 \partial_\xi  \big(p(v)-p(\tilde v_\eps) \big) \in C (0,T; L^2(\bbr)),
\]
which implies $(v,h)(t)\in \mathcal{H}$ for all $t\in[0,T]$. 
\end{remark}

\begin{lemma}\label{lem-rel}
Let $a:\bbr\to\bbr^+$ be any positive smooth bounded function whose derivative is bounded and integrable. Let $\tilde U_\eps:={\tilde v_\eps \choose \tilde h_\eps}$ be the viscous shock in \eqref{small_shock1}. For any solution $U\in  \mathcal{H}_T$ to \eqref{system}, and any absolutely continuous shift $X:[0,T]\to\bbr$, the following holds. 
\begin{align}
\begin{aligned}\label{ineq-0}
\frac{d}{dt}\int_{\bbr} a(\xi)\eta(U^X(t,\xi)|\tilde U_\eps(\xi)) d\xi =\dot X(t) Y(U^X) +\mathcal{J}^{bad}(U^X) - \mathcal{J}^{good}(U^X),
\end{aligned}
\end{align}
where
\begin{align}
\begin{aligned}\label{ybg-first}
&Y(U):= -\int_\bbr a'\eta(U|\tilde U_\eps) d\xi +\int_\bbr a\partial_\xi\nabla\eta(\tilde U_\eps) (U-\tilde U_\eps) d\xi,\\
&\mathcal{J}^{bad}(U):= \int_\bbr a' \big(p(v)-p(\tilde v_\eps)\big) (h-\tilde h_\eps)d\xi +  \sigma_\eps\int_\bbr a \partial_\xi \tilde v_\eps p(v|\tilde v_\eps) d\xi\\
& \qquad\quad -\int_\bbr a' v^\beta \big(p(v)-p(\tilde v_\eps)\big)\partial_\xi \big(p(v)-p(\tilde v_\eps)\big)  d\xi-\int_\bbr a' \big(p(v)-p(\tilde v_\eps)\big) (v^\beta - \tilde v_\eps^\beta) \partial_{\xi} p(\tilde v_\eps)d\xi \\
&\qquad\quad -\int_\bbr a \partial_\xi \big(p(v)-p(\tilde v_\eps)\big) (v^\beta - \tilde v_\eps^\beta) \partial_{\xi} p(\tilde v_\eps)  d\xi,\\
&\mathcal{J}^{good}(U):= \frac{\sigma_\eps}{2}\int_\bbr a'\left| h-\tilde h_\eps\right|^2  d\xi +\sigma_\eps  \int_\bbr  a' Q(v|\tilde v_\eps) d\xi + \int_\bbr a v^\beta \left|\partial_\xi \big(p(v)-p(\tilde v_\eps)\big)\right|^2 d\xi.
\end{aligned}
\end{align}
\end{lemma}

\begin{remark}
In what follows, we will define the weight function $a$ such that $\sigma_\eps a' >0$. Therefore, $-\mathcal{J}^{good}$ consists of three good terms, while $\mathcal{J}^{bad}$ consists of bad terms. 
\end{remark}

\begin{proof}
To derive the desired structure, we  use here a change of variable $\xi\mapsto \xi-X(t)$ as
\beq\label{move-X}
\int_{\bbr} a(\xi)\eta(U^X(t,\xi)|\tilde U_\eps(\xi)) d\xi=\int_{\bbr} a^{-X}(\xi)\eta(U(t,\xi)|\tilde U_\eps^{-X}(\xi)) d\xi.
\eeq
Then, using the same computation in \cite[Lemma 2.3]{Kang-V-NS17} (see also \cite[Lemma 4]{Vasseur_Book}), we have
\begin{align*}
\begin{aligned}
&\frac{d}{dt}\int_{\bbr} a^{-X}(\xi)\eta(U(t,\xi)|\tilde U_\eps^{-X}(\xi)) d\xi\\
&=-\dot X \int_{\bbr} \!a'^{-X} \eta(U|\tilde U_\eps^{-X} ) d\xi +\int_\bbr \!\!a^{-X}\bigg[\Big(\nabla\eta(U)-\nabla\eta(\tilde U_\eps^{-X})\Big)\!\Big(\!\!\!-\partial_\xi A(U)+ \partial_\xi\Big(M(U)\partial_\xi\nabla\eta(U) \Big) \Big)\\
&\qquad -\nabla^2\eta(\tilde U_\eps^{-X}) (U-\tilde U_\eps^{-X}) \Big(-\dot X \partial_\xi\tilde U_\eps^{-X} -\partial_\xi A(\tilde U_\eps^{-X})+ \partial_\xi\Big(M(\tilde U_\eps^{-X})\partial_\xi\nabla\eta(\tilde U_\eps^{-X}) \Big)\Big)  \bigg] d\xi\\
&\quad =\dot X \Big( -\int_\bbr a'^{-X}\eta(U|\tilde U_\eps^{-X}) d\xi +\int_\bbr a^{-X}\partial_\xi\nabla\eta(\tilde U_\eps^{-X}) (U-\tilde U_\eps^{-X})  \Big) +I_1+I_2+I_3+I_4,
\end{aligned}
\end{align*}
where 
\begin{align*}
\begin{aligned}
&I_1:=-\int_\bbr a^{-X} \partial_\xi G(U;\tilde U_\eps^{-X}) d\xi,\\
&I_2:=- \int_\bbr a^{-X} \partial_\xi \nabla\eta(\tilde U_\eps^{-X}) A(U|\tilde U_\eps^{-X}) d\xi,\\
&I_3:=\int_\bbr a^{-X} \Big( \nabla\eta(U)-\nabla\eta(\tilde U_\eps^{-X})\Big) \partial_\xi \Big(M(U) \partial_\xi \big(\nabla\eta(U)-\nabla\eta(\tilde U_\eps^{-X})\big) \Big)  d\xi, \\
&I_4:=\int_\bbr a^{-X} \Big( \nabla\eta(U)-\nabla\eta(\tilde U_\eps^{-X})\Big) \partial_\xi \Big(\big(M(U)-M(\tilde U_\eps^{-X})\big) \partial_\xi \nabla\eta(\tilde U_\eps^{-X}) \Big)  d\xi \\
&I_5:=\int_\bbr a^{-X}(\nabla\eta)(U|\tilde U_\eps^{-X})\partial_\xi \Big(M(\tilde U_\eps^{-X}) \partial_\xi \nabla\eta(\tilde U_\eps^{-X}) \Big)  d\xi.
\end{aligned}
\end{align*}
Using \eqref{relative_e} and \eqref{nablae}, we have
\begin{align*}
\begin{aligned}
I_1&=\int_\bbr a'^{-X} G(U;\tilde U_\eps^{-X}) d\xi = \int_\bbr a'^{-X} \Big(\big(p(v)-p(\tilde v_\eps^{-X})\big) \big(h-\tilde h_\eps^{-X}\big)  -\sigma_\eps \eta(U|\tilde U_\eps^{-X})\Big)  d\xi\\
&= \int_\bbr a'^{-X} \big(p(v)-p(\tilde v_\eps^{-X})\big) \big(h-\tilde h_\eps^{-X}\big) d\xi -\frac{\sigma_\eps}{2}\int_\bbr a'^{-X}\left| h-\tilde h_\eps\right|^2 d\xi  -\sigma_\eps  \int_\bbr  a'^{-X} Q(v|\tilde v_\eps) d\xi ,\\
I_2&=-\int_\bbr a^{-X} \partial_\xi \tilde h_\eps^{-X} p(v|\tilde v_\eps^{-X}) d\xi.
\end{aligned}
\end{align*}
By integration by parts, we have
\begin{align*}
\begin{aligned}
I_3&=\int_\bbr a^{-X} \big(p(v)-p(\tilde v_\eps^{-X})\big)\partial_{\xi}\Big(v^\beta \partial_{\xi}\big(p(v)-p(\tilde v_\eps^{-X})\big) \Big) d\xi \\
&=-\int_\bbr a^{-X}v^\beta |\partial_\xi \big(p(v)-p(\tilde v_\eps^{-X})\big)|^2 d\xi -\int_\bbr a'^{-X} v^\beta \big(p(v)-p(\tilde v_\eps^{-X})\big)\partial_\xi \big(p(v)-p(\tilde v_\eps^{-X})\big)  d\xi,\\
I_4&= \int_\bbr a^{-X} \big(p(v)-p(\tilde v_\eps^{-X})\big)\partial_{\xi}\Big((v^\beta - \tilde v_\eps^\beta) \partial_{\xi} p(\tilde v_\eps^{-X}) \Big) d\xi \\
&=-\int_\bbr a'^{-X} \big(p(v)-p(\tilde v_\eps^{-X})\big) (v^\beta - \tilde v_\eps^\beta) \partial_{\xi} p(\tilde v_\eps^{-X})  d\xi \\
&\qquad -\int_\bbr a^{-X} \partial_\xi \big(p(v)-p(\tilde v_\eps^{-X})\big) (v^\beta - \tilde v_\eps^\beta) \partial_{\xi} p(\tilde v_\eps^{-X})  d\xi.
\end{aligned}
\end{align*}
Since it follows from \eqref{re_shock} and \eqref{nablae} that
\begin{align*}
\begin{aligned}
I_5=\int_\bbr a^{-X}(\nabla\eta)(U|\tilde U_\eps^{-X})\partial_\xi A(\tilde U_\eps^{-X})  d\xi =\int_\bbr a^{-X} p(v|\tilde v_\eps^{-X}) \Big(\partial_\xi \tilde h_\eps^{-X} + \sigma_\eps \partial_\xi \tilde v_\eps^{-X} \Big) d\xi,
\end{aligned}
\end{align*}
we have some cancellation
\[
I_2+I_5=\sigma_\eps\int_\bbr a^{-X} \partial_\xi \tilde v_\eps^{-X} p(v|\tilde v_\eps^{-X}) d\xi.
\]
Therefore, we have
\begin{align*}
\begin{aligned}
&\frac{d}{dt}\int_{\bbr} a^{-X}\eta(U|\tilde U_\eps^{-X}) d\xi\\
&\quad =\dot X \Big( -\int_\bbr a'^{-X}\eta(U|\tilde U_\eps^{-X}) d\xi +\int_\bbr a^{-X}\partial_\xi\nabla\eta(\tilde U_\eps^{-X}) (U-\tilde U_\eps^{-X}) d\xi  \Big)\\
&\quad + \int_\bbr a'^{-X} \big(p(v)-p(\tilde v_\eps^{-X})\big) (h-\tilde h_\eps^{-X}) d\xi -\frac{\sigma_\eps}{2}\int_\bbr a'^{-X}\left| h-\tilde h_\eps\right|^2 d\xi  -\sigma_\eps  \int_\bbr  a'^{-X} Q(v|\tilde v_\eps) d\xi  \\
&\quad + \sigma_\eps\int_\bbr a^{-X} \partial_\xi \tilde v_\eps^{-X} p(v|\tilde v_\eps^{-X}) d\xi -\int_\bbr a'^{-X} v^\beta \big(p(v)-p(\tilde v_\eps^{-X})\big)\partial_\xi \big(p(v)-p(\tilde v_\eps^{-X})\big)  d\xi\\
&\quad -\int_\bbr a'^{-X} \big(p(v)-p(\tilde v_\eps^{-X})\big) \left(v^\beta - (\tilde v_\eps^\beta)^{-X}\right) \partial_{\xi} p(\tilde v_\eps^{-X})d\xi \\
&\quad -\int_\bbr a^{-X} \partial_\xi \big(p(v)-p(\tilde v_\eps^{-X})\big) \left(v^\beta - (\tilde v_\eps^\beta)^{-X}\right) \partial_{\xi} p(\tilde v_\eps^{-X})  d\xi -\int_\bbr a^{-X} v^\beta |\partial_\xi (p(v)-p(\tilde v_\eps^{-X}))|^2 d\xi.
\end{aligned}
\end{align*}
Again, we use a change of variable $\xi\mapsto \xi+X(t)$ to have
\begin{align*}
\begin{aligned}
&\frac{d}{dt}\int_{\bbr} a\eta(U^X|\tilde U_\eps) d\xi\\
&\quad =\dot X \Big( -\int_\bbr a'\eta(U^X|\tilde U_\eps) d\xi +\int_\bbr a\partial_\xi\nabla\eta(\tilde U_\eps) (U^X-\tilde U_\eps) d\xi  \Big)\\
&\qquad + \int_\bbr a' \big(p(v^X)-p(\tilde v_\eps)\big) (h^X-\tilde h_\eps)d\xi  -\frac{\sigma_\eps}{2}\int_\bbr a' \left| h^X-\tilde h_\eps\right|^2 d\xi -\sigma_\eps  \int_\bbr  a' Q(v^X|\tilde v_\eps) d\xi \\
&\qquad + \sigma_\eps\int_\bbr a \partial_\xi \tilde v_\eps p(v^X|\tilde v_\eps) d\xi -\int_\bbr a' (v^\beta)^X \big(p(v^X)-p(\tilde v_\eps)\big)\partial_\xi \big(p(v^X)-p(\tilde v_\eps)\big)  d\xi\\
&\qquad -\int_\bbr a' \big(p(v^X)-p(\tilde v_\eps)\big) \left((v^\beta)^{X} - \tilde v_\eps^\beta\right) \partial_{\xi} p(\tilde v_\eps)d\xi \\
&\qquad -\int_\bbr a \partial_\xi \big(p(v^X)-p(\tilde v_\eps)\big) \left((v^\beta)^X - \tilde v_\eps^\beta\right) \partial_{\xi} p(\tilde v_\eps)  d\xi -\int_\bbr a (v^\beta)^X |\partial_\xi (p(v^X)-p(\tilde v_\eps))|^2 d\xi,
\end{aligned}
\end{align*}
which provides the desired representation.
\end{proof}

\subsection{Construction of the weight function}
We define the weight function $a$ by
\beq\label{weight-a}
a(\xi)=1-\lambda \frac{p(\tilde v_\eps(\xi))-p(v_-)}{[p]},
\eeq
where $[p]:=p(v_+)-p(v_-)$.
We briefly present some useful properties on the weight $a$.\\
First of all, the weight function $a$ is positive and decreasing, and satisfies $1-\lambda\le a\le 1$.\\
Since $[p]=\eps$, $p'(v_-/2)\le p'(\tilde v_\eps)\le p'(v_-)$ and
\beq\label{der-a}
a'=-\lambda \frac{\partial_\xi p(\tilde v_\eps)}{[p]},
\eeq
we have 
\beq\label{d-weight}
|a'|\sim \frac{\lambda}{\eps}|\tilde v_\eps'|.
\eeq

\subsection{Maximization in terms of $h-\tilde h_\eps$}\label{sec:mini}

In order to estimate the right-hand side of \eqref{ineq-1}, we will use Proposition \ref{prop:main3}, i.e., a sharp estimate with respect to $v-\tilde v_\eps$ when $v-\tilde v_\eps\ll 1$, for which we need to rewrite $\mathcal{J}^{bad}$ on the right-hand side of \eqref{ineq-0} only in terms of $v$ near $\tilde v_\eps$, by separating $h-\tilde h_\eps$ from the first term of $\mathcal{J}^{bad}$. Therefore, we will rewrite $\mathcal{J}^{bad}$ into the maximized representation in terms of $h-\tilde h_\eps$ in the following lemma. However, we will keep all terms of $\mathcal{J}^{bad}$ in a region $\{p(v)-\pt >\delta\}$ for small values of $v$, since we use the estimate \eqref{ns1} to control the first term of $\mathcal{J}^{bad}$ in that region.  

\begin{lemma}\label{lem-max}
Let $a:\bbr\to\bbr^+$ be as in \eqref{weight-a}, and $\tilde U_\eps={\tilde v_\eps \choose \tilde h_\eps}$ be the viscous shock in \eqref{small_shock1}. Let $\delta$ be any positive constant.
Then, for any $U\in \mathcal{H}$, 
\begin{align}
\begin{aligned}\label{ineq-1}
\mathcal{J}^{bad} (U) -\mathcal{J}^{good} (U)= \mathcal{B}_\delta(U)- \mathcal{G}_\delta(U),
\end{aligned}
\end{align}
where
\begin{align}
\begin{aligned}\label{badgood}
&\mathcal{B}_\delta(U):=  \sigma_\eps\int_\bbr a \partial_\xi \tilde v_\eps p(v|\tilde v_\eps) d\xi+ \frac{1}{2\sigma_\eps} \int_\bbr a' |p(v)-p(\tilde v_\eps)|^2 {\mathbf 1}_{\{p(v)-\pt \leq\delta\}} d\xi \\
&\qquad \quad + \int_\bbr a' \big(p(v)-p(\tilde v_\eps)\big) (h-\tilde h_\eps) {\mathbf 1}_{\{p(v)-\pt >\delta\}}d\xi \\
& \qquad\quad -\int_\bbr a' v^\beta \big(p(v)-p(\tilde v_\eps)\big)\partial_\xi \big(p(v)-p(\tilde v_\eps)\big)  d\xi-\int_\bbr a' \big(p(v)-p(\tilde v_\eps)\big) (v^\beta - \tilde v_\eps^\beta) \partial_{\xi} p(\tilde v_\eps)d\xi \\
&\qquad\quad -\int_\bbr a \partial_\xi \big(p(v)-p(\tilde v_\eps)\big) (v^\beta - \tilde v_\eps^\beta) \partial_{\xi} p(\tilde v_\eps)  d\xi,\\
&\mathcal{G}_\delta(U):=\frac{\sigma_\eps}{2}\int_\bbr a'\left| h-\tilde h_\eps -\frac{p(v)-p(\tilde v_\eps)}{\sigma_\eps}\right|^2 {\mathbf 1}_{\{p(v)-\pt \leq\delta\}} d\xi  + \frac{\sigma_\eps}{2}\int_\bbr a'\left| h-\tilde h_\eps\right|^2 {\mathbf 1}_{\{p(v)-\pt >\delta\}} d\xi  \\
&\qquad\quad+\sigma_\eps  \int_\bbr  a' Q(v|\tilde v_\eps) d\xi +\int_\bbr a v^\beta |\partial_\xi \big(p(v)-p(\tilde v_\eps)\big)|^2 d\xi.
\end{aligned}
\end{align}

\begin{remark}\label{rem:0}
Since $\sigma_\eps a' >0$ and $a>0$, $-\mathcal{G}_\delta$ consists of four good terms. 
\end{remark}

\end{lemma}
\begin{proof}
For any fixed $\delta>0$, we first rewrite $\mathcal{J}^{bad}$ and $-\mathcal{J}^{good}$ into
\begin{align*}
\begin{aligned}
&\mathcal{J}^{bad}(U):=\underbrace{ \int_\bbr a' \big(p(v)-p(\tilde v_\eps)\big) (h-\tilde h_\eps) {\mathbf 1}_{\{p(v)-\pt \leq\delta\}} d\xi }_{=:J_1}\\
& \qquad\quad +\int_\bbr a' \big(p(v)-p(\tilde v_\eps)\big) (h-\tilde h_\eps) {\mathbf 1}_{\{p(v)-\pt >\delta\}} d\xi 
+  \sigma_\eps\int_\bbr a \partial_\xi \tilde v_\eps p(v|\tilde v_\eps) d\xi\\
& \qquad\quad -\int_\bbr a' v^\beta \big(p(v)-p(\tilde v_\eps)\big)\partial_\xi \big(p(v)-p(\tilde v_\eps)\big)  d\xi-\int_\bbr a' \big(p(v)-p(\tilde v_\eps)\big) (v^\beta - \tilde v_\eps^\beta) \partial_{\xi} p(\tilde v_\eps)d\xi \\
&\qquad\quad -\int_\bbr a \partial_\xi \big(p(v)-p(\tilde v_\eps)\big) (v^\beta - \tilde v_\eps^\beta) \partial_{\xi} p(\tilde v_\eps)  d\xi,
\end{aligned}
\end{align*}
and 
\begin{align*}
\begin{aligned}
&-\mathcal{J}^{good}(U):= \underbrace{- \frac{\sigma_\eps}{2}\int_\bbr a'\left| h-\tilde h_\eps\right|^2  {\mathbf 1}_{\{p(v)-\pt \leq\delta\}} d\xi }_{=:J_2} 
 -\frac{\sigma_\eps}{2}\int_\bbr a'\left| h-\tilde h_\eps\right|^2  d\xi  {\mathbf 1}_{\{p(v)-\pt >\delta\}} d\xi \\
&\qquad\quad -\sigma_\eps  \int_\bbr  a' Q(v|\tilde v_\eps) d\xi -\int_\bbr a v^\beta \left|\partial_\xi \big(p(v)-p(\tilde v_\eps)\big)\right|^2 d\xi.
\end{aligned}
\end{align*}
Applying the quadratic identity $\alpha x^2+ \beta x =\alpha(x+\frac{\beta}{2\alpha})^2-\frac{\beta^2}{4\alpha}$ with $x:=h-\tilde h_\eps$ to the integrands of $J_1+J_2$, we find
\begin{align*}
\begin{aligned}
- \frac{\sigma_\eps}{2} \left| h-\tilde h_\eps\right|^2 +\big(p(v)-p(\tilde v_\eps)\big) (h-\tilde h_\eps)  = - \frac{\sigma_\eps}{2} \left| h-\tilde h_\eps -\frac{p(v)-p(\tilde v_\eps)}{\sigma_\eps}\right|^2+ \frac{1}{2\sigma_\eps}|p(v)-p(\tilde v_\eps)|^2.
\end{aligned}
\end{align*}
Therefore, we have the desired representation \eqref{ineq-1}-\eqref{badgood}.
\end{proof}

\subsection{Main proposition}
The main proposition for the proof of Theorem \ref{thm_main} is the following.

\begin{proposition}\label{prop:main}
There exist $\eps_0,\delta_0, \delta_3\in(0,1/2)$ such that for any $\eps<\eps_0$ and $\delta_0^{-1}\eps<\lambda<\delta_0$, the following is true.\\
For any $U\in\mathcal{H} \cap \{U~|~|Y(U)|\le\eps^2 \}$,
\begin{align}
\begin{aligned}\label{prop:est}
\mathcal{R}(U)&:= -\frac{1}{\eps^4}Y^2(U) +\mathcal{B}_{\delta_3}(U)+\delta_0\frac{\eps}{\lambda} |\mathcal{B}_{\delta_3}(U)|\\
&\qquad -\mathcal{G}_1^-(U)-\mathcal{G}_1^+(U) -\left(1-\delta_0\frac{\eps}{\lambda}\right)\mathcal{G}_2(U) -(1-\delta_0)\mathcal{D}(U)  \le 0,
\end{aligned}
\end{align}
where $Y$ and $\mathcal{B}_{\delta_3}$ are as in \eqref{ybg-first} and \eqref{badgood}, and $\mathcal{G}_1^-, \mathcal{G}_1^+, \mathcal{G}_2, \mathcal{D}$ denote the four terms of $\mathcal{G}_{\delta_3}$ as follows:
\begin{align}
\begin{aligned}\label{ggd}
&\mathcal{G}_1^-(U):=\frac{\sigma_\eps}{2}\int_{\Omega^c} a' |h-\tilde h_\eps |^2  {\mathbf 1}_{\{p(v)-\pt >\delta_3\}} d\xi,\\
&\mathcal{G}_1^+(U):=\frac{\sigma_\eps}{2}\int_\Omega a'\Big(h-\tilde h_\eps -\frac{p(v)-p(\tilde v_\eps)}{\sigma_\eps}\Big)^2  {\mathbf 1}_{\{p(v)-\pt \leq\delta_3 \}} d\xi,\\
&\mathcal{G}_2(U):=\sigma_\eps  \int_\bbr  a' Q(v|\tilde v_\eps) d\xi,\\
&\mathcal{D}(U):= \int_\bbr a v^{\beta} |\partial_\xi (p(v)-p(\tilde v_\eps))|^2 d\xi.
\end{aligned}
\end{align} 
\end{proposition}

\subsection{Proof of Theorem~\ref{thm_main} from Proposition \ref{prop:main}}
We will first show how Proposition \ref{prop:main} implies Theorem~\ref{thm_main}.\\

For any fixed $\eps>0$, we consider a continuous function $\Phi_\eps$ defined by
\beq\label{Phi-d}
\Phi_\eps (y)=
\left\{ \begin{array}{ll}
      \frac{1}{\eps^2},\quad \mbox{if}~ y\le -\eps^2, \\
      -\frac{1}{\eps^4}y,\quad \mbox{if} ~ |y|\le \eps^2, \\
       -\frac{1}{\eps^2},\quad \mbox{if}  ~y\ge \eps^2. \end{array} \right.
\eeq
Let $\eps_0,\delta_0, \delta_3$ be the constants in Proposition \ref{prop:main}. Then, let $\eps, \lambda$ be any constants such that $0<\eps<\eps_0$ and $\delta_0^{-1}\eps<\lambda<\delta_0<1/2$.\\

We define a shift function $X(t)$ as a solution of the nonlinear ODE:
\beq\label{X-def}
\left\{ \begin{array}{ll}
       \dot X(t) = \Phi_\eps (Y(U^X)) \Big(2|\mathcal{J}^{bad}(U^X)|+1 \Big),\\
       X(0)=0, \end{array} \right.
\eeq
where $Y$ and $\mathcal{J}^{bad}$ are as in \eqref{ybg-first}.\\
Then, for the solution $U\in \mathcal{H}_T$, there exists a unique absolutely continuous shift $X$ on $[0,T]$. 
Indeed, since $\tilde v_\eps', \tilde h_\eps', a'$ are bounded, smooth and integrable, using $U\in \mathcal{H}_T$ together with the change of variables $\xi\mapsto \xi-X(t)$ as in \eqref{move-X}, we find that there exists $a, b\in L^2(0,T)$ such that
\[
\sup_{x\in\bbr}| F(t,x)| \le a(t)\quad\mbox{and}\quad \sup_{x\in\bbr} |\partial_x F(t,x)| \le b(t),\quad \forall t\in[0,T],
\]
where $F(t,X)$ denotes the right-hand side  of the ODE \eqref{X-def}. For more details on the existence and uniqueness theory of the ODE, we refer to \cite[Lemma A.1]{CKKV}. \\

Based on \eqref{ineq-0} and \eqref{X-def}, to get the contraction estimate \eqref{cont_main2}, it is enough to prove that for almost every time $t>0$ ,
\beq\label{contem0}
\Phi_\eps (Y(U^X)) \Big(2|\mathcal{J}^{bad}(U^X)|+1 \Big) Y(U^X) +\mathcal{J}^{bad} (U^X) -\mathcal{J}^{good} (U^X) \le0.
\eeq
We define 
\[
\mathcal{F}(U):=\Phi_\eps (Y(U)) \Big(2|\mathcal{J}^{bad}(U)|+1 \Big)Y(U) +\mathcal{J}^{bad}(U) -\mathcal{J}^{good} (U) ,\quad \forall U\in \mathcal{H}.
\]
From \eqref{Phi-d}, we have 
\beq\label{XY}
\Phi_\eps (Y) \Big(2|\mathcal{J}^{bad}|+1 \Big)Y\le
\left\{ \begin{array}{ll}
     -2|\mathcal{B}_{\delta_3}|,\quad \mbox{if}~  |Y|\ge \eps^2,\\
     -\frac{1}{\eps^4}Y^2,\quad  \mbox{if}~ |Y|\le \eps^2. \end{array} \right.
\eeq
Hence,   for all $U\in \mathcal{H}$ satisfying  $|Y(U)|\ge \eps^2 $, we have 
$$
\mathcal{F}(U) \le -|\mathcal{J}^{bad}(U)|-\mathcal{J}^{good}(U) \le 0.
$$
Using \eqref{ineq-1}, \eqref{XY} and Proposition \ref{prop:main}, we find that for all $U\in \mathcal{H}$ satisfying  $|Y(U)|\le \eps^2 $, 
$$
\mathcal{F}(U) \le -\delta_0\frac{\eps}{\lambda}|\mathcal{B}_{\delta_3}(U)| -\delta_0\frac{\eps}{\lambda}\mathcal{G}_2(U) - \delta_0\mathcal{D}(U) \le 0.
$$
 Since $\delta_0 <1/2$,   these two estimates show that for every $U\in \mathcal{H}$ we have 
 $$
\mathcal{F}(U) \le -\delta_0\frac{\eps}{\lambda} |\mathcal{B}_{\delta_3}(U)| -\delta_0\frac{\eps}{\lambda}\mathcal{G}_2(U)  - \delta_0\mathcal{D}(U).
$$
Thus, using the above estimates with $U=U^X$, together with  \eqref{ineq-0}, \eqref{contem0} and the definition of $\mathcal{I}^{good}$, we find that for a.e. $t>0$,
\begin{align}
\begin{aligned}\label{111}
&\frac{d}{dt}\int_{\bbr} a\eta(U^X|\tilde U_\eps)+ \delta_0\frac{\eps}{\lambda}\mathcal{G}_2(U^X) + \delta_0\mathcal{D}(U^X) d\xi = \mathcal{F}(U^X)+ \delta_0\frac{\eps}{\lambda}\mathcal{G}_2(U^X) + \delta_0\mathcal{D}(U^X) \\
&\qquad\qquad\qquad\qquad\qquad\qquad \le -|\mathcal{J}^{bad}(U^X)| \mathbf{1}_{\{|Y(U^X)|\ge\eps^2\}} -\delta_0\frac{\eps}{\lambda} |\mathcal{B}_{\delta_3}(U^X)| \mathbf{1}_{\{|Y(U^X)|\le\eps^2\}} \le 0.
\end{aligned}
\end{align}
Therefore we have 
\beq\label{cont-pre}
\int_{\bbr} a\eta(U^X|\tilde U_\eps) d\xi  + \delta_0\frac{\eps}{\lambda}\mathcal{G}_2(U^X) + \delta_0\mathcal{D}(U^X)\le \int_{\bbr} a\eta(U_0|\tilde U_\eps) d\xi <\infty ,
\eeq
which completes \eqref{cont_main2}.\\

To estimate $|\dot X|$, we first observe that \eqref{Phi-d} and \eqref{X-def} yield
\beq\label{contx}
|\dot X| \le \frac{1}{\eps^2} (2|\mathcal{J}^{bad}(U^X)| +1) , \quad\mbox{for a.e. } t\in (0,T),
\eeq
Notice that it follows from \eqref{111} and $1/2\le a\le 1$ by $\delta_0<1/2$ that
\beq\label{jb-cont}
\int_0^{T}\left(|\mathcal{J}^{bad}(U^X)| \mathbf{1}_{\{|Y(U^X)|\ge\eps^2\}} + \delta_0\frac{\eps}{\lambda} |\mathcal{B}_{\delta_3}(U^X)| \mathbf{1}_{\{|Y(U^X)|\le\eps^2\}} \right)dt \le  2\int_{\bbr} \eta(U_0|\tilde U_\eps) d\xi.
\eeq
To estimate $|\mathcal{J}^{bad}(U^X)|$ globally in time, using  \eqref{ineq-1} and the definitions  of $\mathcal{I}^{good}$ and $\mathcal{G}_{\delta_3}$,
we find that
\begin{align*}
\begin{aligned}
&|\mathcal{J}^{bad}(U^X)|\\
&=|\mathcal{J}^{bad}(U^X)| {\mathbf 1}_{\{|Y(U^X)|\ge \eps^2\}} +|\mathcal{J}^{bad}(U^X)| {\mathbf 1}_{\{|Y(U^X)|\le \eps^2\}}  \\
&=|\mathcal{J}^{bad}(U^X)| {\mathbf 1}_{\{|Y(U^X)|\ge \eps^2\}} +| \mathcal{J}^{good}(U^X) + \mathcal{B}_{\delta_3}(U^X)-\mathcal{G}_{\delta_3}(U^X)| {\mathbf 1}_{\{|Y(U^X)|\le \eps^2\}}\\
&\le |\mathcal{J}^{bad}(U^X)| {\mathbf 1}_{\{|Y(U^X)|\ge \eps^2\}}  +| \mathcal{B}_{\delta_3}(U^X)|{\mathbf 1}_{\{|Y(U^X)|\le \eps^2\}}\\
&\quad + \frac{|\sigma| }{2}\int_\bbr |a'| \Big| \big(h^X-\tilde h_\eps \big)^2 - \left(h^X-\tilde h_\eps -\frac{p(v)-p(\tilde v_\eps)}{\sigma_\eps}\right)^2 \Big|{\mathbf 1}_{\{p(v)-\pt \leq\delta_3 \}}  d\xi\\
&\le |\mathcal{J}^{bad}(U^X)| {\mathbf 1}_{\{|Y(U^X)|\ge \eps^2\}}  +| \mathcal{B}_{\delta_3}(U^X)|{\mathbf 1}_{\{|Y(U^X)|\le \eps^2\}}\\
&\quad  +C \int_\bbr |a'| \Big( \big(h^X-\tilde h_\eps \big)^2 + \big(p(v^X)-p(\tilde v_\eps) \big)^2 \Big) {\mathbf 1}_{\{p(v^X)-\pt \leq\delta_3 \}}  d\xi .
\end{aligned}
\end{align*}
Since for any $v$ satisfying $p(v)-\pt \leq\delta_3$, there exists a positive constant $c_*$ such that $v>c_*^{-1}$ and $|p(v)-\pt|\le c_*$, we use \eqref{pressure2} and \eqref{rel_Q} to have
\begin{align*}
\begin{aligned}
& \int_\bbr \big(p(v)-p(\tilde v_\eps) \big)^2 {\mathbf 1}_{\{p(v)-\pt \leq\delta_3 \}}  d\xi \\
&\quad  \le  c_* \int_{v>c_*^{-1}} \big|p(v)-p(\tilde v_\eps) \big| {\mathbf 1}_{\{v\ge 3v_- \}}  d\xi + \int_{v>c_*^{-1}} \big|p(v)-p(\tilde v_\eps) \big|^2 {\mathbf 1}_{\{v\le 3v_- \}}  d\xi  \\
&\quad \le C\int_{v>c_*^{-1}} \Big( |v-\tilde v_\eps| {\mathbf 1}_{\{v\ge 3v_- \}} + |v-\tilde v_\eps|^2 {\mathbf 1}_{\{v\le 3v_- \}} \Big) d\xi \le C\int_\bbr Q(v| \tilde v_\eps) d\xi .
\end{aligned}
\end{align*}
Therefore, using $a' \leq C \delta_0$ and $\delta_0\le \frac{1}{2}\le a$, we have
\[
|\mathcal{J}^{bad}(U^X)|\le  |\mathcal{J}^{bad}(U^X)| {\mathbf 1}_{\{|Y(U^X)|\ge \eps^2\}}  +| \mathcal{B}_{\delta_3}(U^X)|{\mathbf 1}_{\{|Y(U^X)|\le \eps^2\}} + C \int_\bbr a\eta(U^X|\tilde U_\eps) d\xi,  
\]
which together with \eqref{cont-pre} and \eqref{contx} implies
\[
|\dot X|\le \frac{1}{\eps^2}\left[ \Big( |\mathcal{J}^{bad}(U^X)| {\mathbf 1}_{\{|Y(U^X)|\ge \eps^2\}}  +| \mathcal{B}_{\delta_3}(U^X)|{\mathbf 1}_{\{|Y(U^X)|\le \eps^2\}}  \Big) + C \int_{\bbr} \eta(U_0|\tilde U_\eps) d\xi + 1 \right],
\]
and \eqref{jb-cont} implies
\[
\int_0^{T}\Big( |\mathcal{J}^{bad}(U^X)| {\mathbf 1}_{\{|Y(U^X)|\ge \eps^2\}}  +| \mathcal{B}_{\delta_3}(U^X)|{\mathbf 1}_{\{|Y(U^X)|\le \eps^2\}}  \Big)  dt  \le \frac{2\lambda}{\delta_0\eps}\int_\bbr \eta(U_0|\tilde U_\eps)  d\xi.
\]
Hence we complete \eqref{est-shift1}.\\

\vskip0.1cm The rest of this section is dedicated to the proof of Proposition \ref{prop:main}.

\subsection{Expansion in the size of the shock}\label{section-expansion}

We define the following functionals:
\begin{align}
\begin{aligned}\label{note-in}
&Y_g(v):=-\frac{1}{2\sigma_\eps^2}\int_\bbr a' |p(v)-p(\tilde v_\eps)|^2 d\xi -\int_\bbr a' Q(v|\tilde v_\eps) d\xi -\int_\bbr a \partial_\xi p(\tilde v_\eps)(v-\tilde v_\eps)d\xi\\
&\quad\qquad+\frac{1}{\sigma_\eps}\int_\bbr a \partial_\xi \tilde h_\eps\big(p(v)-p(\tilde v_\eps)\big)d\xi,\\
&\mathcal{I}_1(v):= \sigma_\eps\int_\bbr a \partial_\xi \tilde v_\eps p(v|\tilde v_\eps) d\xi,\\
&\mathcal{I}_2(v):= \frac{1}{2\sigma_\eps} \int_\bbr a' |p(v)-p(\tilde v_\eps)|^2d\xi,\\
&\mathcal{G}_2(v):=\sigma_\eps  \int_\bbr  a' Q(v|\tilde v_\eps) d\xi, \\
&\mathcal{D}(v):=\int_\bbr a\, v^\beta |\partial_\xi (p(v)-p(\tilde v_\eps))|^2 d\xi.
\end{aligned}
\end{align}
Note that all these quantities depend only on $v$ (not on $h$).

\begin{proposition}\label{prop:main3}
For any constant $C_2>0$, there exist $\eps_0, \delta_3>0$, such that for any $\eps\in(0,\eps_0)$, and any $\lambda, \delta\in(0,\delta_3)$ such that $\eps\leq \lambda$,
 the following is true.\\
For any function $v:\bbr\to \bbr^+$ such that $\mathcal{D}(v)+\mathcal{G}_2(v)$
 is finite, if
\beq\label{assYp}
|Y_g(v)|\leq C_2 \frac{\eps^2}{\lambda},\qquad  \|p(v)-p(\vt)\|_{L^\infty(\bbr)}\leq \delta_3,
\eeq
then
\begin{align}
\begin{aligned}\label{redelta}
\mathcal{R}_{\eps,\delta}(v)&:=-\frac{1}{\eps\delta}|Y_g(v)|^2 +\mathcal{I}_1(v)+\delta|\mathcal{I}_1(v)|\\
&\quad\quad+\mathcal{I}_2(v)+\delta\left(\frac{\eps}{\lambda}\right)|\mathcal{I}_2(v)|-\left(1-\delta\left(\frac{\eps}{\lambda}\right)\right)\mathcal{G}_2(v)-(1-\delta)\mathcal{D}(v)\le 0,
\end{aligned}
\end{align}
where note that $\mathcal{I}_1, \mathcal{I}_2 \ge 0$.
\end{proposition}
\begin{proof}
The proof is almost the same as that of \cite[Proposition 3.4]{Kang-V-NS17}, because $Y_g, \mathcal{I}_1, \mathcal{I}_2, \mathcal{G}_2$ defined in \eqref{note-in} are the exactly same functionals as in \cite[Proposition 3.4]{Kang-V-NS17}, and the diffusion $\mathcal{D}$ is slightly different but has the same expansion. 
Note that this proposition corresponds to an expansion in $p(v)$ near $p(\vt)$ (up to $\delta_3$) for $p(\vt)$ close to $p(v_-)$ (up to $\eps$). It is therefore natural that the expansion is similar to the case of $\alpha=\gamma$ as in \cite[Proposition 3.4]{Kang-V-NS17}, since the viscosity is almost constant near $p(v_-)$.
For completeness, the main part of the proof is given in Appendix \ref{app-exp}.
\end{proof}

\subsection{Truncation of the big values of $|p(v)-p(\tilde v_\eps)|$}\label{section-finale}
In order to use Proposition \ref{prop:main3} in the proof of Proposition \ref{prop:main}, we need to show that  the values for $p(v)$ such that $|p(v)-p(\tilde{v}_\eps)|\geq\delta_3$ have a small effect. However, the value of $\delta_3$ is itself conditioned to the constant $C_2$ in the proposition. Therefore, we need first to find a uniform bound on $Y_g$ which is not yet conditioned on the level of truncation $k$.

We consider a truncation on $|p(v)-p(\tilde v_\eps)|$ with a  constant $k>0$. Later we will consider the case  $k=\delta_3$ as in Proposition \ref{prop:main3}. But for now, we consider the general case $k$ to estimate the constant $C_2$.  For that, let $\psi_k$ be a continuous function defined by
\beq\label{psi}
 \psi_k(y)=\inf\left(k,\sup(-k,y)\right).
\eeq
We then define the function $\bar{v}_k$ uniquely (since the function $p$ is one to one) as
\beq\label{trunc-def}
p(\bar v_k)-\pt=\psi_k(p(v)-\pt).
\eeq
We have the following lemma (see \cite[Lemma 3.2]{Kang-V-NS17}).
\begin{lemma}\label{lemmeC2}
For a fixed $v_-\geq0$, $u_-\in\bbr$, there exists $C_2, k_0, \eps_0, \delta_0>0$ such that for any $\eps\leq \eps_0$, $\eps/\lambda\leq \delta_0$ with $\lambda<1/2$, the following is true whenever $|Y(U)|\leq \eps^2$:
\begin{eqnarray}
\label{l1}
&& \int_\bbr|a'||h-\tilde{h}_\eps|^2\,d\xi + \int_\bbr|a'| Q(v|\tilde{v}_\eps)\,d\xi\leq C\frac{\eps^2}{\lambda}, \\
\label{lbis}
&& |Y_g(\bar v_k)|\leq C_2\frac{\eps^2}{\lambda}, \qquad \mathrm{for \ every \ } k\leq k_0.
\end{eqnarray}
\end{lemma}

We now fix the constant $\delta_3$ of Proposition \ref{prop:main3} associated to the constant $C_2$ of Lemma \ref{lemmeC2}. Without loss of generality, we can assume that $\delta_3< k_0$ (since Proposition \ref{prop:main3} is valid for any smaller $\delta_3$). From now on, we set (without confusion)
$$
\bar v:=\bar v_{\delta_3}, \qquad  \bar{U}:=(\bar v, h), \qquad \mathcal{B}:=\mathcal{B}_{\delta_3}, \qquad \mathcal{G}:=\mathcal{G}_{\delta_3}.
$$
Note from Lemma \ref{lemmeC2} that  
\begin{equation}\label{YC2}
|Y_g(\bar v)|\leq C_2 \frac{\eps^2}{\lambda}.
\end{equation}

We first recall the terms $Y$ in \eqref{ybg-first} as
\[
Y= -\int_\bbr a' \frac{|h-\tilde h_\eps|^2}{2} d\xi -\int_\bbr a' Q(v|\tilde v_\eps) d\xi +\int_\bbr a \Big(-\partial_\xi p(\tilde v_\eps)(v-\tilde v_\eps) +\partial_\xi \tilde h_\eps(h-\tilde h_\eps) \Big) d\xi.
\]

In what follows, for simplification, we use the notation:
 \[
 \Omega:=\{\xi~|~ (p(v)-\pt)(\xi) \leq\delta_3\}. 
 \]

We split $Y$ into four parts $Y_g$,  $Y_b$, $Y_l$ and $Y_s$ as follows: \\
\[
Y= Y_g +Y_b +Y_l + Y_s,
\]
where 
\begin{align*}
\begin{aligned}
Y_g&:= -\frac{1}{2\sigma_\eps^2}\int_\Omega a' |p(v)-p(\tilde v_\eps)|^2 d\xi -\int_\Omega a' Q(v|\tilde v_\eps) d\xi -\int_\Omega a \partial_\xi p(\tilde v_\eps)(v-\tilde v_\eps)d\xi\\
&\quad+\frac{1}{\sigma_\eps}\int_\Omega a \partial_\xi \tilde h_\eps\big(p(v)-p(\tilde v_\eps)\big)d\xi,\\
Y_b&:= -\frac{1}{2}\int_\Omega a' \Big(h-\tilde h_\eps-\frac{p(v)-p(\tilde v_\eps)}{\sigma_\eps}\Big)^2 d\xi \\
&\quad -\frac{1}{\sigma_\eps} \int_\Omega a' \big(p(v)-p(\tilde v_\eps)\big)\Big(h-\tilde h_\eps-\frac{p(v)-p(\tilde v_\eps)}{\sigma_\eps}\Big) d\xi,\\
Y_l&=\int_\Omega a \partial_\xi \tilde h_\eps\Big(h-\tilde h_\eps-\frac{p(v)-p(\tilde v_\eps)}{\sigma_\eps}\Big)d\xi,\\
Y_s&=-\int_{\Omega^c} a' Q(v|\tilde v_\eps) d\xi -\int_{\Omega^c} a \partial_\xi p(\tilde v_\eps)(v-\tilde v_\eps)d\xi \\
&\quad -\int_{\Omega^c} a' \frac{|h-\tilde h_\eps|^2}{2} d\xi +\int_{\Omega^c} a \partial_\xi \tilde h_\eps(h-\tilde h_\eps) d\xi.
\end{aligned}
\end{align*}
Notice that $Y_g$ consists of the terms related to $v-\tilde v_\eps$,  while $Y_b$ and $Y_l$ consist of   terms related to $h-\tilde h_\eps$. While $Y_b$ is quadratic, and $Y_l$ is linear in $h-\tilde h_\eps$. 
Since $\{\xi~|~ |p(v)-\pt| \leq\delta_3\} \subset \Omega$, $Y_g(\bar U)$ is the same as $Y_g(v)$ in Proposition \ref{prop:main3}. Therefore we need show that $Y_g(U)-Y_g(\bar U)$, $Y_b(U)$, $Y_l(U)$ and $Y_s(U)$ are negligible by the good term $\mathcal{G}$. 

For the bad terms $\mathcal{B}_{\delta_3}$ in \eqref{badgood}, we will use the following notations :
\beq\label{bad0}
\mathcal{B}_{\delta_3}=\mathcal{B}_1+\mathcal{B}_2^- +\mathcal{B}_2^+ +\mathcal{B}_3 +\mathcal{B}_4+\mathcal{B}_5,
\eeq
where
\begin{align*}
\begin{aligned}
&\mathcal{B}_1=  \sigma_\eps\int_\bbr a \partial_\xi \tilde v_\eps p(v|\tilde v_\eps) d\xi,\\
&\mathcal{B}_2^- =\int_{\Omega^c} a' \big(p(v)-p(\tilde v_\eps)\big) (h-\tilde h_\eps)d\xi , \qquad \mathcal{B}_2^+ = \frac{1}{2\sigma_\eps} \int_\Omega a' |p(v)-p(\tilde v_\eps)|^2 d\xi,\\
&\mathcal{B}_3= -\int_\bbr a' v^\beta \big(p(v)-p(\tilde v_\eps)\big)\partial_\xi \big(p(v)-p(\tilde v_\eps)\big)  d\xi,\\
&\mathcal{B}_4= -\int_\bbr a' \big(p(v)-p(\tilde v_\eps)\big) (v^\beta - \tilde v_\eps^\beta) \partial_{\xi} p(\tilde v_\eps)d\xi, \\
&\mathcal{B}_5= -\int_\bbr a \partial_\xi \big(p(v)-p(\tilde v_\eps)\big) (v^\beta - \tilde v_\eps^\beta) \partial_{\xi} p(\tilde v_\eps)  d\xi.
\end{aligned}
\end{align*}
We also recall the notations $\mathcal{G}_1^-,\mathcal{G}_1^+, \mathcal{G}_2, \mathcal{D}$ in \eqref{ggd} for the good terms.

We now state the following proposition.

\begin{proposition}\label{prop_out}
There exist constants $\eps_0, \deo, C, C^*>0$ (in particular, $C$ depends on the constant $\delta_3$ in Proposition \ref{prop:main}) such that for any $\eps<\eps_0$ and $\delta_0^{-1}\eps<\lambda<\delta_0<1/2$, the following statements hold true.
\begin{itemize}
\item[1.] For any $U$ such that $|Y(U)|\leq \eps^2$,
\begin{eqnarray}
\label{n1}
&&|\mathcal{B}_1(U)-\mathcal{B}_1(\bar U)| \leq C\frac{\eps}{\lambda} \left( \mathcal{D}(U) + \left(\mathcal{G}_2(U)-\mathcal{G}_2(\bar U) \right)+\left(\frac{\eps}{\lambda}\right)^2\mathcal{G}_2(\bar U) \right) ,\\
\label{n12}
&&|\mathcal{B}_2^-(U)|  \le\delta_0 \left(\mathcal{D}(U) +\left(\frac{\eps}{\lambda}\right)^2\mathcal{G}_2(U) \right) + \frac{1}{2} \mathcal{G}_1^-(U),\\
\label{n13}
&&|\mathcal{B}_2^+(U)-\mathcal{B}_2^+(\bar U)|   \le \sqrt{\frac{\eps}{\lambda}}\mathcal{D}(U),\\
\label{n14}
&&|\mathcal{B}_3(U)|+|\mathcal{B}_4(U)|+|\mathcal{B}_5(U)|  \le C \delta_0 \left( \mathcal{D}(U) +\left(\mathcal{G}_2(U)-\mathcal{G}_2(\bar U) \right)+\frac{\eps}{\lambda} \mathcal{G}_2(\bar U)\right),\\
\label{n2}
&&|\mathcal{B}_{\delta_3}(U)| \le C^*\frac{\eps^2}{\lambda} + C \sqrt{\delta_0} \mathcal{D}(U).
\end{eqnarray}
\item[2.] For any $U$ such that $|Y(U)|\leq \eps^2$ and $\mathcal{D}(U)\leq \frac{C^*}{4}\frac{\eps^2}{\lambda}$,
\begin{align}
\begin{aligned}\label{m1}
&|Y_g(U)-Y_g(\bar U)|^2 +|Y_b(U)|^2+|Y_l(U)|^2+|Y_s(U)|^2 \\
&\quad\le C\frac{\eps^2}{\lambda}\left(\sqrt{\frac{\eps}{\lambda}}\mathcal{D}(U)+\left(\mathcal{G}_2(U)-\mathcal{G}_2(\bar U) \right)+\mathcal{G}_1^-(U) + \left(\frac{\lambda}{\eps}\right)^{1/4}\mathcal{G}_1^+(U) +\left(\frac{\eps}{\lambda}\right)^{1/4}\mathcal{G}_2(\bar U) \right).
\end{aligned}
\end{align}
\end{itemize}
\end{proposition}

To prove this proposition, 
we will control the bad terms in different ways for each case of small or big values of $v$, which all correspond to the big values of $|p(v)-p(\tilde v_\eps)|$ (as $|p(v)-p(\tilde v_\eps)|\ge\delta_3$). 
For that, we set
\beq\label{chi}
p(\bar v_s)- p(\tilde v_\eps) := \psi_{\delta_3}^s\big(p(v)- p(\tilde v_\eps)\big),\quad p(\bar v_b)- p(\tilde v_\eps) := \psi_{\delta_3}^b
\big(p(v)- p(\tilde v_\eps)\big),
\eeq
where $\psi_{\delta_3}^s$ and $\psi_{\delta_3}^b$ are one-sided truncations of $\psi_{\delta_3}$ defined in \eqref{psi}, i.e.,
\[\psi_{\delta_3}^s(y)=\inf (\delta_3,y),\quad \psi_{\delta_3}^b(y)=\sup (-\delta_3,y).\]
Notice that the function $\bar v_s$ (resp. $\bar v_b$) represents the truncation of small (resp. big) values of $v$ corresponding to $|p(v)-p(\tilde v_\eps)|\ge\delta_3$.\\

By comparing the definitions of \eqref{trunc-def} and \eqref{chi}, we see
\begin{align}
\begin{aligned}\label{compare1}
& \left( p(\bar v_s)-\pt \right) {\mathbf 1}_{\{p(v)-\pt \ge -\delta_3\}} = \left( p(\bar v)-\pt  \right) {\mathbf 1}_{\{p(v)-\pt \ge -\delta_3\}}  ,\\
& \left( p(\bar v_b)-\pt \right) {\mathbf 1}_{\{p(v)-\pt \le \delta_3\}} = \left( p(\bar v)-\pt\right) {\mathbf 1}_{\{p(v)-\pt \le \delta_3\}}  .
\end{aligned}
\end{align}
We also note that 
\begin{equation}\label{def-bar}
\begin{array}{rl}
p(v)-p(\bar v_s)=& (p(v)-\pt)+(\pt-p(\bar v_s)) \\[0.2cm]
=&\left(I-\psi^s_{\delta^3}\right)\left(p(v)-\pt\right) \\[0.2cm]
=& \left(\left(p(v)-\pt \right)-\delta_3\right)_+,\\[0.2cm]
p(\bar v_b)-p(v)=& (p(\bar v_b)-\pt)+(\pt-p(v)) \\[0.2cm]
=&\left(\psi^b_{\delta^3}-I\right)\left(p(v)-\pt\right) \\[0.2cm]
=& \left(-\left(p(v)-\pt\right)-\delta_3\right)_+,\\
|p(v)-p(\bar v)|=& |(p(v)-\pt)+(\pt-p(\bar v))|\\[0.2cm]
=&|(I-\psi_\delta)(p(v)-\pt)|\\[0.2cm]
=& (|p(v)-\pt|-\delta_3)_+.
\end{array}
\end{equation}
Therefore, using \eqref{trunc-def}, \eqref{chi} and \eqref{def-bar}, we have
\begin{align}
\begin{aligned}\label{eq_D}
\mathcal{D}(U)&=\int_\bbr av^{\beta}  |\partial_\xi (p(v)-p(\tilde v_\eps))|^2 d\xi\\
&=\int_\bbr a v^{\beta} |\partial_\xi (p(v)-p(\tilde v_\eps))|^2 ( {\mathbf 1}_{\{|p(v)-\pt |\leq\delta_3\}} + {\mathbf 1}_{\{p(v)-\pt >\delta_3\}}+ {\mathbf 1}_{\{p(v)-\pt <-\delta_3\}}  )d\xi\\
&=\mathcal{D}(\bar U)+\int_\bbr a v^{\beta} |\partial_\xi (p(v)-p(\bar v_s))|^2 d\xi+\int_\bbr av^{\beta}  |\partial_\xi (p(v)-p(\bar v_b))|^2 d\xi\\
&\ge \int_\bbr a v^{\beta} |\partial_\xi (p(v)-p(\bar v_s))|^2 d\xi+\int_\bbr av^{\beta}  |\partial_\xi (p(v)-p(\bar v_b))|^2 d\xi.
\end{aligned}
\end{align}
On the other hand, since $Q(v|\tilde{v}_\eps)\geq  Q(\bar v|\tilde{v}_\eps)$,
we have
\begin{equation}\label{eq_G}
|\sigma_\eps| \int_\bbr |a'|Q(v|\tilde{v}_\eps)\,d\xi\geq \mathcal{G}_2(U)-\mathcal{G}_2(\bar U)=|\sigma_\eps|\int_\bbr |a'|\left(Q(v|\tilde{v}_\eps)-Q(\bar v|\tilde{v}_\eps)\right)\,d\xi\geq 0,
\end{equation}
which together with  \eqref{l1} yields
\beq\label{l2}
0\leq \mathcal{G}_2(U)-\mathcal{G}_2(\bar U)\leq  \mathcal{G}_2(U)\leq C\frac{\eps^2}{\lambda}.
\eeq

We first present a series of following lemmas.

\begin{lemma}\label{lemma_out1}
Under the same assumption as Proposition \ref{prop_out}, we have
\begin{eqnarray}
\label{big1}
&& \int_\bbr|a'| \big| p(v)-p(\bar v_b) \big|^2 d\xi + \int_\bbr|a'| \big| p(v)-p(\bar v_b) \big| d\xi  \le \sqrt{\frac{\eps}{\lambda}}\mathcal{D}(U),\\
\label{big2}
&&  \int_\bbr|a'| \Big | |p(v)-\pt|^2-  |p(\bar v_b)-\pt|^2\Big| {\mathbf 1}_{\{p(v)-\pt \le \delta_3\}} \,d\xi \le \sqrt{\frac{\eps}{\lambda}}\mathcal{D}(U).
\end{eqnarray}
\begin{eqnarray}
\label{l3}
&& \int_\bbr|a'|^2 v^\beta  |p(v)-p(\bar v)|^2 \,d\xi + \int_\bbr|a'|^2 v^\beta  |p(v)-p(\bar v)| \,d\xi \le C\lambda^2 \left( \mathcal{D}(U) +\frac{\eps}{\lambda} \mathcal{G}_2(U)\right) ,\\
\label{l4}
&&  \int_\bbr|a'|^2 \Big |v^\beta |p(v)-\pt|^2- \bar v^\beta |p(\bar v)-\pt|^2\Big| \,d\xi \le C\lambda^2 \left( \mathcal{D}(U) +\frac{\eps}{\lambda} \mathcal{G}_2(U)\right).
\end{eqnarray}
\end{lemma}
\begin{proof}
\noindent{\bf Proof of \eqref{big1}:} 
We split the proof into two steps.
\vskip0.2cm
\noindent{\it Step 1:} Note first that since $(y-\delta_3/2)_+\geq \delta_3/2$ whenever $(y-\delta_3)_+>0$, we have
\beq\label{y-identity}
(y-\delta_3)_+\leq (y-\delta_3/2)_+{\mathbf 1}_{\{y-\delta_3>0\}}\leq (y-\delta_3/2)_+\left(\frac{(y-\delta_3/2)_+}{\delta_3/2}\right)\leq \frac{2}{\delta_3} (y-\delta_3/2)_+^2.
\eeq
Hence, to show \eqref{big1}, it is enough to show it only for the quadratic part, with $\bar v_b$ defined with $\delta_3/2$ instead of $\delta_3$. We will keep the notation $\bar v_b$ in Step 2 below.

\vskip0.2cm
\noindent{\it Step 2:} 
First, using \eqref{eq_D}, we find that for any $\xi\in\bbr$,
\begin{align}
\begin{aligned}\label{est6}
 |p(v)-p(\bar v_b)(\xi)|  &\le \int_{\xi_0}^\xi \left|\partial_\xi \big( p(v)-p(\bar v_b)\big) \right| {\mathbf 1}_{\{p(v)-\pt < -\delta_3\}}  \,d\xi \\
 &\le C\int_{\xi_0}^\xi v^{\beta/2}\left|\partial_\xi \big( p(v)-p(\bar v_b)\big) \right| {\mathbf 1}_{\{p(v)-\pt < -\delta_3\}}  \,d\xi \\
 &\le C\sqrt{|\xi|+\frac{1}{\eps}}\sqrt{\mathcal{D}(U)}.
 \end{aligned}
\end{align}
For any $\xi$ such that $|(p(v)-p(\bar v))(\xi)|>0$, we have from  \eqref{def-bar} that $|(p(v)-\pt)(\xi)|> \delta_3$. Thus using \eqref{pressure2} and \eqref{rel_Q}, we have $Q(v(\xi)|\tilde{v}_\eps(\xi))\geq \alpha$, for some constant $\alpha>0$ depending only on $\delta_3$. Thus,
\begin{equation}\label{beta2}
{\mathbf 1}_{\{|p(v)-p(\bar v)|>0\}}\leq \frac{Q(v|\tilde{v}_\eps)}{\alpha}.
\end{equation}
Since ${\mathbf 1}_{\{|p(v)-p(\bar v_b)|>0\}}\le {\mathbf 1}_{\{|p(v)-p(\bar v)|>0\}}$, using \eqref{est6}, \eqref{beta2} and \eqref{l1}, we estimate
\begin{align}
\begin{aligned}\label{stand1}
&\int_\bbr|a'| \big| p(v)-p(\bar v_b) \big|^2 d\xi\\
&\le \int_{|\xi|\le\frac{1}{\eps}\sqrt{\frac{\lambda}{\eps}}}|a'||p(v)-p(\bar v_b)|^2  \,d\xi+ \int_{|\xi|\geq\frac{1}{\eps}\sqrt{\frac{\lambda}{\eps}}}|a'||p(v)-p(\bar v_b)|^2 \,d\xi\\
&\leq\left(\sup_{\left[-\sqrt{\frac{\lambda}{\eps^3}},\sqrt{\frac{\lambda}{\eps^3}}\right]}|p(v)-p(\bar v_s)|\right)
\int_{|\xi|\le\frac{1}{\eps}\sqrt{\frac{\lambda}{\eps}}}|a'| {\mathbf 1}_{\{|p(v)-p(\bar v)|>0\}} \,d\xi\\
&\qquad+C\mathcal{D}(U)  \int_{|\xi|\geq\frac{1}{\eps}\sqrt{\frac{\lambda}{\eps}}}|a'| \left(|\xi|+\frac{1}{\eps}\right)\,d\xi\\
&\leq C \mathcal{D}(U) \left(\sqrt{\frac{\lambda}{\eps^3}}
\int_\bbr|a'|\frac{Q(v|\tilde{v}_\eps)}{\alpha}\,d\xi+2\int_{|\xi|\geq\frac{1}{\eps}\sqrt{\frac{\lambda}{\eps}}}|a'| |\xi|\,d\xi\right).
\end{aligned}
\end{align} 
Therefore we have
\[
\int_\bbr|a'| \big| p(v)-p(\bar v_b) \big|^2 d\xi \le C\sqrt{\frac{\eps}{\lambda}}\mathcal{D}(U).
\]
Indeed, using \eqref{l1} and \eqref{tail} (recalling $|a'|=(\lambda/\eps)|\tilde{v}'_\eps|$), we have
\[
 \int_\bbr|a'||p(v)-p(\bar v)|^2\,d\xi \leq  C \mathcal{D}(U)\left(\sqrt{\frac{\eps}{\lambda}} +\lambda \eps
\int_{|\xi|\geq\frac{1}{\eps}\sqrt{\frac{\lambda}{\eps}}}e^{-c\eps|\xi|} |\xi|\,d\xi \right),
\]
and for the last term, we take $\deo$ small enough such that for any $\eps/\lambda\leq \deo$,
$$
\lambda \eps\int_{|\xi|\geq\frac{1}{\eps}\sqrt{\frac{\lambda}{\eps}}}e^{-c\eps|\xi|} |\xi|\,d\xi =\frac{\lambda}{\eps} \int_{|\xi|\geq \sqrt{\frac{\lambda}{\eps}}}  e^{-c|\xi|} |\xi| d\xi \le \frac{\lambda}{\eps} \int_{|\xi|\geq \sqrt{\frac{\lambda}{\eps}}} e^{-\frac{c}{2}|\xi|}d\xi=\frac{2\lambda}{c\eps}e^{-\frac{c}{2}\sqrt{\frac{\lambda}{\eps}}}\leq \sqrt{\frac{\eps}{\lambda}}.
$$

As mentioned in Step 1, recall that $\bar v_b=\bar v_{b_{\delta_3/2}}$ in the above estimate. Then using \eqref{def-bar}, we have
\begin{align}
\begin{aligned}\label{shift1}
\int_\bbr|a'|^2 v^\beta |p(v)-p(\bar v_{b_{\delta_3}})|^2\,d\xi &=\int_\bbr|a'|^2 v^\beta (-(p(v)-\pt)-\delta_3)_+^2\,d\xi \\
&\le \int_\bbr|a'|^2 v^\beta (-(p(v)-\pt)-\delta_3/2)_+^2\,d\xi\\
&=\int_\bbr|a'|^2 v^\beta |p(v)-p(\bar v_{b_{\delta_3/2}})|^2\,d\xi \\
&\leq C\sqrt{\frac{\eps}{\lambda}}\mathcal{D}(U).
\end{aligned}
\end{align}
For the linear part, using \eqref{def-bar} and \eqref{y-identity} with $y:=-(p(v)-\pt)$, we have
\begin{align}
\begin{aligned}\label{shift2}
\int_\bbr|a'|^2 v^\beta |p(v)-p(\bar v_{b_{\delta_3}})|\,d\xi &\le \frac{2}{\delta_3}\int_\bbr|a'|^2 v^\beta  |p(v)-p(\bar v_{b_{\delta_3/2}})|^2 \,d\xi\\
& \leq  C\sqrt{\frac{\eps}{\lambda}}\mathcal{D}(U).
\end{aligned}
\end{align}
Hence, we obtain \eqref{big1}.\\

\noindent{\bf Proof of \eqref{big2}:}  Since it follows from \eqref{compare1} that 
\[
|p(\bar v_b)-\pt| {\mathbf 1}_{\{p(v)-\pt \le \delta_3\}}=|p(\bar v)-\pt|{\mathbf 1}_{\{p(v)-\pt \le \delta_3\}} \le \delta_3,
\]
using \eqref{big1}, we have
\begin{align}
\begin{aligned}\label{dif1}
&\int_\bbr|a'| \Big | |p(v)-\pt|^2-  |p(\bar v_b)-\pt|^2\Big| {\mathbf 1}_{\{p(v)-\pt \le \delta_3\}} \,d\xi \\
&\quad = \int_\bbr |a'| |p( v)-p(\bar v_b)| |p(v)+p(\bar v_b)-2p(\tilde v_\eps)|{\mathbf 1}_{\{p(v)-\pt \le \delta_3\}} \,d\xi\\
&\quad\leq \int_\bbr |a'| |p(v)-p(\bar v_b)| \left( |p(v)-p(\bar v_b)|+2|p(\bar v_b) -\pt|    \right){\mathbf 1}_{\{p(v)-\pt \le \delta_3\}}\,d\xi\\
&\quad \le  \int_\bbr |a'| \left( |p(v)-p(\bar v_b)|^2+2\delta_3 |p(v)-p(\bar v_b)| \right)\,d\xi \le C\sqrt{\frac{\eps}{\lambda}}\mathcal{D}(U).
\end{aligned}
\end{align}

\noindent{\bf Proof of \eqref{l3}:}  
Thanks to \eqref{y-identity}, it is enough to show the quadratic part, with $\bar v$ defined with $\delta_3/2$ instead of $\delta_3$.
For this case, we will keep the notations $\bar v_s$ and $\bar v_b$ below without confusion.\\
We first decompose the quadratic part into two parts:
\[
 \int_\bbr|a'|^2 v^\beta  |p(v)-p(\bar v)|^2\,d\xi =\underbrace{\int_\bbr|a'|^2 v^\beta  |p(v)-p(\bar v_b)|^2 \,d\xi}_{=:Q_b}+ \underbrace{\int_\bbr|a'|^2 v^\beta  |p(v)-p(\bar v_s)|^2 \, d\xi}_{=:Q_s}.
\]
First, using the condition $\beta=\gamma-\alpha\le1$, we have
\begin{align*}
\begin{aligned}
|Q_b| & = \int_\bbr|a'|^2 \frac{v^\beta  |p(v)-p(\bar v_b)|^2 }{|v-\bar v| }  |v-\bar v|  {\mathbf 1}_{\{p(v)-p(\tilde v_\eps)<-\delta_3/2\}} \,d\xi\\
& \le C \int_\bbr|a'|^2 |v-\bar v| {\mathbf 1}_{\{p(v)-p(\tilde v_\eps)<-\delta_3/2\}}  \,d\xi.
\end{aligned}
\end{align*} 
To control the right hand side, we use \eqref{rel_Q1} as follows: If $|v-\bar v|>0$, using \eqref{pressure2}, we find 
\[|\bar v-\tilde v_\eps|\ge \min(c_5^{-1} \delta_3/2, v_-/2-\eps_0).\]
Taking $\delta_*$ in $2)$ of Lemma \ref{lem-pro} such that $\eps_0\leq \delta_*/2$ and $\min(c_5^{-1} \delta_3, v_-/2-\eps_0)\ge \delta_*$, we use \eqref{rel_Q1} with $w=\tilde{v}_\eps$, $u=\bar v$ and $v=v$ to find that there exists a constant $C>0$ such that 
\beq\label{line_lower}     
C|v-\bar v|\leq Q(v|\tilde{v}_\eps)- Q(\bar v|\tilde{v}_\eps).
\eeq
Therefore, using $|a'|\le \eps\lambda$, we find
\beq\label{Qb}
|Q_b| \le  C \int_\bbr|a'|^2  \big(Q(v|\tilde{v}_\eps)-Q(\bar v|\tilde{v}_\eps)\big) \,d\xi \le C\eps\lambda \mathcal{G}_2(U).
\eeq

On the other hand, to control $Q_s$, we will first derive a point-wise estimate \eqref{beta1} as below:  \\
Using $|a'|=(\lambda/\eps)|\tilde{v}'_\eps|$, together with \eqref{lower-v} and \eqref{l1}, we get
\begin{eqnarray*}
2\eps\int_{-1/\eps}^{1/\eps} Q(v|\tilde{v}_\eps)\, d\xi&\leq& \frac{2\eps}{\inf_{[-1/\eps,1/\eps]}|a'|}\int_\bbr|a'|Q(v|\tilde{v}_\eps)\,d\xi\\
&\leq& C \frac{\eps}{\lambda\eps}\frac{\eps^2}{\lambda}=C\left(\frac{\eps}{\lambda}\right)^2.
\end{eqnarray*}
Therefore, there exists $\xi_0\in [-1/\eps,1/\eps]$ such that $Q(v(\xi_0),\tilde{v}_\eps(\xi_0))\leq C(\eps/\lambda)^2$. For $\deo$ small enough, and using  \eqref{pQ-equi0}, we have 
$$
|(p(v)-\pt)(\xi_0)|\leq C\frac{\eps}{\lambda}.
$$
Thus, if $\deo$ is small enough such that $C\eps/\lambda\leq \delta_3/2$, then we have from the definition of $\bar v_s$ that
$$
(p(v)-p(\bar v_s))(\xi_0)=0.
$$
Therefore,  for any $\xi\in \bbr$,
\[
|v^{\beta/2}(p(v)-p(\bar v_s))(\xi)|=\left|\int_{\xi_0}^\xi \partial_\xi \big(v^{\beta/2}(p(v)-p(\bar v_s))\big)\,d\zeta\right|.
\]
To control the right-hand side by the good terms, we observe that 
since $v^{\beta/2}=p(v)^{-(\gamma-\alpha)/2\gamma}$, 
we have
\begin{eqnarray*}
&&\partial_\xi \big(v^{\beta/2}(p(v)-p(\bar v_s))\big)= \partial_\xi \big( p(v)^{-(\gamma-\alpha)/2\gamma} ((p(v)-p(\bar v_s)) \big)\\
&&\qquad =  p(v)^{-(\gamma-\alpha)/2\gamma}\partial_\xi ((p(v)-p(\bar v_s))\\
&&\qquad\quad -\frac{\gamma-\alpha}{2\gamma}p(v)^{-(\gamma-\alpha)/2\gamma} \frac{p(v)-p(\bar v_s)}{p(v)}\partial_\xi \big[ ((p(v)-\pt)+\pt \big]\\
&&\qquad = v^{\beta/2} \partial_\xi ((p(v)-p(\bar v_s)) -\frac{\gamma-\alpha}{2\gamma}v^{\beta/2} \underbrace{ \frac{p(v)-p(\bar v_s)}{p(v)}  \partial_\xi ((p(v)-\pt)}_{=:K}\\
&&\qquad\quad -\frac{\gamma-\alpha}{2\gamma}v^{\beta/2}  \frac{p(v)-p(\bar v_s)}{p(v)} \partial_\xi p(\tilde v_\eps).
\end{eqnarray*}
In particular, note that (by the definition of $\bar v_s$) the part $K$ above can be rewritten by
\[
K= \frac{p(v)-p(\bar v_s)}{p(v)} {\mathbf 1}_{\{p(v)-p(\tilde v_\eps)>\delta_3/2\}}  \partial_\xi ((p(v)-\pt)= \frac{p(v)-p(\bar v_s)}{p(v)}  \partial_\xi ((p(v)-p(\bar v_s)).
\]
Then, using $ |\partial_\xi p(\tilde v_\eps)| \le C |\tilde v_\eps'| $ and
\[
 \frac{p(v)-p(\bar v_s)}{p(v)} \le C,
\]
we have
\[
|\partial_\xi \big(v^{\beta/2}(p(v)-p(\bar v_s))\big)|\le C v^{\beta/2}(|\partial_\xi ((p(v)-p(\bar v_s))| +|\tilde v_\eps'|).
\]
Therefore, using \eqref{eq_D}, we have that for any $\xi\in \bbr$,
\begin{align*}
\begin{aligned}
 |v^{\beta/2} (p(v)-p(\bar v_s))(\xi)|  &=\left|\int_{\xi_0}^\xi \partial_\xi \big(v^{\beta/2} (p(v)-p(\bar v_s))\big)\,d\xi \right| \\
&\le \int_{\xi_0}^\xi \big|\partial_\xi \big(v^{\beta/2} (p(v)-p(\bar v_s))\big)\big| {\mathbf 1}_{\{p(v)-\pt > \delta_3/2\}} d\xi\\
&\le \int_{\xi_0}^\xi  v^{\beta/2}(|\partial_\xi ((p(v)-p(\bar v_s))| +|\tilde v_\eps'| ) {\mathbf 1}_{\{p(v)-\pt > \delta_3/2\}} d\xi\\
 &\le C\sqrt{|\xi|+\frac{1}{\eps}}\bigg(\sqrt{\mathcal{D}(U)} + \sqrt{ \int_\bbr a |\tilde v_\eps'|^2 v^{\beta}{\mathbf 1}_{\{p(v)-\pt > \delta_3/2\}} d\xi} \bigg).
\end{aligned}
\end{align*}
Using the condition $\beta=\gamma-\alpha>0$, we have
\begin{align}
\begin{aligned}\label{same-1}
 \int_\bbr a |\tilde v_\eps'|^2 v^{\beta}{\mathbf 1}_{\{p(v)-\pt > \delta_3/2\}} d\xi  &= \int_\bbr a |\tilde v_\eps'|^2 \frac{v^{\beta}}{|v-\tilde v_\eps|^2}  |v-\tilde v_\eps|^2 {\mathbf 1}_{\{p(v)-\pt > \delta_3/2\}} d\xi\\
&\le  C \int_\bbr a |\tilde v_\eps'|^2  |v-\tilde v_\eps|^2 {\mathbf 1}_{\{p(v)-\pt > \delta_3/2\}} d\xi.
\end{aligned}
\end{align}
In addition, using \eqref{rel_Q} and $|\tilde v_\eps'|\le C\frac{\eps}{\lambda} |a'|$, we have
\beq\label{same-s}
 \int_\bbr a |\tilde v_\eps'|^2  |v-\tilde v_\eps|^2{\mathbf 1}_{\{p(v)-\pt > \delta_3/2\}} d\xi\le C\frac{\eps^2}{\lambda^2}   \int_\bbr |a'| Q(v|\tilde v_\eps) d\xi.
\eeq
Therefore we obtain that  
\beq\label{beta1}
\forall \xi\in \bbr,\quad  |v^{\beta/2} (p(v)-p(\bar v_s))(\xi)| \le C\sqrt{|\xi|+\frac{1}{\eps}}\bigg(\sqrt{\mathcal{D}(U)} +\frac{\eps}{\lambda} \sqrt{\mathcal{G}_2(U)} \bigg).
\eeq

Now, using \eqref{d-weight} with \eqref{tail}, we have
\begin{align*}
\begin{aligned}
Q_s& \le C\bigg(\mathcal{D}(U) +\left(\frac{\eps}{\lambda}\right)^2 \mathcal{G}_2(U) \bigg) \lambda^2\eps  \int_\bbr e^{-C|\eps\xi|} (|\eps \xi|+1) d\xi \\
&\le C\lambda^2 \bigg(\mathcal{D}(U) +\left(\frac{\eps}{\lambda}\right)^2 \mathcal{G}_2(U) \bigg).
\end{aligned}
\end{align*}
Therefore, this and \eqref{Qb} complete the estimate:
\[
\int_\bbr|a'|^2 v^\beta |p(v)-p(\bar v)|^2 \,d\xi \leq C\lambda^2 \left( \mathcal{D}(U) +\frac{\eps}{\lambda} \mathcal{G}_2(U)\right).
\]
Hence using the similar estimates as in \eqref{shift1} and \eqref{shift2} (i.e., using \eqref{def-bar} and \eqref{y-identity} with $y:=|p(v)-p(\tilde v_\eps)|$), we obtain \eqref{l3}.

\noindent{\bf Proof of \eqref{l4}:}  
We first separate it into two parts:
\begin{eqnarray*}
&&\int_\bbr |a'|^2 \left| v^\beta |p(v)-\pt|^2 - \bar v^\beta |p(\bar v) -\pt|^2  \right|\,d\xi\\
&&\quad\le \underbrace{\int_\bbr |a'|^2 v^\beta \left| |p(v)-\pt|^2 - |p(\bar v) -\pt|^2  \right| \,d\xi }_{=:I_1} +\underbrace{\int_\bbr |a'|^2 \big|v^\beta -\bar v^\beta\big| |p(\bar v) -\pt|^2 \,d\xi}_{=:I_2}.
\end{eqnarray*}
Using the same arguments as in \eqref{dif1}, it follows from \eqref{l3} that
\begin{eqnarray*}
I_1&\leq&  \int_\bbr |a'|^2 v^\beta \left( |p(v)-p(\bar v)|^2+2\delta_3 |p(v)-p(\bar v)| \right)\,d\xi\\
\nonumber&\leq& C\lambda^2 \left( \mathcal{D}(U) +\frac{\eps}{\lambda} \mathcal{G}_2(U)\right).
\end{eqnarray*}
For $I_2$, we first separate $I_2$ into two parts:
\[
I_2=\underbrace{\int_\bbr |a'|^2 \big|v^\beta -\bar v^\beta_b\big| |p(\bar v_b) -\pt|^2 \,d\xi}_{=:I_2^b}+\underbrace{\int_\bbr |a'|^2 \big|v^\beta -\bar v^\beta_s\big| |p(\bar v_s) -\pt|^2 \,d\xi}_{=:I_2^s}.
\]
Using the assumption $\beta\le 1$, we have
\begin{align*}
\begin{aligned}
I_2^b&\le \delta_3^2 \int_\bbr |a'|^2 \frac{\big|v^\beta -\bar v^\beta_b\big|}{|v-\bar v|} |v-\bar v|{\mathbf 1}_{\{p(v)-p(\tilde v_\eps)<-\delta_3\}} \,d\xi\\
&\le C\int_\bbr |a'|^2  |v-\bar v|{\mathbf 1}_{\{p(v)-p(\tilde v_\eps)<-\delta_3\}} \,d\xi\\
&\le C \int_\bbr|a'|^2  \big(Q(v|\tilde{v}_\eps)-Q(\bar v|\tilde{v}_\eps)\big) \,d\xi \le C\eps\lambda \mathcal{G}_2(U).
\end{aligned}
\end{align*}
Likewise, we use $\beta>0$ to have
\begin{align*}
\begin{aligned}
I_2^s &\le \delta_3^2  \int_\bbr |a'|^2 \frac{\big|v^\beta -\bar v^\beta_s\big|} { |v-\bar v|}  |v-\bar v| {\mathbf 1}_{\{p(v)-p(\tilde v_\eps)>\delta_3\}}  \,d\xi \\
&\le C \int_\bbr |a'|^2  |v-\bar v| {\mathbf 1}_{\{p(v)-p(\tilde v_\eps)>\delta_3\}} \,d\xi  \le C\eps\lambda \mathcal{G}_2(U).
\end{aligned}
\end{align*}

\end{proof}

\begin{lemma}\label{lemma_out2}
Under the same assumption as Proposition \ref{prop_out}, we have
\begin{eqnarray}
\label{l5}
&&  \int_\bbr|a'| \left|p(v|\tilde{v}_\eps)-p(\bar v|\tilde{v}_\eps)\right | \,d\xi\\
\nonumber
&&\quad \leq C\sqrt{\frac{\eps}{\lambda}} \left(\mathcal{D}(U) +\left(\frac{\eps}{\lambda}\right)^2\mathcal{G}_2(\bar U) \right) +C\left(\mathcal{G}_2(U)-\mathcal{G}_2(\bar U) \right),\\
\label{l50}
&&   \int_\bbr|a'| \left|Q(v|\tilde{v}_\eps)-Q(\bar v|\tilde{v}_\eps)\right | \,d\xi+\int_\bbr|a'| |v-\bar v| \,d\xi \leq  C\left(\mathcal{G}_2(U)-\mathcal{G}_2(\bar U) \right).
\end{eqnarray}
\end{lemma}
\begin{proof}
Following the proof of \cite[Lemma 3.3]{Kang-V-NS17} together with  \eqref{l2}, we have
\begin{eqnarray*}
&&\int_\bbr|a'| \left|Q(v|\tilde{v}_\eps)-Q(\bar v|\tilde{v}_\eps)\right | \,d\xi+\int_\bbr|a'| |v-\bar v| \,d\xi \\
&& \qquad\le C \int_\bbr|a'| \left(Q(v|\tilde{v}_\eps)-Q(\bar v|\tilde{v}_\eps)\right) \,d\xi \le C\left(\mathcal{G}_2(U)-\mathcal{G}_2(\bar U) \right).
\end{eqnarray*}
Following the proof of \cite[Lemma 3.3]{Kang-V-NS17}, we have
\begin{eqnarray*}
&&\int_\bbr|a'| \left|p(v|\tilde{v}_\eps)-p(\bar v|\tilde{v}_\eps)\right | \,d\xi\\
&& \qquad\le \underbrace{\int_\bbr|a'| \left|p(v)-p(\bar v) \right | \,d\xi}_{=:I_1} + C \underbrace{\int_\bbr|a'| \left(Q(v|\tilde{v}_\eps)-Q(\bar v|\tilde{v}_\eps)\right) \,d\xi}_{=:I_2}.
\end{eqnarray*}
First, using \eqref{l2}, we have 
\[
I_2\le C\left(\mathcal{G}_2(U)-\mathcal{G}_2(\bar U) \right).
\]
We separate $I_1$ into three parts:
\[
I_1 =\underbrace{ \int_\bbr|a'| \left|p(v)-p(\bar v_b) \right | \,d\xi}_{=:I_{11}} +\underbrace{\int_\bbr|a'| \left|p(v)-p(\bar v_s) \right | {\mathbf 1}_{\{v<v_-/2\}}\,d\xi}_{=:I_{12}}+\underbrace{\int_\bbr|a'| \left|p(v)-p(\bar v_s) \right | {\mathbf 1}_{\{v\ge v_-/2\}}\,d\xi}_{=:I_{13}}.
\]
By \eqref{big1}, we have
\[
I_{11}\le C\sqrt{\frac{\eps}{\lambda}} \mathcal{D}(U).
\]
For $I_{12}$, we first observe that
\begin{align}
\begin{aligned}\label{pw1}
v^\beta \big| p(v)-p(\bar v_s) \big|^2 &= p(v)^{-\frac{\gamma-\alpha}{\gamma}} |p(v)-p(\bar v_s)|^2  {\mathbf 1}_{\{p(v)-p(\tilde v_\eps)>\delta_3\}}\\
&= \Big(\frac{|p(v)-p(\bar v_s)|}{p(v)} \Big)^{\frac{\gamma-\alpha}{\gamma}} {\mathbf 1}_{\{p(v)-p(\tilde v_\eps)>\delta_3\}}  |p(v)-p(\bar v_s)|^{\frac{\gamma+\alpha}{\gamma}}\\
&\ge C\delta_3^{\frac{\gamma-\alpha}{\gamma}}  |p(v)-p(\bar v_s)|^{\frac{\gamma+\alpha}{\gamma}}.
\end{aligned}
\end{align}
Since (by the smallness of $\eps_0$ and $\delta_3$)
\begin{align}
\begin{aligned}\label{lower-p}
\left|p(v)-p(\bar v_s) \right | {\mathbf 1}_{\{v<v_-/2\}} &\ge \left|p(v)-p(\tilde v_\eps) \right | {\mathbf 1}_{\{v<v_-/2\}} -\left|p(\bar v_s) -\pt \right| {\mathbf 1}_{\{v<v_-/2\}} \\
&\ge \left|p(v)-p(\tilde v_\eps) \right | {\mathbf 1}_{\{v<v_-/2\}} -\delta_3\\
&\ge \left|p(v_-/2)-p(3v_-/4) \right |  -\delta_3\\
&\ge \frac{1}{2}\left|p(v_-/2)-p(3v_-/4) \right |,
\end{aligned}
\end{align}
using \eqref{pw1}, we have
\[
\left|p(v)-p(\bar v_s) \right | {\mathbf 1}_{\{v<v_-/2\}} \le  C |p(v)-p(\bar v_s)|^{\frac{\gamma+\alpha}{\gamma}} \le C v^\beta \big| p(v)-p(\bar v_s) \big|^2.
\]
Then it follows from \eqref{beta1} that for all $\xi\in\bbr$,
\beq\label{pwi}
\left|p(v)-p(\bar v_s) \right | {\mathbf 1}_{\{v<v_-/2\}}(\xi) \le C \left(|\xi|+\frac{1}{\eps}\right) \left(\mathcal{D}(U) +\left(\frac{\eps}{\lambda}\right)^2\mathcal{G}_2(U) \right).
\eeq
Therefore, using \eqref{pwi} together with the same estimate as in \eqref{stand1}, we have
\begin{align}
\begin{aligned}\label{long-est}
I_{12}&\le \int_{|\xi|\le\frac{1}{\eps}\sqrt{\frac{\lambda}{\eps}}}|a'||p(v)-p(\bar v_s)| {\mathbf 1}_{\{v<v_-/2\}} \,d\xi+ \int_{|\xi|\geq\frac{1}{\eps}\sqrt{\frac{\lambda}{\eps}}}|a'||p(v)-p(\bar v_s)| {\mathbf 1}_{\{v<v_-/2\}} \,d\xi\\
& \leq C\left(\mathcal{D}(U) +\left(\frac{\eps}{\lambda}\right)^2\mathcal{G}_2(U) \right) \left(\sqrt{\frac{\lambda}{\eps^3}}
\int_\bbr|a'| Q(v|\tilde{v}_\eps) \,d\xi+2\int_{|\xi|\geq\frac{1}{\eps}\sqrt{\frac{\lambda}{\eps}}}|a'| |\xi|\,d\xi\right)\\
& \leq C\left(\mathcal{D}(U) +\left(\frac{\eps}{\lambda}\right)^2\mathcal{G}_2(U) \right) \left(\sqrt{\frac{\eps}{\lambda}} +\lambda \eps
\int_{|\xi|\geq\frac{1}{\eps}\sqrt{\frac{\lambda}{\eps}}}e^{-c\eps|\xi|} |\xi|\,d\xi \right)\\
&\leq C\sqrt{\frac{\eps}{\lambda}} \left(\mathcal{D}(U) +\left(\frac{\eps}{\lambda}\right)^2\mathcal{G}_2(U) \right).
\end{aligned}
\end{align}

For $I_{13}$, since $\left|p(v)-p(\bar v_s) \right|{\mathbf 1}_{\{v\ge v_-/2\}} \le p(v_-/2)$,  we have
\begin{align*}
\begin{aligned}
I_{13} &= \int_\bbr|a'| \frac{\left|p(v)-p(\bar v_s) \right |}{|v-\bar v|} |v-\bar v| {\mathbf 1}_{\{v\ge v_-/2\} \cap\{p(v)-\pt >\delta_3\} }\,d\xi \\
&\le C \int_\bbr|a'|  |v-\bar v| {\mathbf 1}_{\{v\ge v_-/2\} \cap\{p(v)-\pt >\delta_3\} }\,d\xi  \le C\left(\mathcal{G}_2(U)-\mathcal{G}_2(\bar U) \right).
\end{aligned}
\end{align*}

Hence, we have
\[
I_1\le C\sqrt{\frac{\eps}{\lambda}} \left(\mathcal{D}(U) +\left(\frac{\eps}{\lambda}\right)^2\mathcal{G}_2(U) \right) +C\left(\mathcal{G}_2(U)-\mathcal{G}_2(\bar U) \right),
\]
which gives \eqref{l5}.
\end{proof}

\begin{lemma}\label{lemma_out3}
Under the same assumption as Proposition \ref{prop_out}, we have
\begin{eqnarray}
\label{ns1}
&& \int_{\Omega^c} |a'| \big |p(v)-\pt\big| |h-\tilde h_\eps| d\xi  \le \delta_0 \left(\mathcal{D}(U) +\left(\frac{\eps}{\lambda}\right)^2\mathcal{G}_2(U) \right) + \frac{1}{2} \mathcal{G}_1^-(U),\\
\label{ns2}
&&\int_{\Omega^c} |a'| \left( Q(\bar v|\tilde v_\eps) +|\bar v -\tilde v_\eps | \right)  d\xi \le C \sqrt{\frac{\eps}{\lambda}} \left(\mathcal{D}(U) +\left(\frac{\eps}{\lambda}\right)^2\mathcal{G}_2(U) \right).
\end{eqnarray}
\end{lemma}
\begin{proof}
\noindent{\bf Proof of \eqref{ns1}:}  
We first separate it into two parts:
\begin{align*}
\begin{aligned}
&\int_{\Omega^c} |a'| \big| p(v)-\pt \big| |h-\tilde h_\eps| \,d\xi\\
&\quad\le \underbrace{\int_{\Omega^c} |a'| \big| p(v)-p(\bar v) \big| |h-\tilde h_\eps| \,d\xi}_{=:J_1}  + \underbrace{\int_{\Omega^c} |a'| \big| p(\bar v)-\pt \big| |h-\tilde h_\eps| \,d\xi}_{=:J_2}
\end{aligned}
\end{align*}
We use the definition of $\bar v_s$ and H\"older's inequality to have
\begin{align*}
\begin{aligned}
J_1&=\int_{\Omega^c} |a'| \big| p(v)-p(\bar v_s) \big| |h-\tilde h_\eps| \,d\xi\\
&\le  \left(\int_{\Omega^c}|a'| \big| p(v)-p(\bar v_s) \big|^2 d\xi \right)^{1/2}  \left(\int_{\Omega^c}|a'| \big| h -\tilde h_\eps \big|^2 d\xi \right)^{1/2}.
\end{aligned}
\end{align*}
To estimate $\int_\bbr|a'| \big| p(v)-p(\bar v_s) \big|^2 d\xi$, using \eqref{pw1} and \eqref{beta1}, we find that for any $\xi\in\bbr$,
\begin{align}
\begin{aligned}\label{est5}
 |(p(v)-p(\bar v_s))(\xi)|^2 \le  C \left(|\xi|+\frac{1}{\eps}\right)^{\frac{2\gamma}{\gamma+\alpha}} \left(\mathcal{D}(U) +\left(\frac{\eps}{\lambda}\right)^2\mathcal{G}_2(U) \right) ^{\frac{2\gamma}{\gamma+\alpha}}.
\end{aligned}
\end{align}
Following the similar arguments as in \eqref{stand1}, and using \eqref{est5} with $q:=\frac{2\gamma}{\gamma+\alpha}$ (note $1<q<2$ by $0<\alpha<\gamma$), we obtain
\begin{align*}
\begin{aligned}
&\int_\bbr|a'| \big| p(v)-p(\bar v_s) \big|^2 d\xi\\
&\quad\le \int_{|\xi|\le\frac{1}{\eps}\left(\frac{\lambda}{\eps}\right)^{1/q}}|a'||p(v)-p(\bar v_s)|^2  \,d\xi+ \int_{|\xi|\geq\frac{1}{\eps}\left(\frac{\lambda}{\eps}\right)^{1/q}}|a'||p(v)-p(\bar v_s)|^2 \,d\xi\\
& \quad\leq C\left(\mathcal{D}(U) +\left(\frac{\eps}{\lambda}\right)^2\mathcal{G}_2(U) \right) ^q \left(\frac{1}{\eps^q}
\int_\bbr|a'| Q(v|\tilde{v}_\eps) \,d\xi+2\int_{|\xi|\geq\left(\frac{\lambda}{\eps}\right)^{1/q}}|a'| |\xi|^q\,d\xi\right)\\
&\quad\leq C\left(\mathcal{D}(U) +\left(\frac{\eps}{\lambda}\right)^2\mathcal{G}_2(U) \right)^q \left(\frac{1}{\eps^q} \frac{\eps^2}{\lambda} +\frac{\lambda}{\eps^q} \int_{|\xi|\geq\left(\frac{\lambda}{\eps}\right)^{1/q}} |\xi|^q e^{-c|\xi|}\,d\xi\right)\\
&\quad\leq C\left(\mathcal{D}(U) +\left(\frac{\eps}{\lambda}\right)^2\mathcal{G}_2(U) \right)^q   \frac{\eps^{2-q}}{\lambda}.
\end{aligned}
\end{align*}
Therefore, 
\[
J_1 \le C \sqrt{ \frac{\eps^{2-q}}{\lambda}} \left(\mathcal{D}(U) +\left(\frac{\eps}{\lambda}\right)^2\mathcal{G}_2(U) \right)^{q/2}  \left(\int_{\Omega^c}|a'| \big| h -\tilde h_\eps \big|^2 d\xi \right)^{1/2}.
\]
Using the Young's inequality (recall $1<q<2$), we have
\[
J_1 \le \delta_0 \left(\mathcal{D}(U) +\left(\frac{\eps}{\lambda}\right)^2\mathcal{G}_2(U) \right) +  \frac{C}{\delta_0}\left( \frac{\eps^{2-q}}{\lambda}\right)^{\frac{1}{2-q}} \left(\int_{\Omega^c}|a'| \big| h -\tilde h_\eps \big|^2 d\xi \right)^{\frac{1}{2-q}}.
\]
Since \eqref{l1} yields
\begin{align*}
\begin{aligned}
 \left( \frac{\eps^{2-q}}{\lambda}\right)^{\frac{1}{2-q}} \left(\int_{\Omega^c} |a'| \big| h -\tilde h_\eps \big|^2 d\xi \right)^{\frac{1}{2-q}} &\le C \left( \frac{\eps^{2-q}}{\lambda}\right)^{\frac{1}{2-q}}
\left( \frac{\eps^2}{\lambda} \right)^{\frac{q-1}{2-q}}
 \int_{\Omega^c}|a'| \big| h -\tilde h_\eps \big|^2 d\xi\\
&=C\left( \frac{\eps}{\lambda} \right)^{\frac{2}{2-q}}\int_{\Omega^c}|a'| \big| h -\tilde h_\eps \big|^2 d\xi \\
&\le C\left( \frac{\eps}{\lambda} \right)^2 \int_{\Omega^c}|a'| \big| h -\tilde h_\eps \big|^2 d\xi,
\end{aligned}
\end{align*}
we have
\[
J_1 \le \delta_0 \left(\mathcal{D}(U) +\left(\frac{\eps}{\lambda}\right)^2\mathcal{G}_2(U) \right) + \frac{C}{\delta_0}\left( \frac{\eps}{\lambda} \right)^2 \mathcal{G}_1^-(U).
\]

For $J_2$, we use $\big| p(\bar v)-\pt \big|\le\delta_3$ and Young's inequality to have
\[
J_2 \le \delta_3\int_{\Omega^c} |a'| |h-\tilde h_\eps|  \,d\xi \le \frac{1}{2}\mathcal{G}_1^-(U) +C\underbrace{\int_\bbr |a'| {\mathbf 1}_{\{p(v)-\pt >\delta_3\}}  \,d\xi}_{=:J_{21}}.
\]
To control $J_{21}$, we observe that since $(y-\delta_3/2)_+\geq \delta_3/2$ whenever $(y-\delta_3)_+>0$, we have
\beq\label{y-out}
|p(v)-p(\bar v_{\delta_3/2})|=  (|p(v)-\pt|-\delta_3/2)_+ \ge \frac{\delta_3}{2} {\mathbf 1}_{\{p(v)-\pt >\delta_3\}}.
\eeq
Then, using \eqref{y-out} and \eqref{est5} (with $q:=\frac{2\gamma}{\gamma+\alpha}$) and following the same estimates as in \eqref{stand1}, we have
\begin{align*}
\begin{aligned}
J_{21} &\le C\int_\bbr |a'| |p(v)-p(\bar v_{\delta_3/2})|^{2/q}  {\mathbf 1}_{\{p(v)-\pt >\delta_3\}}  \,d\xi\\
&\leq C\sqrt\frac{\eps}{\lambda}\left(\mathcal{D}(U) +\left(\frac{\eps}{\lambda}\right)^2\mathcal{G}_2(U) \right).
\end{aligned}
\end{align*}
Therefore,
\[
J_2\le \frac{1}{2}\mathcal{G}_1^-(U) +C\sqrt\frac{\eps}{\lambda}\left(\mathcal{D}(U) +\left(\frac{\eps}{\lambda}\right)^2\mathcal{G}_2(U) \right).
\]
Hence we obtain \eqref{ns1}.\\

\noindent{\bf Proof of \eqref{ns2}:} The proof follows from the above estimate for $J_{12}$ as follows:  
\[
\int_{\Omega^c} |a'| \left( Q(\bar v|\tilde v_\eps) +|\bar v -\tilde v_\eps | \right)  d\xi  \le C\int_\bbr |a'| {\mathbf 1}_{\{p(v)-\pt >\delta_3\}}  \,d\xi \le C\sqrt\frac{\eps}{\lambda}\left(\mathcal{D}(U) +\left(\frac{\eps}{\lambda}\right)^2\mathcal{G}_2(U) \right).
\]
\end{proof}

\begin{lemma}\label{lemma_out4}
Under the same assumption as Proposition \ref{prop_out}, we have
\begin{eqnarray}
\label{l7}
&&  \int_\bbr|a'|^2  \frac{|v^\beta-\bar v^\beta|^2}{v^\beta} \,d\xi\le C\lambda  \bigg(\mathcal{D}(U) +\left(\mathcal{G}_2(U)-\mathcal{G}_2(\bar U)\right)+\left(\frac{\eps}{\lambda}\right)^2 \mathcal{G}_2(\bar U) \bigg)  ,\\
\label{l8}
&&  \int_\bbr|a'|^2  \Big| \frac{|v^\beta-\tilde v^\beta_\eps|^2}{v^\beta} - \frac{|\bar v^\beta-\tilde v^\beta_\eps|^2}{\bar v^\beta} \Big| \,d\xi \\
\nonumber
&&\le C\lambda  \bigg(\mathcal{D}(U) +\left(\mathcal{G}_2(U)-\mathcal{G}_2(\bar U)\right)+\left(\frac{\eps}{\lambda}\right)^2 \mathcal{G}_2(\bar U) \bigg) .
\end{eqnarray}
\end{lemma}
\begin{proof}
\noindent{\bf Proof of \eqref{l7}:}  
We first have
\[
\int_\bbr|a'|^2  \frac{|v^\beta-\bar v^\beta|^2}{v^\beta} \,d\xi =\underbrace{\int_\bbr|a'|^2  \frac{|v^\beta-\bar v_b^\beta|^2}{v^\beta} \,d\xi}_{=:I_b} +\underbrace{\int_\bbr|a'|^2  \frac{|v^\beta-\bar v_s^\beta|^2}{v^\beta} \,d\xi}_{=:I_s}.
\]
Since $0<\beta\le 1$, we have
\begin{align*}
\begin{aligned}
I_b&\le C\int_\bbr|a'|^2 |v^\beta-\bar v_b^\beta| {\mathbf 1}_{\{p(v)-p(\tilde v_\eps)<-\delta_3\}} \,d\xi = C\int_\bbr|a'|^2 \frac{|v^\beta-\bar v_b^\beta|} {|v-\bar v| } |v-\bar v|  {\mathbf 1}_{\{p(v)-p(\tilde v_\eps)<-\delta_3\}} \,d\xi\\
&\le C\int_\bbr|a'|^2 |v-\bar v| {\mathbf 1}_{\{p(v)-p(\tilde v_\eps)<-\delta_3\}} \,d\xi \le C\int_\bbr|a'|^2 \left(Q(v|\tilde v_\eps) -Q(\bar v|\tilde v_\eps)\right) \,d\xi\\
&\le C\eps\lambda\left(\mathcal{G}_2(U)-\mathcal{G}_2(\bar U)\right).
\end{aligned}
\end{align*}

For $I_s$, we separate it into two cases of $\alpha\ge 1$ and $\alpha<1$.\\
\noindent{\it Case of  $\alpha\ge 1$ :}
Since $\beta\le\gamma-1$ by  $\alpha\ge 1$, we observe
\[
v^{-\beta}\le v^{-\gamma+1}=Q(v)\quad\mbox{as }~v\to 0,
\]
we have 
\begin{align*}
\begin{aligned}
I_s&= \int_\bbr|a'|^2 \frac{v^{-\beta} |v^\beta-\bar v_s^\beta|^2}{Q(v|\tilde v_\eps) -Q(\bar v|\tilde v_\eps)} {\mathbf 1}_{\{p(v)-p(\tilde v_\eps)>\delta_3\}}\left(Q(v|\tilde v_\eps) -Q(\bar v|\tilde v_\eps)\right) \,d\xi\\ 
&\le C \int_\bbr|a'|^2 \left(Q(v|\tilde v_\eps) -Q(\bar v|\tilde v_\eps)\right) \,d\xi \le C\eps\lambda\left(\mathcal{G}_2(U)-\mathcal{G}_2(\bar U)\right).
\end{aligned}
\end{align*}

\noindent{\it Case of  $0<\alpha< 1$ :}
Since $\frac{1-\alpha}{\gamma}<1$ and
\[
v^\beta p(v)^{\frac{1-\alpha}{\gamma}}Q(v)=v^\beta(v^{-\gamma})^{\frac{1-\alpha}{\gamma}}v^{-\gamma+1}=v^0=1,
\]
using \eqref{line_lower}, we have
\begin{align*}
\begin{aligned}
I_s&= \int_\bbr|a'|^2 \frac{ |v^\beta-\bar v_s^\beta|^2 {\mathbf 1}_{\{p(v)-p(\tilde v_\eps)>\delta_3\}}}{v^\beta |p(v)-p(\bar v_s)|^{\frac{1-\alpha}{\gamma}}\left(Q(v|\tilde v_\eps) -Q(\bar v|\tilde v_\eps)\right) }  |p(v)-p(\bar v_s)|^{\frac{1-\alpha}{\gamma}}\left(Q(v|\tilde v_\eps) -Q(\bar v|\tilde v_\eps)\right)  d\xi\\ 
&\le C \int_\bbr|a'|^2 |p(v)-p(\bar v_s)|^{\frac{1-\alpha}{\gamma}}\left(Q(v|\tilde v_\eps) -Q(\bar v|\tilde v_\eps)\right) \,d\xi.
\end{aligned}
\end{align*}
Then, using the fact from \eqref{pw1} and \eqref{beta1} that
\[
 |p(v)-p(\bar v_s)|^{\frac{\gamma+\alpha}{\gamma}}\le C \left(|\xi|+\frac{1}{\eps}\right) \bigg(\mathcal{D}(U) +\left(\frac{\eps}{\lambda}\right)^2 \mathcal{G}_2(U) \bigg) ,
\]
we have
\begin{align*}
\begin{aligned}
I_s&\le C\bigg(\mathcal{D}(U) +\left(\frac{\eps}{\lambda}\right)^2 \mathcal{G}_2(U) \bigg)^{\frac{1-\alpha}{\gamma+\alpha}} \int_\bbr|a'|^2 \left(|\xi|+\frac{1}{\eps}\right)^{\frac{1-\alpha}{\gamma+\alpha}}\left(Q(v|\tilde v_\eps) -Q(\bar v|\tilde v_\eps)\right)  d\xi.
\end{aligned}
\end{align*}
Notice that since $0<\frac{1-\alpha}{\gamma+\alpha}<1$ by $0<\alpha< 1$, we have
\begin{align*}
\begin{aligned}
|a'| \left(|\xi|+\frac{1}{\eps}\right)^{\frac{1-\alpha}{\gamma+\alpha}} &\le C \eps\lambda e^{-c\eps|\xi|} \left(|\xi|+\frac{1}{\eps}\right)^{\frac{1-\alpha}{\gamma+\alpha}}\\
&\le  C \eps^{1-\frac{1-\alpha}{\gamma+\alpha}} \lambda e^{-c\eps|\xi|} \left(|\eps \xi|+1\right)^{\frac{1-\alpha}{\gamma+\alpha}} \le C\lambda.
\end{aligned}
\end{align*}
Thus, we have
\[
I_s\le  C\lambda \bigg(\mathcal{D}(U) +\left(\frac{\eps}{\lambda}\right)^2 \mathcal{G}_2(U) \bigg)^{\frac{1-\alpha}{\gamma+\alpha}} \left(\mathcal{G}_2(U)-\mathcal{G}_2(\bar U)\right).
\]
Now, using the Young's inequality with $\frac{1}{p}:=\frac{1-\alpha}{\gamma+\alpha}$ and $\frac{1}{p}+\frac{1}{p'}=1$, and then \eqref{l1},
we have
\begin{align*}
\begin{aligned}
I_s&\le C\lambda  \bigg(\mathcal{D}(U) +\left(\frac{\eps}{\lambda}\right)^2 \mathcal{G}_2(U) \bigg) + C\lambda \left(\mathcal{G}_2(U)-\mathcal{G}_2(\bar U)\right)^{p'}\\
&\le C\lambda  \bigg(\mathcal{D}(U) +\left(\frac{\eps}{\lambda}\right)^2 \mathcal{G}_2(U) \bigg) + C\lambda \left(\mathcal{G}_2(U)-\mathcal{G}_2(\bar U)\right).
\end{aligned}
\end{align*}
Hence we complete the proof.

\noindent{\bf Proof of \eqref{l8}:}  
Since $C^{-1}\le \bar v^\beta \le C$ and $|\bar v^\beta-\tilde v^\beta_\eps|\le C$, we have
\begin{align*}
\begin{aligned}
& \int_\bbr|a'|^2  \Big| \frac{|v^\beta-\tilde v^\beta_\eps|^2}{v^\beta} - \frac{|\bar v^\beta-\tilde v^\beta_\eps|^2}{\bar v^\beta} \Big| \,d\xi \\
&\quad \le  \int_\bbr|a'|^2 \left(\frac{1}{v^\beta} \left| |v^\beta-\tilde v^\beta_\eps|^2 -|\bar v^\beta-\tilde v^\beta_\eps|^2 \right| + \frac{|v^\beta-\bar v^\beta|}{v^\beta \bar v^\beta} |\bar v^\beta-\tilde v^\beta_\eps|^2 \right) \,d\xi \\
&\quad \le C \int_\bbr|a'|^2 \left(\frac{1}{v^\beta}  |v^\beta-\bar v^\beta|\left( |v^\beta-\bar v^\beta|+2|\bar v^\beta-\tilde v^\beta_\eps| \right) + \frac{|v^\beta-\bar v^\beta|}{v^\beta} \right) \,d\xi \\
&\quad \le C \int_\bbr|a'|^2 \frac{1}{v^\beta}  |v^\beta-\bar v^\beta|^2  \,d\xi + C \underbrace{\int_\bbr|a'|^2 \frac{1}{v^\beta}  |v^\beta-\bar v^\beta| \,d\xi}_{=:J}.
\end{aligned}
\end{align*}
By \eqref{l7}, it remains to estimate the term $J$. For that, we separate it into two parts:
\[
J=\underbrace{\int_\bbr|a'|^2 \frac{1}{v^\beta} |v^\beta-\bar v^\beta|  {\mathbf 1}_{\{v< v_-/2 \}\cup\{v> 2v_-\} } \,d\xi }_{=:J_1} + \underbrace{\int_\bbr|a'|^2 \frac{1}{v^\beta} |v^\beta-\bar v^\beta| {\mathbf 1}_{\{v_-/2\le v\le 2v_- \}} \,d\xi}_{=:J_2}.
\]
Using the same argument as \eqref{lower-p} together with the definition of $\bar v$, we have
\[
 |v^\beta-\bar v^\beta|  {\mathbf 1}_{\{v< v_-/2 \}\cup\{v> 2v_-\} } >C>0,
\]
which yields
\[
J_1\le C \int_\bbr|a'|^2 \frac{1}{v^\beta} |v^\beta-\bar v^\beta|^2  \,d\xi.
\]
Since
\[
|v-\bar v|{\mathbf 1}_{\{v_-/2\le v\le 2v_- \}} \le |p'(v_-/2)| |p(v)-p(\bar v)|,
\]
we have
\begin{align*}
\begin{aligned}
J_2 &=\int_\bbr|a'|^2 \frac{1}{v^\beta} \frac{|v^\beta-\bar v^\beta|}{|v-\bar v|} |v-\bar v| {\mathbf 1}_{\{v_-/2\le v\le 2v_- \}} \,d\xi \\
&\le C \int_\bbr|a'|^2 |v-\bar v| {\mathbf 1}_{\{v_-/2\le v\le 2v_- \}} \,d\xi \\
&\le C \int_\bbr|a'|^2 v^\beta  |p(v)-p(\bar v)| {\mathbf 1}_{\{v_-/2\le v\le 2v_- \}} \,d\xi.
\end{aligned}
\end{align*}
Therefore, \eqref{l3} and \eqref{l7} give the desired result.
\end{proof}

\vspace{0.5cm}
 
\subsubsection{Proof of Proposition \ref{prop_out}}
\noindent{\bf Proof of \eqref{n1}:} 
It follows from \eqref{l5} together with $|\tilde v_\eps'|\le C\frac{\eps}{\lambda}|a'|$ that
 \[
|\mathcal{B}_1(U)-\mathcal{B}_1(\bar U)|\leq C\frac{\eps}{\lambda} \left(\mathcal{D}(U) +\left(\frac{\eps}{\lambda}\right)^2\mathcal{G}_2(\bar U)+ \left(\mathcal{G}_2(U)-\mathcal{G}_2(\bar U) \right) \right) .
\] 

\noindent{\bf Proof of \eqref{n12}:} 
By \eqref{ns1}, we have
\[
|\mathcal{B}_2^-(U)| \le \delta_0 \left(\mathcal{D}(U) +\left(\frac{\eps}{\lambda}\right)^2\mathcal{G}_2(U) \right) + \frac{1}{2} \mathcal{G}_1^-(U).
\]

\noindent{\bf Proof of \eqref{n13}:} 
We use \eqref{compare1} and \eqref{big2} to have
\[
|\mathcal{B}_2^+(U)-\mathcal{B}_2^+(\bar U)| =\left| \int_\bbr a' \left(| p(v)-\pt |^2 - |p(\bar v_b)-\pt| \right){\mathbf 1}_{\{p(v)-\pt \le \delta_3\}} d\xi \right|  \le \sqrt\frac{\eps}{\lambda} \mathcal{D}(U).
\]

\noindent{\bf Proof of \eqref{n14}:} 
Using Young's inequality together with $\frac{1}{2}\le a\le 1$, we first find
\begin{align*}
\begin{aligned}
&|\mathcal{B}_3(U)|\le \delta_0 \mathcal{D}(U) + \underbrace{ \frac{C}{ \delta_0} \int_\bbr|a'|^2 v^\beta  |p(v)-p(\tilde v_\eps)|^2 \,d\xi}_{=:\mathcal{B}_{6}} ,\\
&|\mathcal{B}_4(U)|\le \underbrace{ \int_\bbr|a'|| \partial_{\xi} p(\tilde v_\eps)| v^\beta  |p(v)-p(\tilde v_\eps)|^2 \,d\xi}_{=:\mathcal{B}_{7}} + 
 \underbrace{\int_\bbr|a'|| \partial_{\xi} p(\tilde v_\eps)|  \frac{|v^\beta-\tilde v_\eps^\beta|^2}{v^\beta} \,d\xi}_{=:\mathcal{B}_{8}} ,\\
&|\mathcal{B}_5(U)|\le \delta_0 \mathcal{D}(U) +  \underbrace{ \frac{C}{\delta_0} \int_\bbr | \partial_{\xi} p(\tilde v_\eps)|^2 \frac{|v^\beta-\tilde v_\eps^\beta|^2}{v^\beta} \,d\xi}_{=:\mathcal{B}_{9}} . 
\end{aligned}
\end{align*}
Using \eqref{l4} and \eqref{l8} together with $| \partial_{\xi} p(\tilde v_\eps)|\le C\frac{\eps}{\lambda}|a'|$ and $\delta_0^{-1}\eps<\lambda<\delta_0$, we have
\begin{align*}
\begin{aligned}
&|\mathcal{B}_6(U)-\mathcal{B}_6(\bar U)|+|\mathcal{B}_7(U)-\mathcal{B}_7(\bar U)|\le C \delta_0  \left( \mathcal{D}(U) +\frac{\eps}{\lambda} \mathcal{G}_2(U)\right),\\
&|\mathcal{B}_8(U)-\mathcal{B}_8(\bar U)|+|\mathcal{B}_9(U)-\mathcal{B}_9(\bar U)|\le C\delta_0 \left( \mathcal{D}(U) + \left(\mathcal{G}_2(U)-\mathcal{G}_2(\bar U) \right)+\left(\frac{\eps}{\lambda}\right)^2 \mathcal{G}_2(\bar U)\right) .
\end{aligned}
\end{align*}
Therefore, we have
\begin{align*}
\begin{aligned}
\sum_{i=3}^5 |\mathcal{B}_i(U)| \le \sum_{i=6}^9 |\mathcal{B}_i(\bar U)|+ C\delta_0 \mathcal{D}(U) + C\delta_0 \left(\left(\mathcal{G}_2(U)-\mathcal{G}_2(\bar U) \right)+\frac{\eps}{\lambda} \mathcal{G}_2(\bar U)\right).
\end{aligned}
\end{align*}
Since $|a'|\le C\eps\lambda$, we have
\[
\sum_{i=6}^9 |\mathcal{B}_i(\bar U)|\le C\frac{\eps\lambda}{\delta_0} \int_\bbr|a'| \bar v^\beta  |p(\bar v)-p(\tilde v_\eps)|^2 \,d\xi 
+C \eps^2 \int_\bbr|a'| \frac{|\bar v^\beta-\tilde v_\eps^\beta|^2}{\bar v^\beta} \,d\xi .
\]
Using $C^{-1}\le \bar v\le C$ and \eqref{pQ-equi0} together with $ |p(\bar v)-p(\tilde v_\eps)|\le\delta_3$, we have
\[
\int_\bbr|a'| \bar v^\beta  |p(\bar v)-p(\tilde v_\eps)|^2 \,d\xi \le C\int_\bbr|a'| Q(\bar v|\tilde v_\eps) \,d\xi \le C \mathcal{G}_2(\bar U). 
\]
Moreover, since 
\begin{align*}
\begin{aligned}
\int_\bbr|a'| \frac{|\bar v^\beta-\tilde v_\eps^\beta|^2}{\bar v^\beta} \,d\xi &\le C
\int_\bbr|a'| \frac{|\bar v^\beta-\tilde v_\eps^\beta|^2}{Q(\bar v|\tilde v_\eps)}Q(\bar v|\tilde v_\eps) \,d\xi \\
 &\le C \int_\bbr|a'| Q(\bar v|\tilde v_\eps) \,d\xi \le C \mathcal{G}_2(\bar U),
\end{aligned}
\end{align*}
we have
\[
\sum_{i=6}^9 |\mathcal{B}_i(\bar U)|\le C\delta_0\frac{\eps}{\lambda} \mathcal{G}_2(\bar U).
\]
Hence we have the desired estimate \eqref{n14}.\\

\noindent{\bf Proof of \eqref{n2}:}
First, using \eqref{p-est1}, \eqref{pQ-equi0} and \eqref{l1}, we have
\begin{align*}
\begin{aligned}
|\mathcal{B}_1(\bar U)|+|\mathcal{B}_2^+(\bar U)| &\leq C\int_\bbr |a'| Q(\bar v|\tilde{v}_\eps)\,d\xi \leq C\int_\bbr |a'| Q( v|\tilde{v}_\eps)\,d\xi  \le C\frac{\eps^2}{\lambda}.
\end{aligned}
\end{align*}
Then, it follows from \eqref{bad0}, \eqref{n1}-\eqref{n14} and \eqref{l2} that
\[
|\mathcal{B}_{\delta_3}(U)| \le C\frac{\eps^2}{\lambda} + C\sqrt{\delta_0} \left(\mathcal{D}(U)+\mathcal{G}_2(U) \right) \le C^* \frac{\eps^2}{\lambda}+ C\sqrt{\delta_0} \mathcal{D}(U).
\]
\noindent{\bf Proof of \eqref{m1}:} We split the proof in three steps.
\vskip0.2cm
\noindent{\it Step 1:} 
First of all, we use the notations $Y_1^s, Y_2^s, Y_3^s$ and $Y_4^s$ for the terms of $Y_s$ as follows:
\[
Y_s =\underbrace{-\int_{\Omega^c} a' Q(v|\tilde v_\eps) d\xi}_{=:Y_1^s} \underbrace{-\int_{\Omega^c} a \partial_\xi p(\tilde v_\eps)(v-\tilde v_\eps)d\xi}_{=:Y_2^s} \underbrace{-\int_{\Omega^c} a' \frac{|h-\tilde h_\eps|^2}{2} d\xi}_{=:Y_3^s} \underbrace{+\int_{\Omega^c} a \partial_\xi \tilde h_\eps(h-\tilde h_\eps) d\xi}_{=:Y_4^s}.
\]

We use \eqref{big1},\eqref{big2}, \eqref{l50} together with \eqref{compare1} to have
\begin{align}
\begin{aligned}\label{Ygd}
&|Y_g(U)-Y_g(\bar U)| +|Y_1^s(U)-Y_1^s(\bar U)| +|Y_2^s(U)-Y_2^s(\bar U)| \\
&\quad \le C \int_\bbr |a'| \Big(\big| |p(v)-\pt|^2-|p(\bar v_b)-\pt|^2\big| +\big|Q(v|\tilde{v}_\eps)-Q(\bar v|\tilde{v}_\eps)\big| \\
&\qquad\quad +|v-\bar v|+ |p(v)-p(\bar v_b)|\Big)\,d\xi \leq C\sqrt{\frac{\eps}{\lambda}}\mathcal{D}(U)+C\left(\mathcal{G}_2(U)-\mathcal{G}_2(\bar U) \right).
\end{aligned}
\end{align}
On the other hand, \eqref{ns2} yields
\beq\label{Y12}
|Y_1^s(\bar U)|+|Y_2^s(\bar U)|\le \int_{\Omega^c} |a'| \left( Q(\bar v|\tilde v_\eps) +|\bar v -\tilde v_\eps | \right)  d\xi \le C \sqrt{\frac{\eps}{\lambda}} \left(\mathcal{D}(U) +\left(\frac{\eps}{\lambda}\right)^2\mathcal{G}_2(U) \right).
\eeq

Next, by the definitions of $\mathcal{G}_1^\pm$ in \eqref{ggd}, we have
\begin{align*}
\begin{aligned}
|Y_3^s(U)|+|Y_b(U)| &\le C\mathcal{G}_1^-(U)+C\mathcal{G}_1^+(U) + C\int_\Omega |a'| |p(v)-\pt|^2 \,d\xi \\
&\le C(\mathcal{G}_1^-(U)+\mathcal{G}_1^+(U)+|\mathcal{B}_{\delta_3}(U)|),
\end{aligned}
\end{align*}
Moreover, since
\[
\mathcal{G}_1^+(U) \leq C\int_\Omega |a'| \left(|h-\tilde h_\eps|^2+|p(v)-\pt|^2\right)\,d\xi \le C\int_\Omega |a'| |h-\tilde h_\eps|^2\,d\xi +C |\mathcal{B}_{\delta_3}(U)|,
\]
using \eqref{n2}, we have
\[
|Y_3^s(U)|+|Y_b(U)|  \le C\int_\bbr |a'| |h-\tilde h_\eps|^2 \,d\xi +C^* \frac{\eps^2}{\lambda}+ C\sqrt{\delta_0} \mathcal{D}(U).
\]
Therefore, using \eqref{l1}, \eqref{l2}, and the assumption $\mathcal{D}(U)\leq C^*\eps^2/\lambda$, it follows from \eqref{Ygd}, \eqref{Y12} and the above estimate that
\[
|Y_g(U)-Y_g(\bar U)| +|Y_1^s(U)| +|Y_2^s(U)| +|Y_3^s(U)|+|Y_b(U)|  \le C\frac{\eps^2}{\lambda}.
\]
\vskip0.2cm
\noindent{\it Step 2:} 
First of all, using Young's inequality and \eqref{pQ-equi0}, \eqref{n13}, we estimate
\begin{align*}
\begin{aligned}
&\left|  \int_\Omega a' \big(p(v)-p(\tilde v_\eps)\big)\Big(h-\tilde h_\eps-\frac{p(v)-p(\tilde v_\eps)}{\sigma_\eps}\Big) d\xi \right|\\
&\quad \le\left(\frac{\lambda}{\eps}\right)^{1/4}\mathcal{G}_1^+(U)+C\left(\frac{\eps}{\lambda}\right)^{1/4}\int_\Omega |a'| |p(v)-\pt|^2\,d\xi\\
&\quad \le\left(\frac{\lambda}{\eps}\right)^{1/4}\mathcal{G}_1^+(U)+C\left(\frac{\eps}{\lambda}\right)^{1/4}\left(\mathcal{B}_2^+(\bar U) + \left( \mathcal{B}_2^+(U)-\mathcal{B}_2^+(\bar U) \right) \right)\\
&\quad \le\left(\frac{\lambda}{\eps}\right)^{1/4}\mathcal{G}_1^+(U)+C\left(\frac{\eps}{\lambda}\right)^{1/4}\left(\mathcal{G}_2(\bar U) +\sqrt{\frac{\eps}{\lambda}} \mathcal{D}(U) \right).
\end{aligned}
\end{align*}
Therefore, this estimate, \eqref{Ygd} and \eqref{Y12} yield
\begin{align*}
\begin{aligned}
&|Y_g(U)-Y_g(\bar U)| +|Y_1^s(U)| +|Y_2^s(U)| +|Y_3^s(U)|+|Y_b(U)|   \\
&\le C\sqrt{\frac{\eps}{\lambda}}\mathcal{D}(U)+C\left(\mathcal{G}_2(U)-\mathcal{G}_2(\bar U) \right)+C\mathcal{G}_1^-(U) + 2\left(\frac{\lambda}{\eps}\right)^{1/4}\mathcal{G}_1^+(U) +C\left(\frac{\eps}{\lambda}\right)^{1/4}\mathcal{G}_2(\bar U).
\end{aligned}
\end{align*}
\vskip0.2cm
\noindent{\it Step 3:} 
For the remaining terms, using H\"older's inequality together with $|\tilde{h}_\eps'|\le C\frac{\eps}{\lambda} |a'|$, we estimate
\begin{align*}
\begin{aligned}
&|Y_4^s(U)|^2\leq C\left(\frac{\eps}{\lambda}\right)^2 \left(\int_\bbr|a'|\,d\xi\right)\int_{\Omega^c} |a'| |h-\tilde{h}_\eps|^2\,d\xi \leq C\frac{\eps^2}{\lambda} \mathcal{G}_1^-(U),\\
&|Y_l(U)|^2\leq C\left(\frac{\eps}{\lambda}\right)^2 \left(\int_\bbr|a'|\,d\xi\right)\int_{\Omega} |a'| \left(h-\tilde h_\eps -\frac{p(\bar v)-p(\tilde v_\eps)}{\sigma_\eps} \right)^2\,d\xi \leq C\frac{\eps^2}{\lambda} \mathcal{G}_1^+(U).
\end{aligned}
\end{align*}
Therefore, this together with Step1 and Step2 yield
\begin{align*}
\begin{aligned}
&|Y_g(U)-Y_g(\bar U)|^2 +|Y_b(U)|^2 +|Y_l(U)|^2+|Y_s(U)|^2  \\
&\le\left(|Y_g(U)-Y_g(\bar U)| +|Y_1^s(U)| +|Y_2^s(U)| +|Y_3^s(U)|+|Y_b(U)|\right)^2 +  |Y_4^s(U)|^2 +|Y_l(U)|^2  \\
&\le C\frac{\eps^2}{\lambda}\left(\sqrt{\frac{\eps}{\lambda}}\mathcal{D}(U)+\left(\mathcal{G}_2(U)-\mathcal{G}_2(\bar U) \right)+\mathcal{G}_1^-(U) + \left(\frac{\lambda}{\eps}\right)^{1/4}\mathcal{G}_1^+(U) +\left(\frac{\eps}{\lambda}\right)^{1/4}\mathcal{G}_2(\bar U) \right).
\end{aligned}
\end{align*}

\subsection{Proof of Proposition \ref{prop:main}}
We now prove the main Proposition  \ref{prop:main}. We split the proof into two steps, depending on the strength of the dissipation term $\mathcal{D}(U)$.

\vskip0.2cm
\noindent{\it Step 1:}  
We first consider the case of
$
\mathcal{D}(U)\geq 4 C^* \frac{\eps^2}{\lambda}, $  where the constant $C^*$ is defined as in Proposition \ref{prop_out}.
Then using $\eqref{n2}$ and taking $\delta_0$ small enough, we have
\begin{align*}
\begin{aligned}
\mathcal{R}(U)&\le-\frac{|Y(U)|^2}{\eps^4}+\left(1+\delta_0\frac{\eps}{\lambda}\right)|\mathcal{B}_{\delta_3}(U)|-\mathcal{G}_1^-(U)-\mathcal{G}_1^+(U) -\left(1-\delta_0\frac{\eps}{\lambda}\right)\mathcal{G}_2(U) -(1-\delta_0)\mathcal{D}(U)  \\
&\leq 2|\mathcal{B}_{\delta_3}(U)|-(1-\delta_0)\mathcal{D}(U)\\
& \leq 2C^*\frac{\eps^2}{\lambda}-\left(1-\delta_0-2C\sqrt{\delta_0}\right)\mathcal{D}(U)\\
& \leq 2 C^* \frac{\eps^2}{\lambda}-\frac{1}{2}\mathcal{D}(U)\leq 0,
\end{aligned}
\end{align*}
which gives the desired result.

\vskip0.2cm
\noindent{\it Step 2:}  We now assume the other alternative, i.e., $\mathcal{D}(U)\leq 4 C^* \frac{\eps^2}{\lambda}.$ \\
We will use Proposition \ref{prop:main3} to get the desired result.  First of all, we have \eqref{YC2}, and for the small constant $\delta_3$ of Proposition \ref{prop:main3} associated to the constant $C_2$ of \eqref{YC2}, we have $|p(\bar v)-\pt|\leq \delta_3$.\\
Using
$$
Y_g(\bar U)=Y(U)-(Y_g(U)-Y_g(\bar U))-Y_b(U)-Y_l(U)-Y_s(U),
$$
we have
$$
|Y_g(\bar U)|^2\leq 4\left(|Y(U)|^2+|Y_g(U)-Y_g(\bar U)|^2+ |Y_b(U)|^2+|Y_l(U)|^2+|Y_s(U)|^2\right),
$$
which can be written as
$$
-4|Y(U)|^2\leq -|Y_g(\bar U)|^2+4|Y_g(U)-Y_g(\bar U)|^2+ 4|Y_b(U)|^2+4|Y_l(U)|^2+4|Y_s(U)|^2.
$$
Now, let us take $\delta_0$ small enough such that $\delta_0\le\delta_3^9$. (In fact, since we see from the proofs of Lemma \ref{lemma_out1}-\ref{lemma_out4} that the constants $C$ in Proposition \ref{prop_out} depend on $\delta_3$ as algebraically negative power of it, we take  $\delta_0$ smaller enough if needed.)\\
Then we find that for any $\eps<\eps_0(\le\delta_3)$ and $\eps/\lambda<\delta_0(\le\delta_3^9)$,
\begin{align*}
\begin{aligned}
\mathcal{R}(U)&\le-\frac{4|Y(U)|^2}{\eps\delta_3}+\mathcal{B}_{\delta_3}(U) +\delta_0\frac{\eps}{\lambda} |\mathcal{B}_{\delta_3}(U)| \\
&\quad -\mathcal{G}_1^-(U)-\mathcal{G}_1^+(U) -\left(1-\delta_0\frac{\eps}{\lambda}\right)\mathcal{G}_2(U) -(1-\delta_0)\mathcal{D}(U) \\
&\leq -\frac{|Y_g(\bar U)|^2}{\eps\delta_3}+ \left(\mathcal{B}_1(\bar U)+\mathcal{B}_2^+(\bar U)\right)+\delta_0\frac{\eps}{\lambda}\left(|\mathcal{B}_1(\bar U)|+|\mathcal{B}_2^+(\bar U)|\right)\\
&\quad-\left(1-\delta_3\frac{\eps}{\lambda}\right)\mathcal{G}_2(\bar U)-(1-\delta_3)\mathcal{D}(\bar U)\\
&\quad \underbrace{+\frac{4}{\eps\delta_3}\left(|Y_g(U)-Y_g(\bar U)|^2+|Y_b(U)|^2+|Y_l(U)|^2+|Y_s(U)|^2\right)}_{=:J_1} \\
&\quad \underbrace{+\left(1+\delta_0\frac{\eps}{\lambda}\right)\left(|\mathcal{B}_1(U)-\mathcal{B}_1(\bar U)|+|\mathcal{B}_2^+(U)-\mathcal{B}_2^+(\bar U)| +|\mathcal{B}_2^-(U)| +\sum_{i=3}^5|\mathcal{B}_i(U)|  \right)}_{=:J_2} \\
&\quad -\mathcal{G}_1^-(U) -\mathcal{G}_1^+(U)-\frac{1}{2}\left(\mathcal{G}_2(U)-\mathcal{G}_2(\bar U)\right)-\frac{\delta_3}{2}\frac{\eps}{\lambda}\mathcal{G}_2(\bar U)-(\delta_3-\delta_0)\mathcal{D}(U),
\end{aligned}
\end{align*}
where we used $\mathcal{D}(\bar U)\le \mathcal{D}(U)$ by \eqref{eq_D}. We claim that $J_1$, $J_2$ are controlled by the last line above.  
Indeed, it follows from \eqref{n1}-\eqref{n14} and \eqref{m1} that for any $\eps/\lambda<\delta_0(\le\delta_3^9)$,
\begin{align*}
\begin{aligned}
J_1&\le \frac{C}{\delta_3}\frac{\eps}{\lambda}\left(\sqrt{\frac{\eps}{\lambda}}\mathcal{D}(U)+\left(\mathcal{G}_2(U)-\mathcal{G}_2(\bar U) \right)+\mathcal{G}_1^-(U) + \left(\frac{\lambda}{\eps}\right)^{1/4}\mathcal{G}_1^+(U) +\left(\frac{\eps}{\lambda}\right)^{1/4}\mathcal{G}_2(\bar U) \right)\\
&\le \frac{C}{\delta_3}\left(\frac{\eps}{\lambda}\right)^{1/4} \left(\mathcal{D}(U)+\left(\mathcal{G}_2(U)-\mathcal{G}_2(\bar U) \right)+\mathcal{G}_1^-(U) + \mathcal{G}_1^+(U) +\frac{\eps}{\lambda}\mathcal{G}_2(\bar U) \right)\\
&\le \frac{1}{4}\delta_3 \left(\mathcal{D}(U)+\left(\mathcal{G}_2(U)-\mathcal{G}_2(\bar U) \right)+\mathcal{G}_1^-(U) + \mathcal{G}_1^+(U) +\frac{\eps}{\lambda}\mathcal{G}_2(\bar U) \right),\\
J_2&\le C \sqrt\delta_0 \left( \mathcal{D}(U) +\left(\mathcal{G}_2(U)-\mathcal{G}_2(\bar U) \right)+\frac{\eps}{\lambda} \mathcal{G}_2(\bar U)\right)\\
&\le \frac{1}{4}\delta_3\left( \mathcal{D}(U) +\left(\mathcal{G}_2(U)-\mathcal{G}_2(\bar U) \right)+\frac{\eps}{\lambda} \mathcal{G}_2(\bar U)\right).
\end{aligned}
\end{align*}
Therefore, we have
\begin{align*}
\begin{aligned}
\mathcal{R}(U) &\le -\frac{|Y_g(\bar U)|^2}{\eps\delta_3}+\left(\mathcal{B}_1(\bar U)+\mathcal{B}_2^+(\bar U)\right)+\delta_3\frac{\eps}{\lambda}\left(|\mathcal{B}_1(\bar U)|+|\mathcal{B}_2^+(\bar U)|\right)\\
&\quad -\left(1-\delta_3\frac{\eps}{\lambda}\right)\mathcal{G}_2(\bar U)-(1-\delta_3)\mathcal{D}(\bar U).
\end{aligned}
\end{align*}
Since the above quantities $Y_g(\bar U), \mathcal{B}_1, \mathcal{B}_2^+(\bar U), \mathcal{G}_2(\bar U)$ and $\mathcal{D}(\bar U)$ depends only on $\bar v$ through $\bar U$, and 
$\mathcal{B}_1(\bar U)=\mathcal{I}_1(\bar v)$ and $\mathcal{B}_2^+(\bar U)=\mathcal{I}_2^+(\bar v)$,
it follows from Proposition \ref{prop:main3} that $\mathcal{R}(U)\le 0$. \\
Hence we complete the proof of Proposition \ref{prop:main}.\\

\section{Proof of Theorem \ref{thm_inviscid}}\label{sec:main}
\setcounter{equation}{0}

\subsection{Proof of \eqref{ini_conv} : Well-prepared initial data}
For a given datum $(v^0,u^0)$ satisfying \eqref{basic_ini}, let $\{(v_0^{r},u_0^r)\}_{r>0}$ be a sequence of truncations defined by
\[
v_0^r=\left\{ \begin{array}{ll}
       v^0{\mathbf 1}_{\{r\le v^0\le r^{-1}\}}\quad &\mbox{if } -r^{-1}\le x\le r^{-1},\\
       v_-\quad &\mbox{if } x\le -r^{-1},\\
       v_+\quad &\mbox{if } x\ge r^{-1}, \end{array} \right.
\]
and
\[
u_0^r=\left\{ \begin{array}{ll}
       u^0{\mathbf 1}_{\{-r^{-1}\le u^0\le r^{-1}\}}\quad &\mbox{if } -r^{-1}\le x\le r^{-1},\\
       u_-\quad &\mbox{if } x\le -r^{-1},\\
       u_+\quad &\mbox{if } x\ge r^{-1}. \end{array} \right.
\]
Then, we consider a mollification of the above sequence: using $\phi_{\nu}(x):=\frac{1}{\sqrt\nu}\phi_1\big(\frac{x}{\sqrt\nu}\big)$
where $\phi_1$ is a smooth mollifier with supp$\phi_1=[-1,1]$, consider a double sequence $\{(v_0^{r,\nu},u_0^{r,\nu})\}_{r,\nu>0}$ defined by
\[
v_0^{r,\nu}=v_0^{r}*\phi_{\nu},\quad u_0^{r,\nu}=u_0^{r}*\phi_{\nu}.
\]
First, we will show 
\beq\label{claim-con1}
\lim_{r\to 0}\lim_{\nu\to0}\int_\bbr  Q\big(v_0^{r,\nu}| \tilde v\big(\frac{x}{\nu}\big)\big)dx= \int_{-\infty}^0 Q(v^0|v_-)dx+\int_0^{\infty} Q(v^0|v_+)dx.
\eeq
For a fixed $r$, since $v_0^{r,\nu} \to v_0^{r}$ a.e., and $\tilde v\big(\frac{x}{\nu}\big)\to v_-$ a.e. $x<0$ as $\nu\to0$, using 
\begin{align*}
\begin{aligned}
& Q\big(v_0^{r,\nu}| \tilde v\big(\frac{x}{\nu}\big)\big)-  Q(v_0^r|v_-) =  \big(Q(v_0^{r,\nu})-Q(v_0^r)\big) + \Big( Q(v_-)-Q\big( \tilde v\big(\frac{x}{\nu}\big) \big) \Big)\\
&\qquad -\Big( Q'\big( \tilde v\big(\frac{x}{\nu}\big) \big) -Q'(v_-) \Big) \big(v_0^{r,\nu}- \tilde v\big(\frac{x}{\nu}\big) \big)
-Q'(v_-) \Big( \big(v_0^{r,\nu}-v_0^r \big) -  \big( \tilde v\big(\frac{x}{\nu}\big) - v_- \big)  \Big),
\end{aligned}
\end{align*}
we have
\[
 Q\big(v_0^{r,\nu}| \tilde v\big(\frac{x}{\nu}\big)\big) \to  Q(v_0^r|v_-)\quad \mbox{a.e. } x<0,\quad \mbox{as }~\nu\to0.
\]
Likewise,
\[
 Q\big(v_0^{r,\nu}| \tilde v\big(\frac{x}{\nu}\big)\big) \to Q(v_0^r|v_+)\quad \mbox{a.e. } x>0,\quad \mbox{as }~\nu\to0.
\]
Moreover, since 
\[
|v_0^{r}-\bar v| \le  \max(r^{-1},v_\pm) {\mathbf 1}_{\{|x|\le r^{-1}\}},
\]
we have
\begin{align*}
\begin{aligned}
 Q\big(v_0^{r,\nu}| \tilde v\big(\frac{x}{\nu}\big)\big) &\le C_r \big|v_0^{r,\nu}-\tilde v\big(\frac{x}{\nu}\big)\big|^2
\le C_r \Big( \big|v_0^{r,\nu}- \bar v\big|^2 + \big|\tilde v\big(\frac{x}{\nu}\big) -\bar v\big|^2 \Big)\\
&\le   \max(r^{-2},v_\pm^2){\mathbf 1}_{\{|x|\le r^{-1}+1\}} +\big|\tilde v-\bar v\big|^2,\quad \nu<1.
\end{aligned}
\end{align*}
Since ${\mathbf 1}_{\{|x|\le r^{-1}+1\}} +\big|\tilde v-\bar v\big|^2\in L^1(\bbr)$, the dominated convergence theorem implies
\[
\lim_{\nu\to0}\int_\bbr  Q\big(v_0^{r,\nu}| \tilde v\big(\frac{x}{\nu}\big)\big)dx= \int_{-\infty}^0 Q(v_0^r|v_-)dx+\int_0^{\infty} Q(v_0^r|v_+)dx.
\]
Furthermore, since $Q(v_0^r|\bar v)\le Q(v^0|\bar v)\in L^1(\bbr)$ and $v_0^r\to v^0$ a.e. as $r\to 0$, we have 
\[
\lim_{r\to0} \Big(\int_{-\infty}^0 Q(v_0^r|v_-)dx\,+\,\int_0^{\infty} Q(v_0^r|v_+)dx\Big) = \int_{-\infty}^0 Q(v^0|v_-)dx\,+\,\int_0^{\infty} Q(v^0|v_+)dx,
\]
which completes \eqref{claim-con1}.\\
Hence by the diagonal extraction of \eqref{claim-con1}, there exists a sequence (still denoted by $v_0^\nu$) such that
\[
\lim_{\nu\to0}\int_\bbr Q\big(v_0^{\nu}| \tilde v\big(\frac{x}{\nu}\big)\big)dx= \int_{-\infty}^0 Q(v^0|v_-)dx+\int_0^{\infty} Q(v^0|v_+)dx.
\]
In particular, we have from the above construction that $v_0^\nu$ converges to $v^0$ in $L_{loc}^{1}(\bbr)$, and especially :
\beq\label{v-inicon}
v_0^\nu \to v^0 \quad \mbox{in } W_{loc}^{-s,1}(\bbr),~ s>0.
\eeq
where this convergence will be used in the proof of \eqref{v-conti}.\\

Using the same argument as above, we show
\begin{align}
\begin{aligned}\label{claim-con2}
\lim_{r\to0}\lim_{\nu\to0}\int_\bbr \frac{1}{2}\left(u_0^{r,\nu} +\nu \left(p(v_0^{r,\nu})^{\frac{\alpha}{\gamma}}\right)_x  -\tilde u^\nu(x) - \nu \left(p\left(\tilde v^\nu(x) \right)^{\frac{\alpha}{\gamma}}\right)_x\right)^2 dx = \int_{-\infty}^\infty  \frac{|u^0-\bar u|^2}{2}dx,
\end{aligned}
\end{align}
where $\tilde u^\nu(x)=\tilde u (x/\nu)$ and $\tilde v^\nu(x)=\tilde v (x/\nu)$.\\
Indeed, since $|v_0^r|\le r^{-1}$ for any small $r>0$, we have
\begin{align*}
\begin{aligned}
\left|\nu  \left(p(v_0^{r,\nu})^{\frac{\alpha}{\gamma}}\right)_x\right| &=\nu \frac{\alpha}{\gamma}p(v_0^{r,\nu})^{\frac{\alpha}{\gamma}-1}\left|p'(v_0^{r,\nu})\right| \left| v_0^{r}* \left(\phi_\nu (x)\right)_x \right|\\
& \le \frac{\alpha}{\gamma} p(r^{-1})^{\frac{\alpha}{\gamma}-1} |p'(r)| r^{-1} \int_\bbr \big|\phi_1'\big(\frac{y}{\sqrt\nu}\big)\big|dy\\
&\le C(r) \sqrt{\nu} ,
\end{aligned}
\end{align*}
which means
\[
\left\|\nu   \left(p(v_0^{r,\nu})^{\frac{\alpha}{\gamma}}\right)_x\right\|_{L^{\infty}(\bbr)}\to 0\quad \mbox{as}~ \nu\to0. 
\]
Moreover, since $\tilde v'\big(\frac{x}{\nu}\big)\to 0$ a.e. as $\nu\to0$, we have
\[
\left(u_0^{r,\nu} +\nu \left(p(v_0^{r,\nu})^{\frac{\alpha}{\gamma}}\right)_x  -\tilde u\left(\frac{x}{\nu}\right) - \nu \left(p\left(\tilde v\left(\frac{x}{\nu}\right)\right)^{\frac{\alpha}{\gamma}}\right)_x\right)^2
\to \left\{ \begin{array}{ll}
       |u_0^r-u_-|^2 \quad \mbox{for a.e. } x<0,\\
        |u_0^r-u_+|^2 \quad \mbox{for a.e. } x>0. \end{array} \right. 
\]
Furthermore, since 
\begin{align*}
\begin{aligned}
&\left(u_0^{r,\nu} +\nu \left(p(v_0^{r,\nu})^{\frac{\alpha}{\gamma}}\right)_x  -\tilde u\left(\frac{x}{\nu}\right) - \nu \left(p\left(\tilde v\left(\frac{x}{\nu}\right)\right)^{\frac{\alpha}{\gamma}}\right)_x\right)^2\\
&\quad \le 2\left( |u_0^{r,\nu}-\bar u|^2 + \left|\tilde u\left(\frac{x}{\nu}\right)-\bar u\right|^2 + \left|\nu \left(p(v_0^{r,\nu})^{\frac{\alpha}{\gamma}}\right)_x \right|^2 + \left|\nu  \left(p\left(\tilde v\left(\frac{x}{\nu}\right)\right)^{\frac{\alpha}{\gamma}}\right)_x\right|^2\right)\\
&\quad \le C(r)\left( {\mathbf 1}_{\{|x|\le r^{-1}+1\}} +\big|\tilde u-\bar u\big|^2 +\left| \phi_1'\left(\frac{x}{\sqrt\nu}\right) \right|^2 +\left|\tilde v'\left(\frac{x}{\nu}\right)\right|^2 \right)\\
&\quad \le C(r)\left( {\mathbf 1}_{\{|x|\le r^{-1}+1\}} +\big|\tilde u-\bar u\big|^2 +\left| \phi_1'\left(x\right) \right|^2 +\left|\tilde v'\left(x\right)\right|^2 \right) =: g(x),
\end{aligned}
\end{align*}
and $g\in L^1(\bbr)$, the dominated convergence theorem implies
\begin{align*}
\begin{aligned}
&\lim_{\nu\to0}\int_\bbr \frac{1}{2}\left(u_0^{r,\nu} +\nu \left(p(v_0^{r,\nu})^{\frac{\alpha}{\gamma}}\right)_x  -\tilde u\left(\frac{x}{\nu}\right) - \nu \left(p\left(\tilde v\left(\frac{x}{\nu}\right)\right)^{\frac{\alpha}{\gamma}}\right)_x\right)^2 dx\\
&\quad= \int_{-\infty}^0 \frac{|u_0^r-u_-|^2}{2}dx+\int_0^{\infty} \frac{|u_0^r-u_+|^2}{2}dx.
\end{aligned}
\end{align*}
Furthermore, since $|u_0^r - \bar u|^2\le |u^0 - \bar u|^2 \in L^1(\bbr)$ and $u_0^r\to u^0$ a.e. as $r\to 0$, we have 
\[
\lim_{r\to0} \Big(\int_{-\infty}^0 \frac{|u_0^r-u_-|^2}{2}dx+\int_0^{\infty} \frac{|u_0^r-u_+|^2}{2}dx\Big) = \int_{-\infty}^\infty  \frac{|u^0-\bar u|^2}{2}dx,
\]
which completes \eqref{claim-con2}. Hence, using the diagonal extraction as before, there exists a sequence (still denoted by $u_0^\nu$) such that
\begin{align*}
\begin{aligned}
\lim_{\nu\to0}\int_\bbr \frac{1}{2}\left(u_0^{r} +\nu \left(p(v_0^{r})^{\frac{\alpha}{\gamma}}\right)_x  -\tilde u^\nu(x) - \nu \left(p\left(\tilde v^\nu(x) \right)^{\frac{\alpha}{\gamma}}\right)_x\right)^2 dx = \int_{-\infty}^\infty  \frac{|u^0-\bar u|^2}{2}dx.
\end{aligned}
\end{align*}

\subsection{Proof for the main part of Theorem \ref{thm_inviscid}}
We here present a proof for the second part (ii) of Theorem \ref{thm_inviscid}.

\subsubsection{{\bf Uniform estimates in $\nu$}}
Let $\{(v^{\nu}, u^{\nu})\}_{\nu>0}$ be a sequence of solutions on $(0,T)$ to \eqref{inveq} with the initial datum $(v^{\nu}_0, u^{\nu}_0)$. Our starting point is to apply Theorem \ref{thm_general} to the below functions: 
\begin{align*}
\begin{aligned}
&v(t,x)=v^{\nu}(\nu t, \nu x),\quad \tilde v(x):= \tilde v^{\nu}(\nu x),\quad u(t,x)=u^{\nu}(\nu t, \nu x),\quad \tilde u(x):= \tilde u^{\nu}(\nu x).
\end{aligned}
\end{align*}
That is, using \eqref{cont_main} in Theorem \ref{thm_general} together with \eqref{d-weight}, we have
\begin{align*}
\begin{aligned}
&\int_{-\infty}^\infty E\big((v,u)(t,x)| (\tilde v, \tilde u) (x-X(t))\big) dx \\
&\quad+\int_{0}^{T/\nu} \int_{-\infty}^{\infty} |\tilde v' (x)| Q\left(v(t,x)|\tilde v(x-X(t))\right) dx dt \\
&\quad + \int_{0}^{T/\nu}\int_{-\infty}^{\infty} v^{\gamma-\alpha}(t,x)\big|\partial_x\big(p(v(t,x))-p(\tilde v(x-X(t)))\big)\big|^2dxdt  \\
&\le   C\int_{-\infty}^{\infty} E\big((v_0^\nu,u_0^\nu)(\nu x)| (\tilde v, \tilde u)(x)\big) dx.
\end{aligned}
\end{align*}
Then by the change of variables $t\mapsto t/\nu, x\mapsto x/\nu$, we have
\begin{align*}
\begin{aligned}
&\int_{-\infty}^{\infty}  E_\nu \big((v^{\nu},u^{\nu})(t,x)| (\tilde v^{\nu}, \tilde u^{\nu})(x-X_{\nu}(t))\big) dx \\
&\qquad+\int_{0}^{T}\int_{-\infty}^{\infty} |(\tilde v^{\nu})' (x)| Q\left(v^{\nu}(t,x)|\tilde v^{\nu}(x-X_{\nu}(t))\right) dx dt \\
&\qquad +\nu\int_{0}^{T}\int_{-\infty}^{\infty} (v^{\nu})^{\gamma-\alpha}(t,x)\big|\partial_x\big(p(v^{\nu}(t,x))-p(\tilde v^{\nu}(x-X_{\nu}(t)))\big)\big|^2dxdt\\
&\quad \le C\int_{-\infty}^{\infty} E_\nu\big((v^{\nu}_0,u^{\nu}_0)(x)| (\tilde v^{\nu}, \tilde u^{\nu})(x) \big) dx,
\end{aligned}
\end{align*}
where $X_{\nu}(t):= \nu X(t/\nu)$, and 
\beq\label{E_nu}
E_\nu((v_1,u_1)|(v_2,u_2)) :=\frac{1}{2}\left(u_1 +\nu \Big(p(v_1)^{\frac{\alpha}{\gamma}}\Big)_x -u_2 -\nu \Big(p(v_2)^{\frac{\alpha}{\gamma}}\Big)_x  \right)^2 +Q(v_1|v_2).
\eeq
For simplification, we introduce the variables: 
\beq\label{effective}
h^{\nu}:=u^{\nu}+\nu \Big(p(v^\nu)^{\frac{\alpha}{\gamma}}\Big)_x,\quad \tilde h^{\nu}:=\tilde u^{\nu}+\nu \Big(p(\tilde v^\nu)^{\frac{\alpha}{\gamma}}\Big)_x, 
\eeq
Then, recalling \eqref{eta_def}, the above estimate implies
\begin{align*}
\begin{aligned}
&\int_{-\infty}^{\infty} \eta\big((v^{\nu},h^{\nu})(t,x)| (\tilde v^{\nu}, \tilde h^{\nu})(x-X_{\nu}(t))\big) dx \\
&\qquad+\int_{0}^{T}\int_{-\infty}^{\infty} |(\tilde v^{\nu})' | \, Q\left(v^{\nu}(t,x)|\tilde v^{\nu}(x-X_{\nu}(t))\right) dx dt \\
&\qquad +\nu\int_{0}^{T}\int_{-\infty}^{\infty} (v^{\nu})^{\gamma-\alpha}\big|\partial_x\big(p(v^{\nu}(t,x))-p(\tilde v^{\nu}(x-X_{\nu}(t)))\big)\big|^2dxdt\\
&\quad \le C\int_{-\infty}^{\infty} E_\nu\big((v^{\nu}_0,u^{\nu}_0)| (\tilde v^{\nu}, \tilde u^{\nu}) \big) dx.
\end{aligned}
\end{align*}
Therefore, using \eqref{ini_conv}, we find that 
\begin{align}
\begin{aligned}\label{ineq-m}
&\mbox{for any $\delta\in(0,1)$, there exists $\nu_*$ such that for all $\nu<\nu_*$}, \\
&\int_{-\infty}^{\infty} \eta\big((v^{\nu},h^{\nu})(t,x)| (\tilde v^{\nu}, \tilde h^{\nu})(x-X_{\nu}(t))\big) dx\\
&\qquad+\int_{0}^{T}\int_{-\infty}^{\infty} |(\tilde v^{\nu})' | \, Q\left(v^{\nu}(t,x)|\tilde v^{\nu}(x-X_{\nu}(t))\right) dx dt \\
&\qquad +\nu\int_{0}^{T}\int_{-\infty}^{\infty} (v^{\nu})^{\gamma-\alpha}\big|\partial_x\big(p(v^{\nu}(t,x))-p(\tilde v^{\nu}(x-X_{\nu}(t)))\big)\big|^2dxdt\\
&\quad \le C\mathcal{E}_0 +\delta,
\end{aligned}
\end{align}
where 
\[
\mathcal{E}_0:=\int_{-\infty}^{\infty} \eta\big((v^0,u^0)| (\bar v, \bar u)\big) dx.
\]

\subsubsection{\bf Proof of \eqref{wconv}} 
We first prove the weak convergence \eqref{wconv}.\\

$\bullet$ {\bf Convergence of $\{v^\nu\}_{\nu>0}$ :}
For the given two end states $v_\pm$, we first fix a constant $M>1$ such that 
\[
\big(\min\{v_-,v_+\},\max\{v_-,v_+\} \big)\subset (M^{-1},M).
\]
Then we fix the constant $k_0>1$ in Lemma \ref{lem_Q1}, and set
\beq\label{Kdef}
K:=\max\{3M, k_0 \}.
\eeq
For the constant $K>1$, let $\overline\psi$ be a continuous function defined by
\beq\label{psi_k}
\overline\psi(x)=\left\{ \begin{array}{ll}
       x,\quad\mbox{if } K^{-1}\le x\le K,\\
       K^{-1},\quad\mbox{if } x< K^{-1},\\
       K,\quad\mbox{if } x> K.\end{array} \right.
\eeq
Then we set
\beq\label{v-trunc}
\underline v^\nu  :=\overline\psi(v^\nu),\quad v_e^\nu := v^\nu- \underline v^\nu.
\eeq
Note that the truncation $\underline v^\nu$ will be used in the proof of \eqref{uni-est} below.\\
Since
\[
|v_e^\nu|\le \max\Big\{(K^{-1}-v^\nu)_+, (v^\nu-K)_+ \Big\},
\]
and $M^{-1}<\tilde v^\nu (x-X_{\nu})<M$, we use \eqref{Q1} in Lemma \ref{lem_Q1} to have
\[
|v_e^\nu|\le C Q\big(v^\nu| \tilde v^\nu (x-X_{\nu}) \big).
\]
Then it follows from \eqref{ineq-m} that for all $\nu<\nu_*$,
\[
\int_{-\infty}^{\infty} |v_e^\nu| dx \le C\big( \mathcal{E}_0 +1\big).
\]
Therefore, $\{v_e^\nu\}_{\nu>0}$ is bounded in $L^{\infty}(0,T;L^1(\bbr))$.
Moreover, since the definition \eqref{v-trunc} implies that $\{\underline v^\nu\}_{\nu>0}$ is bounded in $L^{\infty}((0,T)\times\bbr)\subset L^\infty(0,T;L^1_{loc}(\bbr))$, we obtain that 
\beq\label{cpt-v}
\{v^\nu\}_{\nu>0} \mbox{ is bounded in } L^\infty(0,T;L^1_{loc}(\bbr)).
\eeq
Therefore,
there exists $v_{\infty} $ such that
\beq\label{vwc}
v^\nu \rightharpoonup v_{\infty} \quad \mbox{in} ~\mathcal{M}_{\mathrm{loc}}((0,T)\times\bbr),
\eeq
and
\[
v_\infty \in L^\infty(0,T; L^\infty(\bbr)+\mathcal{M}(\bbr)). 
\]
\vskip0.2cm

$\bullet$ {\bf Convergence of $\{u^\nu\}_{\nu>0}$ :}

We split the proof into two steps.
\vskip0.2cm
\noindent{\it Step 1:}
We will first show convergence of $\{h^\nu\}_{\nu>0}$.\\
For the given two end states $u_\pm$, we first fix a constant $L>1$ such that 
\[
\big(\min\{u_-,u_+\},\max\{u_-,u_+\} \big)\subset \big(-\frac{L}{2},\frac{L}{2}\big).
\]
Then let $\overline\varphi$ be a continuous function defined by  
\beq\label{phi_k}
\overline\varphi(x)=\left\{ \begin{array}{ll}
       x,\quad\mbox{if }  |x|\le L,\\
       -L,\quad\mbox{if } x< -L,\\
       L,\quad\mbox{if } x> L.\end{array} \right.
\eeq 
Then we set
\beq\label{h-trunc}
\underline h^\nu  :=\overline\varphi(h^\nu),\quad h_e^\nu := h^\nu- \underline h^\nu.
\eeq
Note that the truncation $\underline h^\nu$ will be used in the proof of \eqref{uni-est} below.\\
Likewise, since
\[
|h_e^\nu|\le \max\Big\{(-h^\nu-L)_+, (h^\nu-L)_+ \Big\},
\]
and $-L<\tilde h^\nu (x-X_{\nu})<L$, we have
\[
|h_e^\nu|\le \max\Big\{(-h^\nu-L)_+, (h^\nu-L)_+ \Big\}\le |h^\nu- \tilde h^\nu (x-X_{\nu})  |.
\]
Then it follows from \eqref{ineq-m} that for all $\nu<\nu_*$,
\[
\int_{-\infty}^{\infty} |h_e^\nu|^2dx \le C\big(\mathcal{E}_0 +1\big).
\]
Therefore, $\{h_e^\nu\}_{\nu>0}$ is bounded in $L^{\infty}(0,T;L^2(\bbr))$. 
Moreover, since $\{\underline h^\nu\}_{\nu>0}$ is bounded in $L^{\infty}((0,T)\times\bbr)\subset L^\infty(0,T;L^2_{loc}(\bbr))$, we have
\beq\label{cpt-h}
\{h^\nu\}_{\nu>0} \mbox{ is bounded in } L^\infty(0,T;L^2_{loc}(\bbr)).
\eeq
Therefore,
there exists $u_{\infty} \in  L^\infty(0,T;L^2_{loc}(\bbr))$ such that
\beq\label{hwc}
h^\nu \rightharpoonup u_{\infty} \quad \mbox{in} ~L^\infty(0,T;L^2_{loc}(\bbr)) .
\eeq

\vskip0.2cm
\noindent{\it Step 2:}
We now prove that $u^\nu \rightharpoonup u_{\infty}$ in $\mathcal{M}_{\mathrm{loc}}((0,T)\times\bbr)$. \\
 Since $u^\nu=h^\nu-\nu  \Big(p(v^\nu)^{\frac{\alpha}{\gamma}}\Big)_x$, it is enough to show that 
\beq\label{pfpr}
\nu  \Big(p(v^\nu)^{\frac{\alpha}{\gamma}}\Big)_x \rightharpoonup 0 \quad \mbox{in} ~\mathcal{M}((0,T)\times\bbr).
\eeq
In fact, since 
\[
 \Big(p(v^\nu)^{\frac{\alpha}{\gamma}}\Big)_x =\frac{\alpha}{\gamma} p(v^\nu)^{\frac{\alpha-\gamma}{\gamma}} p(v^\nu)_x = \frac{\alpha}{\gamma} (v^\nu)^\beta p(v^\nu)_x,\quad \mbox{(recall $\beta=\gamma-\alpha$)},
\]
it is enough to show that 
\beq\label{weak_pv}
\nu   (v^\nu)^\beta p(v^\nu)_x \rightharpoonup 0 \quad \mbox{in} ~\mathcal{M}((0,T)\times\bbr).
\eeq
For that, we separate $\nu   (v^\nu)^\beta p(v^\nu)_x$ into two parts:
\begin{align*}
\begin{aligned} 
&\nu   (v^\nu)^\beta p(v^\nu)_x = \underbrace{ \nu (v^\nu)^\beta \Big(p(v^\nu) -p(\tilde v^\nu (x-X_{\nu}(t))) \Big)_x }_{=:J_1}+\underbrace{ \nu (v^\nu)^\beta  p(\tilde v^\nu (x-X_{\nu}(t)))_x }_{=:J_2}  
\end{aligned}
\end{align*}
For any $\Psi\in C_c((0,T)\times\bbr)$, using \eqref{ineq-m} and \eqref{rel_Q} together with the condition $0<\beta\le1$, we find that for all $\nu<\nu_*$,
\begin{align*}
\begin{aligned} 
\int_{\bbr^+\times\bbr} J_1\Psi dx dt &\le \nu \sqrt {\int_{\bbr^+\times\bbr}  (v^\nu)^\beta \left| \big(p(v^\nu) -p(\tilde v^\nu (x-X_{\nu}(t))) \big)_x \right|^2 \Psi dx dt}  \sqrt {\int_{\bbr^+\times\bbr}  (v^\nu)^\beta \Psi dx dt}   \\
&\le C\sqrt\nu  \sqrt{\mathcal{E}_0+1} \sqrt {\int_{\bbr^+\times\bbr}  (v^\nu)^\beta \left(  {\mathbf 1}_{\{v^\nu\le 3v_-\}}+ {\mathbf 1}_{\{v^\nu\ge 3v_-\}} \right)  \Psi dx dt}  \\
&\le  C\sqrt\nu \sqrt{\mathcal{E}_0+1} \sqrt {1+ \int_{\bbr^+\times\bbr}  |v^\nu - \tilde v^\nu (x-X_{\nu}(t)) | {\mathbf 1}_{\{v^\nu\ge 3v_-\}}  \Psi dx dt}  \\
&\le  C\sqrt\nu  \sqrt{\mathcal{E}_0+1} \sqrt {1+ \int_{\bbr^+\times\bbr}  Q\left(v^\nu | \tilde v^\nu (x-X_{\nu}(t)) \right)  \Psi dx dt}  \\
&\le  C\sqrt\nu  \sqrt{\mathcal{E}_0+1} \sqrt {1+ \int_{\mbox{supp} (\Psi)}  \eta\big((v^{\nu},h^{\nu})(t,x)| (\tilde v^{\nu}, \tilde h^{\nu})(x-X_{\nu}(t))\big)   dx dt}  \\
&\le  C\sqrt\nu  (\mathcal{E}_0+1).
\end{aligned}
\end{align*}
Likewise, we find that for all $\nu<\nu_*$,
\begin{align*}
\begin{aligned} 
\int_{\bbr^+\times\bbr} J_2\Psi dx dt &\le \nu C \int_{\bbr^+\times\bbr}  p(\tilde v^\nu (x-X_{\nu}(t)))_x {\mathbf 1}_{\{v^\nu\le 3v_-\}} \Psi dx dt  \\
&\quad + \nu \int_{\bbr^+\times\bbr}  (v^\nu)^\beta  p(\tilde v^\nu (x-X_{\nu}(t)))_x {\mathbf 1}_{\{v^\nu\ge 3v_-\}}  \Psi dx dt  \\
&\le \nu C \int_{\bbr^+\times\bbr}  |(\tilde v^\nu)' (x) | \Psi dx dt  \\
&\quad + C \nu \int_{\bbr^+\times\bbr}  |(\tilde v^\nu)' (x) |   Q\left(v^\nu | \tilde v^\nu (x-X_{\nu}(t)) \right)  {\mathbf 1}_{\{v^\nu\ge 3v_-\}}  \Psi dx dt  \\
&\le \nu C \int_{\bbr}  |\tilde v' (x) | dx   + C \nu (\mathcal{E}_0+1).
\end{aligned}
\end{align*}
Therefore we have 
\[
\int_{\bbr^+\times\bbr}  \nu   (v^\nu)^\beta p(v^\nu)_x\Psi dx dt \to 0\quad\mbox{as }\nu\to0,
\]
which implies \eqref{weak_pv}, and thus,
\beq\label{uwc}
u^\nu \rightharpoonup u_{\infty}  \quad \mbox{in} ~\mathcal{M}_{\mathrm{loc}}((0,T)\times\bbr).
\eeq
Hence we complete the proof of \eqref{wconv}.\\
 
\subsubsection{\bf Convergence of $\{X_{\nu}\}_{\nu>0}$}

\begin{lemma}\label{lem-X}
There exists $X_\infty\in \mbox{BV}(0,T)$ such that 
\beq\label{X-con}
X_\nu \to X_\infty \quad \mbox{in } L^1(0,T),\quad \mbox{up to subsequence as }\nu\to0.
\eeq
\end{lemma} 
\begin{proof}
First, since $X_{\nu}'(t)=X'(t/\nu)$, it follows from \eqref{est-shift} that 
\[
|X_{\nu}'(t)|\le C(1+ f_\nu (t)),
\]
where $f_\nu(t):= f\big(\frac{t}{\nu}\big)$. Notice that \eqref{est-shift} and \eqref{ini_conv} imply that for any $\nu<\nu_*$,
\begin{align*}
\begin{aligned} 
\|f_\nu \|_{L^1(0,T)}&=\nu\|f \|_{L^1(0,T/\nu)} \le \nu C\int_{-\infty}^{\infty}\eta\big((v_0^1,h_0^1)| (\tilde v, \tilde h)\big) dx \\
&\le C\Big(\int_{-\infty}^{\infty} \eta\big((v^0,u^0)| (\bar v, \bar u)\big) dx +1 \Big).
\end{aligned}
\end{align*}
Thus, $f_\nu$ is uniformly bounded in $L^1(0,T)$. Therefore, $X_{\nu}'$ is uniformly bounded in $L^1(0,T)$.\\
Moreover, since $X_\nu (0)=0$ and thus,
\[
|X_{\nu}(t)|\le Ct+ C\int_0^t f_\nu (s) ds,
\]
$X_{\nu}$ is also uniformly bounded in $L^1(0,T)$.\\
Therefore, by the compactness of BV (see for example \cite[Theorem 3.23]{AFP}), we have the desired convergence.
\end{proof}

\subsubsection{\bf Proof of \eqref{uni-est}}
Consider a mollifier
\beq\label{def-time}
\phi_\eps(t):=\frac{1}{\eps}\phi \big(\frac{t}{\eps}\big)\quad\mbox{for any }\eps>0,
\eeq
where $\phi:\bbr\to\bbr$ is a nonnegative smooth function such that  $\int_{\bbr}\phi=1$ and $\mbox{supp } \phi = [-1,1]$.\\
For the truncations $\underline v^{\nu}, \underline h^{\nu}$ defined by  \eqref{v-trunc}, \eqref{h-trunc} with $L, K$ fixed, we let
\[
L_{\nu}:= \int_0^{T} \phi_\eps(s) \int_\bbr  \eta \big( (\underline v^{\nu}, \underline h^{\nu})(s,x) | (\tilde v^{\nu},\tilde h^{\nu})(x-X_{\nu}(s))  \big) dxds,
\]
Using the definition of the truncations together with \eqref{ineq-m} and $\int_0^T  \phi_\eps = 1$, we find that for all $\nu<\nu_*$,
\beq\label{L-est}
L_{\nu} \le \int_0^{T}  \phi_\eps(s) \int_\bbr  \eta \big( ( v^{\nu}, h^{\nu})(s,x) | (\tilde v^{\nu},\tilde h^{\nu})(x-X_{\nu}(s))  \big) dxds \le C\mathcal{E}_0 +\delta.
\eeq

Then we have the following.

\begin{lemma}\label{lem-LR}
For the fixed constants $L, K$, let
\[
R_{\nu}:= \int_0^{T}  \phi_\eps(s) \int_\bbr  \eta \big( (\underline v^{\nu}, \underline h^{\nu})(s,x) | (\bar v,\bar u) (x-X_{\infty}(s))  \big) dxds.
\]
Then
\[
\Big| L_{\nu} - R_{\nu} \Big| \to 0 \quad\mbox{up to a subsequence as }\nu\to0.
\]
\end{lemma} 
\begin{proof}
Since $\underline h^\nu, \tilde h^{\nu}, \bar u$ are bounded, we have
\[
\big| |\underline h^{\nu}-\tilde h^{\nu}(x-X_{\nu})|^2 - |\underline h^{\nu}-\bar u(x-X_{\infty})|^2 \big| \le C\big|\tilde h^{\nu}(x-X_{\nu})-\bar u(x-X_{\infty})\big|.
\]
We separate the right-hand side into two parts:
\[
|\tilde h^\nu(x-X_\nu)-\bar u(x-X_\infty)|\le \underbrace{|\tilde h^\nu(x-X_\nu)-\bar u(x-X_\nu)|}_{=:I_1}+\underbrace{|\bar u (x-X_\nu)-\bar u(x-X_\infty)|}_{=:I_2}.
\] 
Since $\tilde h^{\nu}=\tilde u^{\nu}+\nu p(\tilde v^{\nu})_x$, using $\|\tilde u^\nu-\bar u\|_{L^1(\bbr)}=\nu \|\tilde u-\bar u\|_{L^1(\bbr)}$, we have
\[
\|I_1\|_{L^1(\bbr)}=\|\tilde h^\nu-\bar u\|_{L^1(\bbr)}\le \|\tilde u^{\nu}-\bar u\|_{L^1(\bbr)}+\nu \|p(\tilde v^{\nu})_x\|_{L^1(\bbr)} \le C\nu.
\]
Moreover, since $\|I_2\|_{L^1(\bbr)}=|u_- -u_+||X_\nu -X_\infty|$, it follows from Lemma \ref{lem-X} that
\begin{align*}
\begin{aligned} 
&\int_0^{T}  \phi_\eps(s) \int_\bbr \frac{1}{2}\big| |\underline h^{\nu}-\tilde h^{\nu}(x-X_{\nu}(s))|^2 - |\underline h^{\nu}-\bar u(x-X_{\infty}(s))|^2 \big|dxds \\
&\qquad\le C\nu +C\int_0^{\infty}  \phi_\eps(s) |X_\nu(s) -X_\infty(s)| ds \to 0.
\end{aligned}
\end{align*}
Likewise, since $\underline v^{\nu}$ is bounded, using the definition of $Q(\cdot|\cdot)$, we have
\begin{align*}
\begin{aligned} 
&\big| Q(\underline v^{\nu} | \tilde v^{\nu}(x-X_{\nu})) - Q(\underline v^{\nu} | \bar v(x-X_{\infty})) \big| \\
&\quad \le \left|Q(\tilde v^{\nu}(x-X_{\nu}))-Q(\bar v(x-X_{\infty})) \right| +\left| Q'(\tilde v^{\nu}(x-X_{\nu})) \right| \big|\tilde v^{\nu}(x-X_{\nu})-\bar v(x-X_{\infty})\big|  \\
&\qquad +\left(|\underline v^{\nu}| + \left| \bar v(x-X_{\infty})  \right| \right) \left| Q'(\tilde v^{\nu}(x-X_{\nu}))-Q'(\bar v(x-X_{\infty}))  \right| \\
&\quad \le C\big|\tilde v^{\nu}(x-X_{\nu})-\bar v(x-X_{\infty})\big|.
\end{aligned}
\end{align*}
Therefore, following the same computations as above, we have the desired result.
\end{proof}

Recalling \eqref{psi_k} and \eqref{phi_k}, we now consider
\begin{align*}
\begin{aligned} 
&\iint_{(0,T)\times\bbr}  \phi_\eps(s)  \eta \big( ( v^{\nu}, h^{\nu})(s,x) | (\bar v,\bar u)(x-X_\infty(s))  \big) dxds\\
&\quad = \underbrace{\iint_{h^\nu\in [-L,L]}  \phi_\eps  \frac{\big| \underline h^{\nu} - \bar u (x-X_{\infty})  \big|^2}{2} dxds}_{=:J_1} 
+\underbrace{ \iint_{h^\nu\notin [-L,L]}   \phi_\eps  \frac{\big| h^{\nu} - \bar u (x-X_{\infty})  \big|^2}{2}   dxds}_{=:J_2}\\
&\qquad + \underbrace{\iint_{v^\nu\in [K^{-1},K]}  \phi_\eps Q \big( \underline v^{\nu} | \bar v (x-X_{\infty})  \big) dxds}_{=:J_3} +\underbrace{ \iint_{v^\nu\notin [K^{-1},K]}   \phi_\eps Q \big( v^{\nu} |\bar  v (x-X_{\infty})  \big)  dxds}_{=:J_4}.
\end{aligned}
\end{align*}
Note that (using \eqref{L-est})
\[
J_1+J_3 \le R_\nu = (R_\nu- L_\nu) +L_\nu  \le (R_\nu- L_\nu) +C\mathcal{E}_0 +\delta.
\]
For $J_2$, we use the fact that since $\bar u, \tilde h^\nu \in \big(\min\{u_-,u_+\},\max\{u_-,u_+\} \big)\subset \big(-L/2, L/2\big)$, we find
\[
\big| h^{\nu} - \bar u (x-X_{\infty})  \big| \le 3 \big| h^{\nu} - \tilde h^\nu (x-X_\nu)  \big|\quad\mbox{for all }h^\nu\notin [-L,L].
\]
Then using \eqref{ineq-m},
\[
J_2\le \frac{9}{2}\iint_{h^\nu\notin [-L,L]}  \phi_\eps \big| h^{\nu} -  \tilde h^\nu (x-X_\nu)  \big|^2  dxds \le C(\mathcal{E}_0 +\delta).
\]
Likewise for $J_4$, since $\bar v, \tilde v^\nu \in \big(\min\{v_-,v_+\},\max\{v_-,v_+\} \big)\subset \big(M^{-1}, M\big)$, using \eqref{Q2} in Lemma \ref{lem_Q1} with the choice \eqref{Kdef}, we have
\[
J_4\le C \iint_{v^\nu\notin [K^{-1},K]}  \phi_\eps Q \big( v^{\nu} |\tilde  v^\nu (x-X_\nu)  \big)  dxds \le C(\mathcal{E}_0 +\delta).
\]
Therefore, we have
\begin{align}
\begin{aligned} \label{last-1}
&\iint_{(0,T)\times\bbr} \phi_\eps(s)  \frac{|h^\nu(t,x)-\bar u(x-X_{\infty}(s))|^2}{2}  ds dx \\
&\quad+ \iint_{(0,T)\times\bbr} \phi_\eps(s)  Q(v^\nu(t,x) | \bar v (x-X_{\infty}(s))) ds dx \le |R_\nu- L_\nu| +C(\mathcal{E}_0 +\delta).
\end{aligned}
\end{align}

Now, it remains to show that the left-hand side of \eqref{last-1} is lower semi-continuous with respect to the weak convergences \eqref{vwc} and \eqref{hwc}. \\
First of all, using the weak lower semi-continuity of the $L^2$-norm (for example see \cite{evans-w}) together with \eqref{hwc}, we have
\begin{align}
\begin{aligned} \label{conv-hu}
&\iint_{(0,T)\times\bbr} \phi_\eps(s)  \frac{|u_\infty(t,x)-\bar u(x-X_{\infty}(s))|^2}{2}  ds dx \\
&\quad \le \liminf_{\nu\to 0}\iint_{(0,T)\times\bbr} \phi_\eps(s)  \frac{|h^\nu(t,x)-\bar u(x-X_{\infty}(s))|^2}{2}  ds dx .
\end{aligned}
\end{align}

However, since $v_\infty$ is a measure in space as $v_\infty \in L^\infty(0,T; L^\infty(\bbr)+\mathcal{M}(\bbr))$, we may use the generalized relative functional \eqref{dQ} to handle the measure $v_\infty$. \\
In the following lemma, we show the weakly lower semi-continuity of the functional 
\[
dQ(v^\nu | \bar v(x-X_\infty))
\] in the left-hand side of \eqref{last-1}. In fact, Lemma \ref{lem:mlsc} deals with more general case where $\{v^\nu\}_{\nu>0}$ is the sequence of measures.
Without loss of generality, we only handle the case of $v_->v_+$, and set 
\beq\label{def-OM}
\Omega_M :=\{ (t,x)\in(0,T)\times\bbr~|~ x< X_\infty(t) \} .
\eeq
Since $X_\infty\in BV(0,T)$, we have that
\beq\label{bdry}
\mbox{Lebesgue measure on $\bbr^2$ of $\partial\Omega_M$ (:= the boundary of $\Omega_M$) is zero},
\eeq
and the complement of $\overline{\Omega_M}$ (:= the closure of $\Omega_M$) in $(0,T)\times\bbr$ is as follows:
\beq\label{com-Om}
(\overline{\Omega_M})^c=\{ (t,x)\in(0,T)\times\bbr~|~ x> X_\infty(t) \}. 
\eeq
Note that 
\beq\label{def-Omv}
\bar v(x-X_\infty(t))=\left\{ \begin{array}{ll}
        v_- \qquad \mathrm{for  \ \ } (t,x)\in \Omega_M,\\
        v_+ \qquad \mathrm{for  \ \ } (t,x)\in (\overline{\Omega_M})^c .\end{array} \right.
\eeq
\begin{lemma}\label{lem:mlsc}
Assume $v_+<v_-$. Consider the set \eqref{def-OM} and the properties \eqref{bdry},  \eqref{com-Om}.\\
Let $\Phi:\bbr^+\times\bbr\to\bbr$ be any compactly supported nonnegative function.\\
Let $\{v^k\}_{k=1}^{\infty}$ be a sequence of positive measures in $L^\infty((0,T)\times\bbr)+\mathcal{M}((0,T)\times\bbr)$ such that for some constant $C_0>0$ (independent of $k$),
\[
\int_{(0,T)\times\bbr} \Phi(t,x)~d Q\left(v^k | \bar v(x-X_\infty(t))\right) (t,x) \le C_0 ,
\]
where 
\[
dQ\left(v^k |\bar v(x-X_\infty(t))\right)(t,x) := Q\left(v^k_a|\bar v(x-X_\infty(t))\right) dt dx +  |Q'(\overline V(t,x))| dv^k_s (t,x) ,
\]
where $d v^k (t,x) := v^k_a(t,x) dtdx + dv^k_s (t,x)$ \mbox{(by Radon-Nikodym's theorem)}, and
\beq\label{def-vom}
\overline{V}(t,x):=\left\{ \begin{array}{ll}
        v_- \qquad \mathrm{for  \ \ } (t,x)\in \overline{\Omega_M},\\
        v_+ \qquad \mathrm{for  \ \ } (t,x)\in (\overline{\Omega_M})^c .\end{array} \right.
\eeq
Then, there exists a limit $v_\infty\in L^\infty((0,T)\times\bbr)+\mathcal{M}((0,T)\times\bbr)$ such that $v^k \rightharpoonup v_\infty$ in $\mathcal{M}_{\mathrm{loc}}(\bbr^+\times\bbr)$,
and 
\[
\int_{(0,T)\times\bbr}  \Phi(t,x)~d Q\left(v_\infty | \bar v(x-X_\infty(t))\right) (t,x)  \le  C_0.
\]
\end{lemma}
\begin{proof}
Since $v^k$ are positive measures in $L^\infty((0,T)\times\bbr)+\mathcal{M}((0,T)\times\bbr)$, Radon-Nikodym's theorem implies that there exist positive measures $v^k_a \in L^\infty(\bbr) +L^1(\bbr)$ and $dv^k_s$ (singular part of $v^k$) such that 
\[
d v^k (t,x) = v^k_a(t,x) dtdx + dv^k_s (t,x) .
\]
To truncate $ v^k_a$ by some big constant, we first use the fact that for any $\eps>0$, there exists $\xi>0$ with $\xi>\max(2 v_-, 2 v_+^{-1})$ such that for all $v>\xi$,
\beq \label{large-con}
(|Q'( {\bf{\bar v}})| +\eps) v \ge Q(v|{\bf{\bar v}}) \ge (|Q'({\bf{\bar v}})| -\eps) v ,
\eeq
where ${\bf{\bar v}}:=\bar v(x-X_\infty(t))$. Indeed, this is straightforwardly verified by the definition of the relative functional $Q(\cdot|\cdot)$, and $Q(v)\to 0$ as $v\to\infty$. \\
For such a constant $\xi$, we define
\[
v^k_\xi := \inf (v^k_a,\xi) ,
\]
and 
\[
Q_\xi (v):=\left\{ \begin{array}{ll}
       Q(v),\quad\mbox{if }  v\ge \xi^{-1},\\
        Q'(\xi^{-1}) (v-\xi^{-1}) + Q(\xi^{-1}),\quad\mbox{if } v\le\xi^{-1}.\end{array} \right.
\]
Note that $v\mapsto Q_\xi(v)$ is nonnegative and convex $C^1$-function on $[0,\infty)$, and $Q_\xi'({\bf{\bar v}})=Q'({\bf{\bar v}})$ (by $\xi^{-1}<v_+/2<{\bf{\bar v}}$). 
Then, we consider its relative functional: for any $v_1,v_2\ge 0$,
\[
Q_\xi(v_1|v_2) := Q_\xi(v_1) -Q_\xi(v_2)-Q_\xi'(v_2)(v_1-v_2) .
\]
Then, using \eqref{large-con}, we have
\beq\label{Q-cut}
d Q_\xi (v^k | {\bf{\bar v}}) \ge Q_\xi (v^k_\xi | {\bf{\bar v}}) dtdx + (|Q_\xi'(\overline V)| -\eps) (dv^k - v^k_\xi dtdx) - 2\eps d v^k ,
\eeq
which means that $d Q_\xi (v^k | {\bf{\bar v}}) -\big[ Q_\xi (v^k_\xi | {\bf{\bar v}}) dtdx + (|Q_\xi'(\overline V)| -\eps) (dv^k - v^k_\xi dtdx) - 2\eps d v^k\big]$ is nonnegative measure.
Indeed, this is verified as follows: If $v^k_a\le\xi$, then $v^k_\xi=v^k_a$, and so
\begin{align*}
\begin{aligned}
\mbox{LHS} &:=Q_\xi (v^k_a | {\bf{\bar v}}) dtdx+ |Q_\xi'(\overline V)| dv^k_s = Q(v^k_\xi| {\bf{\bar v}}) dtdx+ |Q_\xi'(\overline V)| dv^k_s \ge \mbox{RHS} ,
\end{aligned}
\end{align*}
where the last inequality follows from the facts that (by Radon-Nikodym's theorem) the measure $v^k - v^k_\xi$ is positive and
$v^k - v^k_\xi= (v^k_a - v^k_\xi) + v^k_s = v^k_s$.\\
If $v^k_a>\xi$, then $v^k_\xi=\xi$, and using \eqref{large-con}, $\mbox{LHS} \ge (|Q_\xi'({\bf{\bar v}})| -\eps) v^k_a dtdx + |Q_\xi'(\overline V)| dv^k_s $.
Since ${\bf{\bar v}}=\overline V$ $dtdx$-a.e. (by \eqref{def-Omv} and \eqref{def-vom}), we have
\begin{align*}
\begin{aligned}
\mbox{LHS} &\ge (|Q_\xi'(\overline V)| -\eps) dv^k = (|Q_\xi'(\overline V)| +\eps) \xi dtdx + (|Q_\xi'(\overline V)| -\eps) (d v^k- \xi dtdx) - 2\eps\xi dtdx\\
&= (|Q_\xi'({\bf{\bar v}})| +\eps) \xi dtdx + (|Q_\xi'(\overline V)| -\eps) (d v^k- v^k_\xi dtdx) - 2\eps\xi dtdx \\
&\ge  Q_\xi (\xi | {\bf{\bar v}}) dtdx + (|Q_\xi'(\overline V)| -\eps) (d v^k- v^k_\xi dtdx) - 2\eps\xi dtdx .
\end{aligned}
\end{align*}
Thus, using $\xi=v^k_\xi\le v^k_a \le v^k$, we have \eqref{Q-cut}.\\

Therefore, we use \eqref{Q-cut} to have
\begin{align*}
\begin{aligned} 
C_0 &\ge \limsup_{k\to\infty}\int_{(0,T)\times\bbr} \Phi(t,x)~d Q\left(v^k | \bar v(x-X_\infty(t))\right) (t,x) \\
&\ge \limsup_{k\to\infty}\int_{(0,T)\times\bbr} \Phi(t,x) \Big[ Q_\xi (v^k_\xi | \bar v(x-X_\infty(t)))dtdx \\
&\quad\quad \qquad\quad\qquad\qquad\qquad + (|Q'_\xi(\overline V)| -\eps) d(v^k - v^k_\xi) -2\eps d v^k \Big] .
\end{aligned}
\end{align*}
We set $\Omega_m:=(\overline{\Omega_M})^c$, and define  
\[
\Omega_m^{\delta} := \{ (t,x)\in \Omega_m~|~ d((t,x)|\Omega_m^c)>\delta \},\quad \forall \delta>0.
\]
Then we define a smooth function $\psi_1^{\delta}$ such that
\beq\label{defpsi1}
\psi_1^\delta (t,x) :=\left\{ \begin{array}{ll}
       1,\quad\mbox{on } (\Omega_m^\delta)^c ,\\
        0,\quad\mbox{on } \Omega_m^{2\delta} .\end{array} \right.
\eeq
Then, using this together with the facts that \eqref{def-Omv}, \eqref{def-vom} and $|Q'_\xi(v_-)|\le |Q_\xi'(v_+)|$ by $v_+<v_-$,
we have
\begin{align*}
\begin{aligned} 
C_0 &\ge \limsup_{k\to\infty} \Big[ \int_{\Omega_M} \Phi  Q_\xi (v^k_\xi | v_-)dtdx +\int_{\Omega_m} \Phi  Q_\xi (v^k_\xi | v_+)dtdx + (|Q_\xi'(v_-)| -\eps) \int \Phi \psi_1^\delta d(v^k - v^k_\xi) \\
&\quad  + (|Q_\xi'(v_+)| -\eps) \int \Phi (1-\psi_1^\delta) d(v^k - v^k_\xi) -2\eps \int \Phi d v^k  \Big].
\end{aligned}
\end{align*}
Note that since $|v^k_\xi|\le \xi$ for all $k$, there exists $v_*$ such that
\[
v^k_\xi  \rightharpoonup v_* \quad\mbox{in }~ L^\infty .
\]
Moreover, since the function $v\mapsto Q_\xi(v|c)$ with any constant $c$ is convex, the weak lower semi-continuity of convex functions (for example, see \cite{evans-w}) implies
\begin{align*}
\begin{aligned} 
& \liminf_{k\to\infty} \Big[ \int_{\Omega_M} \Phi  Q_\xi (v^k_\xi | v_-)dtdx +\int_{\Omega_m} \Phi  Q_\xi (v^k_\xi | v_+)dtdx  \Big] \\
&\quad \ge   \int_{\Omega_M} \Phi  Q_\xi (v_* | v_-)dtdx +\int_{\Omega_m} \Phi  Q_\xi (v_*  | v_+)dtdx .
\end{aligned}
\end{align*}
Also, since $v^k \rightharpoonup v_\infty$ in $\mathcal{M}_{\mathrm{loc}}(\bbr^+\times\bbr)$, and thus 
\beq\label{twoconv}
v^k -v^k_\xi \rightharpoonup v_\infty - v_* \quad\mbox{ in }~\mathcal{M}_{\mathrm{loc}}(\bbr^+\times\bbr)\quad\mbox{(by the uniqueness of the decomposition)}, 
\eeq
we have
\begin{align}
\begin{aligned} \label{cest1}
C_0 & \ge  \int_{\Omega_M} \Phi  Q_\xi (v_*  | v_-)dtdx +\int_{\Omega_m} \Phi  Q_\xi (v_*  | v_+)dtdx + (|Q_\xi'(v_-)| -\eps) \int \Phi \psi_1^\delta d(v_\infty - v_* ) \\
&\quad  + (|Q'(v_+)| -\eps) \int \Phi (1-\psi_1^\delta) d(v_\infty - v_* ) -2\eps \int \Phi d v_\infty =: \mathcal{R} .
\end{aligned}
\end{align}
By Radon-Nikodym's theorem, there exist positive measures $v_a \in L^\infty(\bbr) +L^1(\bbr)$ and $dv_s$ (singular part of $v_\infty$) such that
\beq\label{RNv}
d v_\infty (t,x) = v_a(t,x) dtdx + dv_s (t,x) .
\eeq
Note that since the measure $v^k -v^k_\xi $ is positive, it follows from \eqref{twoconv} and \eqref{RNv} that $v_\infty -v_* $, $v_a - v_* $ and $dv_s$ are all nonnegative.\\
Since $dv_\infty - v_*  dtdx= (v_a-v_* ) dtdx+ dv_s$ (by the uniqueness of the decomposition), we rewrite $\mathcal{R} $ in \eqref{cest1} as
\[
\mathcal{R} = \mathcal{R}_1 +\mathcal{R}_2+\mathcal{R}_3,
\]
where
\begin{align*}
\begin{aligned} 
 \mathcal{R}_1  & :=  |Q_\xi'(v_-)|  \int \Phi \psi_1^\delta dv_s +  |Q_\xi'(v_+)| \int \Phi (1-\psi_1^\delta) dv_s ,\\
 \mathcal{R}_2 &  := \int_{\Omega_M} \Phi  Q_\xi (v_*  | v_-) dtdx+ |Q'_\xi(v_-)| \int \Phi \psi_1^\delta (v_a - v_*  ) dtdx  \\
&\quad    +\int_{\Omega_m} \Phi   Q_\xi (v_*  | v_+) dtdx+ |Q'_\xi(v_+)|  \int \Phi (1-\psi_1^\delta)  (v_a - v_*  )  dtdx ,\\
\mathcal{R}_3 &  := - 3\eps \int \Phi  d v_\infty +\eps \int \Phi v_*  dtdx.
\end{aligned}
\end{align*}
Using $\overline{\Omega_M}\subset (\Omega_m^\delta)^c$ and \eqref{defpsi1}, we have
\begin{align*}
\begin{aligned} 
 \mathcal{R}_1  & \ge |Q'_\xi(v_-)|  \int_{\overline{\Omega_M}} \Phi  dv_s +  |Q'_\xi(v_+)| \int_{\Omega_m^{2\delta}} \Phi dv_s .
\end{aligned}
\end{align*}
Since $\Phi dv_s$ is a positive measure, and
\beq\label{conv-om}
\Omega_m^{2\delta} \nearrow \cup_{\delta>0} \Omega_m^{2\delta} = (\overline{\Omega_M})^c,
\eeq
we have
\[
\lim_{\delta\to 0} \int_{\Omega_m^{2\delta}} \Phi dv_s = \int_{ (\overline{\Omega_M})^c} \Phi dv_s .
\]
Thus,
\[
 \mathcal{R}_1  \ge \int   |Q'_\xi(\overline V)| \Phi dv_s .
\]

For $ \mathcal{R}_2$, we use \eqref{defpsi1} to have
\begin{align*}
\begin{aligned} 
 \mathcal{R}_2 &  \ge \int_{\Omega_M} \Phi \Big[ Q_\xi (v_*  | v_-) + |Q'_\xi(v_-)|  (v_a - v_*  ) \Big] dtdx  \\
&\quad    +\int_{\Omega_m^{2\delta}} \Phi \Big[   Q_\xi (v_*  | v_+) + |Q'_\xi(v_+)|  (v_a - v_*  ) \Big] dtdx .\\
\end{aligned}
\end{align*}
Then, we have
\begin{align*}
\begin{aligned} 
 \mathcal{R}_2 &  \ge \int_{\Omega_M} \Phi Q_\xi (v_a | v_-)  dtdx    +\int_{\Omega_m^{2\delta}} \Phi   Q_\xi (v_a | v_+)  dtdx ,
\end{aligned}
\end{align*}
where we used the equality that for any $w_1, w_2\ge 0$ and any $c>0$,
\[
Q_\xi (w_1+w_2 | c) \le Q_\xi (w_1 | c) + |Q'_\xi (c)|w_2.
\]
Indeed, it follows from $Q_\xi'\le 0$ and the definition of $Q_\xi(\cdot|\cdot)$ that
\[
Q_\xi (w_1+w_2 | c) -Q_\xi (w_1 | c) - |Q'_\xi (c)|w_2 = Q_\xi (w_1+w_2)-Q_\xi(w_1) \le 0 .
\]
Since \eqref{conv-om} imiplies
\[
\lim_{\delta\to 0} \int_{\Omega_m^{2\delta}} \Phi   Q_\xi (v_a | v_+)  dtdx= \int_{ (\overline{\Omega_M})^c} \Phi   Q_\xi (v_a | v_+)  dtdx ,
\]
we use \eqref{bdry} to have
\[
 \mathcal{R}_2 \ge \int \Phi   Q_\xi (v_a | \bar v(x-X_\infty(t)))  dtdx .
\]
Therefore, we have
\[
\mathcal{R} \ge \int \Phi   Q_\xi (v_a | \bar v(x-X_\infty(t)))  dtdx +  \int  \Phi  |Q'_\xi(\overline V)| dv_s - 3\eps \int \Phi  d v_\infty ,
\]
that is,
\[
\int \Phi   Q_\xi (v_a | \bar v(x-X_\infty(t)))  dtdx +  \int  \Phi  |Q'_\xi(\overline V)| dv_s \le \mathcal{R}  + 3\eps \int \Phi  d v_\infty.
\]
Therefore, taking $\xi\to\infty$ and using Fatou's lemma, we have
\[
\int \Phi   Q (v_a | \bar v(x-X_\infty(t)))  dtdx +  \int  \Phi  |Q'(\overline V)| dv_s \le \mathcal{R}  + 3\eps \int \Phi  d v_\infty.
\]
Then taking $\eps\to 0$, we have
\[
\int \Phi   Q (v_a |\bar v(x-X_\infty(t)))  dtdx +  \int  \Phi  |Q'(\overline V)| dv_s \le \mathcal{R} 
\]
This completes the proof.
\end{proof}

To apply Lemma \ref{lem:mlsc} to \eqref{last-1}, we define a smooth function $\psi_0^R$ such that for any $R>0$,
\[
\mathbf{1}_{|x|\le R} \le \psi_0^R (x) \le \mathbf{1}_{|x|\le 2R} .
\]
Then, it follows from \eqref{last-1} that
\begin{align}
\begin{aligned} \label{last-2}
&\iint_{(0,T)\times\bbr} \phi_\eps(s)   \frac{|h^\nu(t,x)-\bar u(x-X_{\infty}(s))|^2}{2}  ds dx \\
&\quad+ \iint_{(0,T)\times\bbr} \phi_\eps(s)  \psi_0^R (x)   Q(v^\nu(t,x) | \bar v (x-X_{\infty}(s))) ds dx \le |R_\nu- L_\nu| +C(\mathcal{E}_0 +\delta).
\end{aligned}
\end{align}
Thus, using Lemma \ref{lem:mlsc} together with the weak convergence \eqref{vwc}, we have
\begin{align*}
\begin{aligned} 
& \iint_{(0,T)\times\bbr} \phi_\eps(s)  \psi_0^R (x)  ~ d Q(v_\infty (t,x) | \bar v (x-X_{\infty}(s))) ds dx \\
&\quad \le \liminf_{\nu\to 0}\iint_{(0,T)\times\bbr} \phi_\eps(s)  \psi_0^R (x)   Q(v^\nu(t,x) | \bar v (x-X_{\infty}(s))) ds dx .
\end{aligned}
\end{align*}
Here, the measure $v_\infty$ has the decomposition \eqref{RNv}, and
\[
dQ\left(v_\infty |\bar v(x-X_\infty(t))\right)(t,x) = Q\left(v_a|\bar v(x-X_\infty(t))\right) dt dx +  |Q'(\overline V(t,x))| dv_s (t,x) ,
\]
where $\overline{V}(t,x)$ is defined by \eqref{def-vom} with \eqref{def-OM}.\\
Then, using $\bbr^+\times (-R,R)\nearrow \bbr^+\times\bbr$ as $R\to\infty$, we have
\begin{align*}
\begin{aligned} 
& \iint_{(0,T)\times\bbr} \phi_\eps(s)  ~ d Q(v_\infty (t,x) | \bar v (x-X_{\infty}(s))) ds dx \\
&\quad \le \liminf_{\nu\to 0}\iint_{(0,T)\times\bbr} \phi_\eps(s)  \psi_0^R (x)   Q(v^\nu(t,x) | \bar v (x-X_{\infty}(s))) ds dx .
\end{aligned}
\end{align*}
Therefore, this together with \eqref{last-2}, \eqref{conv-hu} and Lemma \ref{lem-LR} yields
\begin{align*}
\begin{aligned} 
&\iint_{(0,T)\times\bbr} \phi_\eps(s)  \frac{|u_\infty(t,x)-\bar u(x-X_{\infty}(s))|^2}{2}  ds dx \\
& \quad +\iint_{(0,T)\times\bbr} \phi_\eps(s)  ~ d Q(v_\infty (t,x) | \bar v (x-X_{\infty}(s))) ds dx  \le C(\mathcal{E}_0 +\delta) .
\end{aligned}
\end{align*}
Taking $\eps\to0$ (recall \eqref{def-time}), we obtain that
\[
d Q(v_\infty | \bar v (\cdot-X_{\infty}(\cdot))) \in L^\infty(0,T;\mathcal{M}(\bbr)),
\]
and, for a.e. $t\in(0,T)$,
\[
\int_{\bbr } \frac{|u_\infty(t,x)-\bar u(x-X_{\infty}(t))|^2}{2}  dx + \left(\int_{x\in \bbr } d Q(v_\infty | \bar v (x-X_{\infty}(\cdot))) \right)(t) \le C(\mathcal{E}_0 +\delta).
\]
Since $\delta>0$ is arbitrary, we obtain 
\[
\int_{\bbr } \frac{|u_\infty(t,x)-\bar u(x-X_{\infty}(t))|^2}{2}  dx + \left(\int_{x\in \bbr } d Q(v_\infty | \bar v (x-X_{\infty}(\cdot))) \right)(t) \le C \mathcal{E}_0 ,
\]
which gives \eqref{uni-est}.

\subsubsection{\bf Weak continuity of the limit $v_\infty$}
In order to prove \eqref{X-control}, 
we may first prove that $v_\infty$ is weakly continuous in time, and
\beq\label{v-conti}
\lim_{t\to 0+}\int_\bbr\varphi(x) v_\infty(t, dx) = \int_\bbr \varphi(x) v^0(x) dx,\qquad \forall \varphi\in C_0(\bbr) .
\eeq
We first claim that
\beq\label{prove-u} 
\{u^{\nu}\}_{\nu>0} \mbox{ is bounded in } L^2(0,T;L^1_{loc}(\bbr)).
\eeq
For that, recall from \eqref{effective} that $u^{\nu}=h^{\nu}-\nu \big(p(v^\nu)^{\frac{\alpha}{\gamma}}\big)_x$. First, we have \eqref{cpt-h}. To get a uniform boundedness of $\nu \big(p(v^\nu)^{\frac{\alpha}{\gamma}}\big)_x$, we use the same estimates as in Step 2 for the proof of \eqref{pfpr}. Indeed, since
\begin{align*}
\begin{aligned} 
\nu \Big|\big(p(v^\nu)^{\frac{\alpha}{\gamma}}\big)_x\Big| &=\nu \frac{\alpha}{\gamma} (v^\nu)^\beta \big|p(v^\nu)_x \big|\\
&\le \nu (v^\nu)^\beta \big|\big(p(v^\nu) -p(\tilde v^\nu (x-X_{\nu}(t))) \big)_x\big|+ \nu (v^\nu)^\beta \big| p(\tilde v^\nu (x-X_{\nu}(t)))_x \big| \\
&\le C\bigg(\sqrt\nu (v^\nu)^\beta \big|\big(p(v^\nu) -p(\tilde v^\nu (x-X_{\nu}(t))) \big)_x\big|^2 + \big(1+Q\left(v^\nu | \tilde v^\nu (x-X_{\nu}(t)) \right) \big) \\
&\qquad + \big|( \tilde v^\nu)' \big| \big( 1+ Q\left(v^\nu | \tilde v^\nu (x-X_{\nu}(t)) \right) \big) \bigg),
\end{aligned}
\end{align*}
using \eqref{ineq-m}, we have
\[
\nu \big(p(v^\nu)^{\frac{\alpha}{\gamma}}\big)_x ~\mbox{ is uniformly bounded in } L^2(0,T;L^1_{loc}(\bbr)).
\]
Therefore, we have \eqref{prove-u}.\\
Then, \eqref{prove-u} together with the equation $v_t^\nu - u_x^\nu =0$ in \eqref{inveq} implies
\[
v_t^\nu ~\mbox{ is uniformly bounded in } L^2(0,T;W^{-1,1}_{loc}(\bbr)).
\]
Hence, by Aubin-Lions lemma, this and \eqref{cpt-v} together with \eqref{vwc} imply that (up to a subsequence)
\[
v^{\nu} \to v_\infty \mbox{  in } C([0,T];W^{-s,1}_{loc}(\bbr)),\quad s>0,
\]
which together with \eqref{v-inicon} completes the proof of \eqref{v-conti}.

\subsubsection{\bf Proof of \eqref{X-control}}
First of all, since $X_\infty \in BV ((0,T))$, there exists a positive constant $r=r(T)$ such that $\|X_\infty\|_{L^{\infty}((0,T))}= r$. Then we consider a nonnegative smooth function $\psi: \bbr \to\bbr$ such that $\psi(x)=\psi(-x)$, and $\psi'(x)\le 0$ for all $x\ge 0$, and $|\psi'(x)|\le 2/r$ for all $x\in\bbr$, and
\[
\psi(x)=\left\{ \begin{array}{ll}
       1,\quad\mbox{if }  |x|\le r,\\
       0,\quad\mbox{if } |x|\ge 2r.\end{array} \right.
\]
On the other hand, let $\theta:\bbr\to\bbr$ be a nonnegative smooth function such that $\theta (s)=\theta(-s)$, $\int_{\bbr}\theta=1$ and $\mbox{supp } \theta = [-1,1]$, and let 
\[
\theta_\delta(s):=\frac{1}{\delta}\theta \big(\frac{s-\delta}{\delta}\big)\quad\mbox{for any }\delta>0.
\]
Then for a given $t\in(0,T)$, and any $\delta<t/2$, we define a nonnegative smooth function
\[
\varphi_{t, \delta} (s) := \int_0^s \Big(\theta_\delta (\tau) -\theta_\delta (\tau- t)  \Big) d\tau.
\]
Since $v^{\nu}_t - u^{\nu}_x =0$ by $\eqref{inveq}_1$, it follows from \eqref{wconv} that the limits $v_{\infty}$ and $u_{\infty}$ satisfy
\beq\label{limit-con}
\int_{[0,T]\times\bbr} \big(  \varphi_{t, \delta}'(s) \psi(x) d v_{\infty}(s,x) - \varphi_{t, \delta} (s) \psi'(x) u_{\infty}(s,x) dsdx \big)  =0.
\eeq
Since $\varphi_{t, \delta}'(s) = \theta_\delta (s) -\theta_\delta (s-t)$, we decompose the left-hand side above into three parts as
\[
I_1^\delta+I_2^\delta+I_3^\delta=0,
\]
where
\begin{align*}
\begin{aligned} 
&I_1^\delta:=\int_{[0,T]\times\bbr}  \theta_\delta (s) \psi(x) d v_{\infty}(s,x)  ,\\
&I_2^\delta:= -\int_{[0,T]\times\bbr}  \theta_\delta (s-t)   \psi(x) d v_{\infty}(s,x) ,\\
&I_3^\delta:= -\int_0^T \int_\bbr\varphi_{t, \delta} (s) \psi'(x) u_{\infty}(s,x)dxds.
\end{aligned}
\end{align*}
Using \eqref{v-conti} and the fact that $\int_\bbr  \psi(x) v_\infty(s, dx) $ is continuous in $s$, we find that as $\delta\to0$ :
\[
I_1^\delta\to \int_\bbr  \psi(x) v^0 (x) dx,\quad I_2^\delta \to -\int_\bbr  \psi(x) v_{\infty}(t,dx) ,
\]
and 
\[
I_3^\delta \to -\int_0^T \int_\bbr \psi'(x) u_{\infty}(s,x) dxds.
\]
Therefore, it follows from \eqref{limit-con} that
\[
\underbrace{\int_\bbr  \psi(x) \big(v_{\infty}(t,dx) -v^0(x)dx \big) }_{=:J_1} + \underbrace{\int_0^T \int_\bbr \psi'(x) u_{\infty}(s,x) dxds}_{=:J_2} =0.
\]
To show \eqref{X-control} from the above equation, we will use the stability estimate \eqref{uni-est} and the Rankine-Hugoniot condition. \\
For that, we decompose $J_1$ into three parts:
\[
J_1=J_{11}+J_{12}+J_{13},
\]
where 
\begin{align*}
\begin{aligned} 
J_{11} &=\int_\bbr  \psi(x) \big(v_{\infty}(t,dx) -\bar v(x-X_\infty(t))dx \big) , \\
J_{12} &= \int_\bbr  \psi(x) \big(\bar v(x-X_\infty(t)) -\bar v(x) \big)  dx,\\
J_{13} &=\int_\bbr  \psi(x) \big(\bar v(x) - v^0(x) \big)  dx.
\end{aligned}
\end{align*}
Likewise, we decompose $J_2$ into two parts:
\[
J_2=J_{21}+J_{22},
\]
where 
\begin{align*}
\begin{aligned} 
J_{21} &=\int_0^t \int_\bbr  \psi'(x) \big(u_{\infty}(s,x) -\bar u(x-X_\infty(s)) \big)  dxds, \\
J_{22} &= \int_0^t \int_\bbr  \psi'(x) \bar u(x-X_\infty(s)) dx ds.
\end{aligned}
\end{align*}
Since $|X_\infty(t)|\le r$ for all $t\in(0,T)$, we have
\[
J_{12} +J_{22} = (v_--v_+)X_\infty(t) + t (u_--u_+).
\]
Then using the Rankine-Hugoniot condition $\eqref{end-con}_1$, i.e., $\sigma= -\frac{u_--u_+}{v_--v_+}$, we have
\[
J_{12} +J_{22} =\big(X_\infty(t)-\sigma t \big) (v_--v_+).
\]
To control $J_{11}$ by the initial perturbation $\mathcal{E}_0=\int_{-\infty}^{\infty} \eta\big((v^0,u^0)| (\bar v, \bar u)\big) dx$, 
we recall the (unique) decomposition of the measure $v_\infty$ by
\[
d v_\infty (t,dx) = v_a(t,x) dx + v_s (t,dx) .
\]
Using \eqref{rel_Q}, we have
\begin{align*}
\begin{aligned} 
|J_{11}| &\le  \int_{-2r}^{2r} \big|v_a (t,x) -\bar v(x-X_\infty(t))\big| {\mathbf 1}_{\{v\le 3v_-\}} dx +\int_{-2r}^{2r} \big|v_a (t,x) -\bar v(x-X_\infty(t))\big| {\mathbf 1}_{\{v\ge 3v_-\}} dx \\
&\quad +\int_\bbr \psi(x) v_s (t,dx) \\
&\le  \frac{1}{\sqrt{c_1}}\int_{-2r}^{2r}  \sqrt{Q\big(v_a (t,x)|\bar v(x-X_\infty(t))\big)} dx + \frac{1}{c_2}\int_\bbr  Q\big(v_a(t,x)|\bar v(x-X_\infty(t))\big) dx\\
&\quad +\frac{1}{|Q'(v_-)|}\int_\bbr \psi(x)|Q'(\overline V)| v_s (t,dx) ,
\end{aligned}
\end{align*}
where note that $|Q'(\overline V)|\ge |Q'(v_-)|>0$ by  \eqref{def-vom}.\\
Thus, we use the stability estimate \eqref{uni-est} to have
\[
|J_{11}|  \le C\sqrt{r} \sqrt{\mathcal{E}_0} + C\mathcal{E}_0.
\]
Using the same estimates as above, and \eqref{ini_conv}, we have
\[
|J_{13}| \le  C\sqrt{r} \sqrt{\mathcal{E}_0} + C\mathcal{E}_0.
\]
Likewise, 
\[
|J_{21}| \le \frac{2}{r} \int_0^t \int_{[-2r,-r]\cup[r,2r]} \big|u_\infty(s,x) -\bar u(x-X_\infty(s))\big| dxds \le  \frac{C}{\sqrt{r}}t\sqrt{\mathcal{E}_0}.
\]
Hence we have
\[
(v_--v_+) | X_\infty(t) - \sigma t | \le C\Big( \mathcal{E}_0 + (1+t)\sqrt{\mathcal{E}_0} \Big),
\]
which completes the proof.\\

\begin{appendix}
\setcounter{equation}{0}

\section{Proof of Proposition \ref{prop:main3}} \label{app-exp}
We rewrite the functionals $Y_g, \mathcal{I}_1, \mathcal{I}_2, \mathcal{G}_2, \mathcal{D}$ with respect to the following variables 
\[
w:=p(v)-p(\tilde v_\eps),\quad y:=\frac{p(\tilde v_\eps(\xi))-p(v_-)}{p(v_+)-p(v_-)}.
\]
Since $p(\tilde v_\eps(\xi))$ is increasing in $\xi$, we use the change of variable $\xi\in\bbr\mapsto y\in[0,1]$. \\
Notice that $a=1-\lambda y$ and $|a-1|\leq \delta_3$ by \eqref{weight-a}, and
\beq\label{change-d}
\frac{dy}{d\xi}= \frac{p(\tilde v_\eps)'}{p(v_+)-p(v_-)},\quad\mbox{where }|p(v_+)-p(v_-)|=\eps.
\eeq

As in \cite[Proposition 3.4]{Kang-V-NS17}, we use the same notations:
\[
W:=\frac{\lambda}{\eps} w,\quad\quad \alpha_\gamma := \frac{\gamma \sqrt{-p'(v_-)} p(v_-)} {\gamma+1}>0.
\]
First of all, note that $Y_g, \mathcal{I}_1, \mathcal{I}_2, \mathcal{G}_2$ are respectively the same functionals  as $Y_g, \mathcal{B}_1, \mathcal{B}_2, \mathcal{G}_2$ in \cite[Proposition 3.4]{Kang-V-NS17} except for the term $\frac{1}{2}\int_\bbr a'' |p(v)-p(\tilde v_\eps)|^2 d\xi$ in $\mathcal{B}_2$, which is negligible by $\mathcal{I}_2$ because of $|a''|\le C\eps |a'|$ (see \cite[(2.29) and (3.31)]{Kang-V-NS17}). \\
Thus, it follows from \cite[(3.39), (3.34), (3.33), (3.35)]{Kang-V-NS17} that
\begin{align}
\begin{aligned}\label{all-func}
 & -2\alpha_\gamma \frac{\lambda^2}{\eps^3} \frac{|Y_g|^2}{\eps \delta_3}\leq -\frac{\alpha_\gamma}{\delta_3\sigma^4}\left| \int_0^1W^2\,dy+2\int_0^1 W\,dy\right|^2+C\delta_3\int_0^1W^2\,dy,\\
&{2\alpha_\gamma} \frac{\lambda^2}{\eps^3}|\mathcal{I}_1|\leq \left(1 +C(\eps_0+\delta_3)\right)\int_0^1 W^2\,dy,\\
&{2\alpha_\gamma} \frac{\lambda^2}{\eps^3}|\mathcal{I}_2|\leq \left(\frac{\alpha_\gamma}{\sigma}\left(\frac{\lambda}{\eps}\right)+C(\eps_0+\delta_3)\right)\int_0^1 W^2\,dy,\\
&-2\alpha_\gamma \frac{\lambda^2}{\eps^3}\mathcal{G}_2\leq  \left(-\frac{\alpha_\gamma}{\sigma}\left(\frac{\lambda}{\eps}\right) +C\delta_3     \right)\int_0^1W^2\,dy +\frac{2}{3}\int_0^1W^3\,dy+C\eps_0\int_0^1|W|^3\,dy.
 \end{aligned}
\end{align}
Therefore, it remains to estimate the diffusion $\mathcal{D}$ as follows:\\
First, by the change of variable, we have
\[
\mathcal{D}=\int_0^1 (1-\lambda y) |\partial_y w|^2 v^\beta \Big(\frac{dy}{d\xi}\Big) dy.
\]
Since it follows from \eqref{small_shock1} that
\[
\tilde v_\eps^\beta p(\tilde v_\eps)'=\sigma_\eps (\tilde v_{\eps}-v_-) + \frac{p(\tilde v_\eps)-p(v_-)}{\sigma_\eps},
\]
using \eqref{change-d}, we find
\beq\label{inst-vy}
\eps\, \tilde v_\eps^\beta \frac{dy}{d\xi}=\frac{1}{\sigma_\eps}\Big(\sigma_\eps^2 (\tilde v_{\eps}-v_-) + p(\tilde v_\eps)-p(v_-)\Big).
\eeq
Since the right-hand side of \eqref{inst-vy} is the same as the one in the proof of \cite[Lemma 3.1]{Kang-V-NS17}, we have
\[
\frac{\tilde v_\eps^\beta}{y(1-y)}\frac{dy}{d\xi}=\frac{\eps}{\sigma_\eps(v_--v_+)}\left( \frac{\tilde v_{\eps}-v_-}{p(\tilde v_\eps)-p(v_-)}+\frac{\tilde v_{\eps}-v_+}{p(v_+)-p(\tilde v_\eps)} \right).
\]
Thus, it follows from the proof of \cite[Lemma 3.1]{Kang-V-NS17} that
$$
\left|\frac{\tilde v_\eps^\beta}{y(1-y)}\frac{dy}{d\xi} -\frac{\eps}{2\alpha_\gamma}\right|\leq C\eps^2.
$$
Then, using $|(v^\beta/\tilde v_\eps^\beta)-1| \le C\delta_3$, we have
\begin{eqnarray*}
\mathcal{D}&\geq&(1-\lambda)\int_0^1  |\partial_y w|^2 v^\beta \Big(\frac{dy}{d\xi}\Big) dy = (1-\lambda)\int_0^1  |\partial_y w|^2 \frac{v^\beta}{\tilde v_\eps^\beta}\tilde v_\eps^\beta \Big(\frac{dy}{d\xi}\Big) dy\\
&\geq&(1-\lambda) \left(\frac{\eps}{2\alpha_\gamma}-C\eps^2 -C\delta_3 \right)   \int_0^1y(1-y)  |\partial_y w|^2  \, dy\\
&\geq&\frac{\eps}{2\alpha_\gamma}(1-C(\delta_3+\eps_0))  \int_0^1y(1-y)  |\partial_y w|^2 dy.
\end{eqnarray*}
After the normalization, we obtain
\begin{equation}\label{newD}
-2\alpha_\gamma \frac{\lambda^2}{\eps^3}\mathcal{D}\leq -(1-C(\eps_0+\delta_3))\int_0^1y(1-y) |\partial_y W|^2dy.
\end{equation}

To finish the proof, we first observe that
for any $\delta<\delta_3$, 
 \begin{eqnarray*}
 &&\mathcal{R}_{\eps,\delta}(v)\leq -\frac{1}{\eps\delta_3}|Y_g(v)|^2 +(1+\delta_3)|\mathcal{I}_1(v)|\\
 &&\qquad\qquad\qquad +\left(1+\delta_3\left(\frac{\eps}{\lambda}\right)\right)|\mathcal{I}_2(v)|-\left(1-\delta_3\left(\frac{\eps}{\lambda}\right)\right)\mathcal{G}_2(v)-(1-\delta_3)\mathcal{D}(v).
 \end{eqnarray*}
Then, \eqref{all-func} and \eqref{newD} together with $\eps_0\le\delta_3$ imply
\begin{align*}
\begin{aligned}
2\alpha_\gamma\left(\frac{\lambda^2}{\eps^3}\right)\mathcal{R}_{\eps,\delta}(v)
& \leq -\frac{1}{C_\gamma\delta_3}\left(\int_0^1W^2\,dy+2\int_0^1 W\,dy\right)^2+(1+C_* \delta_3)\int_0^1 W^2\,dy\\
&\quad+\frac{2}{3}\int_0^1 W^3\,dy +C_*\delta_3 \int_0^1 |W|^3\,dy -(1-C_* \delta_3)\int_0^1 y(1-y)|\partial_y W|^2\,dy.
 \end{aligned}
\end{align*}

To finish the proof, we use the nonlinear Poincar\'e type inequality \cite[Proposition 3.3]{Kang-V-NS17} as follow:
\begin{proposition}\label{prop:W}{\cite[Proposition 3.3]{Kang-V-NS17}}
For a given $C_1>0$, there exists $\deltat>0$, such that for any $\delta<\deltat$ the following is true.\\ For any $W\in L^2(0,1)$ such that 
$\sqrt{y(1-y)}\partial_yW\in L^2(0,1)$, if $\int_0^1 |W(y)|^2\,dy\leq C_1$, then
\begin{align}
\begin{aligned}\label{Winst}
&-\frac{1}{\delta}\left(\int_0^1W^2\,dy+2\int_0^1 W\,dy\right)^2+(1+\delta)\int_0^1 W^2\,dy\\
&\qquad\qquad+\frac{2}{3}\int_0^1 W^3\,dy +\delta \int_0^1 |W|^3\,dy  -(1-\delta)\int_0^1 y(1-y)|\partial_y W|^2\,dy \leq0.
 \end{aligned}
\end{align}
\end{proposition}

First, using the same estimate as in \cite[(3.38)]{Kang-V-NS17}, we find the constant $C_1>0$ such that
\[
\int_0^1W^2\,dy\leq C_1.
\]
Then, let us fix the value of the $\delta_2$ of Proposition \ref{prop:W} corresponding to  the constant $C_1$.\\
We consider  $\bar{\delta}=\max(C_\gamma, C_*) \delta_3$, and  choose $\delta_3$ small enough, such that 
$\bar{\delta}$ is smaller than $\delta_2$. Then we have 
\begin{eqnarray*}
&& 2\alpha_\gamma\left(\frac{\lambda^2}{\eps^3}\right)\mathcal{R}_{\eps,\delta}(v) \leq - \frac{1}{\delta_2}\left(\int_0^1W^2\,dy+2\int_0^1 W\,dy\right)^2+(1+\delta_2)\int_0^1 W^2\,dy\\
&&\qquad\qquad\qquad \qquad +\frac{2}{3}\int_0^1 W^3\,dy +\delta_2\int_0^1 |W|^3\,dy  -(1- \delta_2)\int_0^1 y(1-y)|\partial_y W|^2\,dy. 
\end{eqnarray*}
Therefore, using Proposition  \ref{prop:W}, we have 

\[
2\alpha_\gamma\left(\frac{\lambda^2}{\eps^3}\right)\mathcal{R}_{\eps,\delta}(v) \le 0,
\]
which completes the proof.

\end{appendix}

\bibliography{Kang-Vasseur2015}
\end{document}